\documentclass[11pt, letterpaper]{amsart}

\renewcommand{\Re}{\operatorname{Re}}
\renewcommand{\Im}{\operatorname{Im}}

\newcommand{\defeq}{\stackrel{\rm{def}}{=}}

\usepackage{amssymb}
\usepackage{amsthm}
\usepackage{amsxtra}
\usepackage{mathrsfs}
\usepackage{graphicx}
\usepackage{color}
\usepackage{inputenc}

\newcommand{\R}{\mathbb R}
\newcommand{\C}{\mathbb C}

\newcommand{\Z}{\mathbb Z}

\newtheorem{theorem}{Theorem}[section]

\newtheorem{proposition}{Proposition}[section]
\newtheorem{lemma}[proposition]{Lemma}

\newtheorem{claim}[theorem]{Claim}
\theoremstyle{remark}
\newtheorem{remark}[proposition]{Remark}

\numberwithin{equation}{section}

\ifx\pdfoutput\undefined
  \DeclareGraphicsExtensions{.pstex, .eps}
\else
  \ifx\pdfoutput\relax
    \DeclareGraphicsExtensions{.pstex, .eps}
  \else
    \ifnum\pdfoutput>0
      \DeclareGraphicsExtensions{.pdf}
    \else
      \DeclareGraphicsExtensions{.pstex, .eps}
    \fi
  \fi
\fi

\begin{document}
	\title[]
	{Threshold solutions for the Hartree equation}
	
\author[Anudeep K. Arora]{Anudeep Kumar Arora}
\address{Department of Mathematics, Statistics, and
	Computer Science, University of Illinois at Chicago, Chicago, IL, USA}
\curraddr{}
\email{anudeep@uic.edu}
\thanks{} 

\author[Svetlana Roudenko]{Svetlana Roudenko}
\address{Department of Mathematics \& Statistics\\
	Florida International University,  Miami, FL, USA}
\curraddr{}
\email{sroudenko@fiu.edu}
\thanks{}
	
	
	
	\date{}

	\begin{abstract}
		We consider the focusing $5$d Hartree equation, which is $L^2$-supercritical, with finite energy initial data, and investigate the solutions at the mass-energy threshold. We establish the existence of special solutions following the work of Duyckaerts-Roudenko \cite{DR10} for the $3$d focusing cubic nonlinear Schr\"odinger equation (NLS). In particular, apart from the ground state solution $Q$, which is global but non-scattering, there exist special solutions $Q^+$ and $Q^-$, which in one time direction approach $Q$ exponentially, and in the other time direction $Q^+$ blows up in finite time and $Q^-$ exists for all time, exhibiting scattering behavior. We then characterize all radial threshold solutions either as scattering and blow up solutions in both time directions (similar to the solutions under the mass-energy threshold, see Arora-Roudenko \cite{AKAR1}), or as the special solutions described above. To obtain the existence and classification result, in this paper we perform a thorough and meticulous investigation of the spectral properties of the linearized operator associated to the Hartree equation.
	\end{abstract}

\maketitle

	\section{Introduction}
\noindent
We consider the focusing Hartree equation in 5d,
\begin{align}\label{H}
\begin{cases}
iu_t + \Delta u  + \left(\frac{1}{|x|^{3}}\ast |u|^2\right) u = 0, \quad t \in \R, ~ x \in \R^5, \\
u(x,0)=\,u_0(x)\in H^1(\R^5).
\end{cases}
\end{align}
The Hartree equation \eqref{H} can be considered as a classical limit of a field equation describing a quantum mechanical non-relativistic many-boson system interacting through a two body potential. Lieb \& Yau \cite{LY87} mention it in a context of developing theory for stellar collapse and a special case with $\frac{1}{|x|}$ in 3D is referred to as the Coulomb potential, which goes back to the work of Lieb \cite{Lieb77}. Unlike the standard NLS with pure nonlinearity $|u|^pu$, the distinct feature of Hartree equation is that it models systems with long-range interactions. 

The Hartree equation \eqref{H} is locally wellposed in $H^1(\R^5)$, see \cite{GV80}, \cite{C03}. The solutions to \eqref{H}, during their lifespan, formally conserve the total energy 
\begin{equation}\label{energyH}
E[u] \defeq \frac{1}{2}\int_{\R^5}^{}|\nabla u|^2- \frac{1}{4}Z_H(u)  = E[u_0], \ \ \text{where}\ \ Z_H(u) \defeq \int_{\R^5}(|x|^{-3}\ast|u|^2)|u|^{2}\,dx,
\end{equation}
along with their mass and momentum,
\begin{equation}\label{mass+moment}
M[u]\defeq\int_{\R^5}^{}|u|^2\,dx = M[u_0],\ \quad P[u]\defeq\Im\left(\int_{\R^5}^{}\bar{u}\,\nabla u\,dx\right) = P[u_0].
\end{equation}
The solutions to \eqref{H} 
is invariant under the scaling
\begin{equation}\label{scale}
u_{\lambda}(x,t) = \lambda^2 u(\lambda x,\lambda^2 t),
\end{equation}
which identifies the criticality of the equation \eqref{H} as 
$\dot{H}^{1/2}$-critical. 
 The equation \eqref{H} admits solitary wave (global and non-scattering) solutions of the form $u(x,t)=e^{it}Q(x)$, where $Q$ solves the following nonlinear nonlocal elliptic equation
\begin{equation}\label{HQ}
	-Q + \Delta Q +\left(\frac{1}{|x|^{3}}\ast Q^{2}\right)Q=0.
\end{equation}
   We are interested in the ground state solution $Q$, i.e., positive, smooth, decaying at infinity solution of the equation \eqref{HQ}. The existence of solutions using a variational approach was first proved by Lieb \cite{Lieb77}, see also Lions \cite{Lions80} and a relatively recent work by Moroz \& Schaftingen \cite{MS13}. The existence result via ordinary differential equations approach was proved in \cite{Penrose98, TM99, ChSV08}. The regularity, decay asymptotics and symmetry of solutions have been shown in \cite{CCS12}, see also \cite{MZ10, MS13}. It is also known that in this case, the ground state is unique, see \cite{AKAR1} for dimension $2 < d < 6$ and the convolution of the form $\frac{1}{|x|^b}\ast |u|^2$ with $b=d-2$. The uniqueness proof goes back to Lieb \cite{Lieb77} for $d=3$ and in \cite{KLR09} for $d = 4$. We note that the uniqueness of the ground state is only known when the convolution is of the form $\frac{1}{|x|^{d-2}}\ast |u|^2$; uniqueness in a general case is not known.
   
In \cite{GW10}, the authors studied the dichotomy for global versus finite time solutions for \eqref{H} under the mass-energy threshold $M[u_0]E[u_0] < M[Q]E[Q]$ and showed $H^1(\R^5)$ scattering for the global solutions. In particular, depending on the size of the initial mass and gradient, the dichotomy holds, i.e., if $\|u_0\|_{L^2(\R^5)}\|\nabla u_0\|_{L^2(\R^5)} < \|Q\|_{L^2(\R^5)}\|\nabla Q\|_{L^2(\R^5)}$, then the solution will exist for all time and scatter in $H^1(\R^5)$; if $\|u_0\|_{L^2(\R^5)}\|\nabla u_0\|_{L^2(\R^5)} > \|Q\|_{L^2(\R^5)}\|\nabla Q\|_{L^2(\R^5)}$, then the solution will blow up in finite time. 

Our aim is to understand the dynamics of solutions to the Hartree equation
exactly at the
mass-energy threshold 
\begin{equation}\label{ME}
M[u_0]E[u_0] = M[Q]E[Q].
\end{equation}
The first result studying the behavior of solutions at the mass-energy threshold goes back to Duyckaerts \& Merle \cite{DM09} for the energy-critical focusing nonlinear Schr\"odinger equation (NLS)
\begin{align*}
\begin{cases}
iu_t + \Delta u  + |u|^{\frac{4}{d-2}} u = 0, \quad t \in \R, ~ x \in \R^d, ~~ d\in\{3,4,5\}, \\
u(x,0)=\,u_0(x)\in \dot{H}^1(\R^d).
\end{cases}
\end{align*}
This was followed by the result of Duyckaerts \& Roudenko \cite{DR10}, in which the authors described the behavior of solutions to the focusing $\dot{H}^{1/2}$-critical NLS in $3$d 
\begin{align*}
\begin{cases}
iu_t + \Delta u  + |u|^{2} u = 0, \quad t \in \R, ~ x \in \R^3, \\
u(x,0)=\,u_0(x)\in H^1(\R^3).
\end{cases}
\end{align*}
 Both of the above results exhibited a richer dynamics for the behavior of solutions, in particular, there exist two special solutions, $Q^+$ and $Q^-$. These special solutions in one time direction approach the soliton $Q$ in $H^1(\R^5)$, i.e., there exists $e_0>0$ such that
 \[
 \|Q^{\pm} - e^{it}Q\|_{H^1(\R^5)}\leq Ce^{-e_0t} \ \ \text{for}\ \ t\geq 0.
 \]
However, in the other time direction $Q^-$ scatters and $Q^+$ blows up in finite time. The existence and classification result of these special solutions is derived from the spectral properties of the linearized operator around the soliton $e^{it}Q$. These special solutions satisfy $M[Q^{\pm}] = M[Q]$ and $E[Q^{\pm}]=E[Q]$, but the gradient size is different: $\|\nabla Q^-\|_{L^2(\R^5)} < \|\nabla Q\|_{L^2(\R^5)}$ and $\|\nabla Q^+\|_{L^2(\R^5)} > \|\nabla Q\|_{L^2(\R^5)}$. For the extension of \cite{DR10} to all intercritical cases, see \cite{CFR22}. For a similar result on NLS with an obstacle, refer to \cite{DLR22}.

Our goal is to study the behavior of solutions to \eqref{H} at the mass-energy threshold \eqref{ME}, which to our best knowledge is an open problem, specifically, due to the nonlocal potential of convolution type and understanding spectral properties of the operators appearing because of this potential. For this purpose in this paper we consider this specific 5d $\dot{H}^{1/2}$-critical case as it allows us to develop the spectral theory of the linearized operators (see Section \ref{S:lin-oper}), which is one of the main ingredients to obtain the threshold characterization and solutions classification. The reason for the $5d$ is that the Hartree equation (with the convolution $|x|^{d-2}$) is mass-critical in $4d$ and energy-critical in $6d$. But more importantly, this is the case where the decay of the potential is sufficient to develop the spectral properties of the linearized operators.  

We characterize all radial threshold solutions and show that in this case the classification is similar to the NLS equation: either a solution (i) scatters in both time directions, or (ii) blows up  in both time directions, or (iii) behaves as $Q$ up to symmetries, or (iv) exhibits asymmetric behavior in different time directions: as either of the special solutions $Q^+$ and $Q^-$, more precisely, approaching (exponentially) $Q$ (the ground state
solution) in one time direction and in the other time direction behaving either as $Q^+$, which has a finite time of existence, or as $Q^-$, which exists for all time and scatters.

The first main result of this paper is to prove the existence of special solutions to \eqref{H} exactly at the critical mass-energy threshold. 
\begin{theorem}\label{mainthm1}
	There exist two radial solutions $Q^+$ and $Q^-$ of \eqref{H} such that 
	\begin{enumerate}
		\item[a.] $Q^\pm$ is defined at least in $[0,+\infty)$ with $M[Q^+] = M[Q] = M[Q^-]$, $E[Q^+] = E[Q] = E[Q^-]$ and there exists $e_0>0$ such that for all $t\geq 0$ 
		$$
		\|Q^{\pm} - e^{it}Q\|_{H^1(\R^5)} \leq Ce^{-e_0t},
		$$
		\item[b.] $\|\nabla Q^+\|_{L^2(\R^5)} > \|\nabla Q\|_{L^2(\R^5)}$ and $Q^+$ blows up in finite negative time,
		\item [c.] $\|\nabla Q^-\|_{L^2(\R^5)} < \|\nabla Q\|_{L^2(\R^5)}$ and $Q^-$ is globally defined and scatters for negative time.
	\end{enumerate}
\end{theorem}

The second result of the paper concerns with the classification of all radial solutions to \eqref{H} exactly at the mass-energy threshold. While the analysis works for all solutions, it is our proof of coercivity for the linearized energy (Proposition \ref{positivity}) requires the radiality assumption. 
\begin{theorem}\label{mainthm2}
	Suppose $u(t)$ is a radial solution to \eqref{H} with $u_0\in H^1_{rad}(\R^5)$ satisfying \eqref{ME}. Then 
	\begin{enumerate}
		\item[a.] If $\|u_0\|_{L^2(\R^5)}\|\nabla u_0\|_{L^2(\R^5)} > \|Q\|_{L^2(\R^5)}\|\nabla Q\|_{L^2(\R^5)}$, then either $u(t)$
		blows-up in finite positive and negative time, or $u = Q^+$ up to the symmetries.  
		\item[b.] If $\|u_0\|_{L^2(\R^5)}\|\nabla u_0\|_{L^2(\R^5)} = \|Q\|_{L^2(\R^5)}\|\nabla Q\|_{L^2(\R^5)}$, then $u=Q$ up to the symmetries.
		\item[c.] If $\|u_0\|_{L^2(\R^5)}\|\nabla u_0\|_{L^2(\R^5)} < \|Q\|_{L^2(\R^5)}\|\nabla Q\|_{L^2(\R^5)}$, then $u(t)$ is globally defined and either it scatters in both time directions, or $u=Q^-$ up to the symmetries.     
	\end{enumerate} 
\end{theorem}
Here, the meaning of `$u=v$ up to the symmetries' is that $u$ and $v$ agree up to scaling, phase rotation, time translation and time reversal symmetries associated to the equation \eqref{H}. We emphasize that the treatment of the Hartree case is different from the nonlinear Schr\"odinger case, in particular, the nonlocal nonlinear term of convolution type presents various challenges in proving the spectral properties of the linearized operator associated to the Hartree equation \eqref{H}. For example, for challenges in the 4d model refer to  \cite{KLR09}, and in more general case, refer to \cite{Thesis}. Nevertheless, we are able to study the linearized operators via Newton's theorem (see Claim \ref{Tkernel}) along with spherical harmonics, which allow us to write the Newtonian potential in terms of Gegenbauer polynomials (see \eqref{gen-expansion} for details) and obtain characterization of their null spaces, see Proposition \ref{null-sp}. Next, we prove the existence of eigenfunctions for the linearized operators and using the theory of Bessel potentials we show the decay of eigenfunctions at infinity. This culminates with the proof of characterization of real spectrum of the linearized operator. We then proceed with finding the special solutions and completing both theorems above.
\smallskip
 
The paper is organized as follows: in Section \ref{S:lin-oper} we carefully investigate the spectral properties of the linearized operators coming from the linearization of the Schr\"odinger operator with convolution Hartree potential around the ground state. This allows as to investigate the special solutions $Q^+$ and $Q^-$ in Section \ref{S:special-solutions}. 
In Sections \ref{S:less} and \ref{S:more} we study the threshold radial solutions with gradients being either less or more of those of the ground state, respectively, and in Section \ref{S:Uniq} we obtain the uniqueness and finish the proof of Theorem \ref{mainthm2}.

\textbf{Acknowledgments.}
A.K.A. and S.R. thank Thomas Duyckaerts for helpful discussions and suggestions on this project,  
A.K.A. is especially grateful to Oussama Landoulsi for many enlightening conversations related to this topic. 
The research of A.K.A. and S.R. on this project was partially supported by the NSF grant DMS-1927258 (PI:Roudenko).
\smallskip

\textbf{Notation.} For a given $1\leq p <\infty$, we define the space
\[
L^p(\R^5) = \Big\{f\,:\,\R^5\rightarrow\R\,\,:\,\,\int_{\R^5}|f(x)|^p\,dx<\infty\Big\},
\]
equipped with the norm 
\[
\|f\|_{L^p(\R^5)}=\left(\int_{\R^5}|f(x)|^5\,dx\right)^{{1}/{5}}.
\]
If $p=\infty$, we denote by $L^{\infty}(\R^5)$ the space of essentially bounded functions. Define the Schwartz space $\mathcal{S}(\R^5)$, the space of the $C^{\infty}$-functions decaying at infinity, i.e.,
\[
\mathcal{S}(\R^5) = \{f\in C^{\infty}(\R^5)\,\,:\,\,x^\alpha\partial^\beta f\in L^{\infty}(\R^5)\,\,\text{for every}\,\,\alpha,\beta\in\Z_+^{5}\}.
\] 
The Fourier transform on $\R^5$ for $f\in\mathcal{S}(\R^5)$ is given by
\[
\hat{f}(\xi) = (2\pi)^{-5/2}\int_{\R^5}^{}e^{-ix\xi}f(x)\,dx.
\] 
Define the operator $|\nabla|^s$ for $s\geq 0$ as
\[
\widehat{|\nabla|^sf}(\xi) = |\xi|^s\hat{f}(\xi),
\]
which is also known as the Riesz operator of order $-s$, where $s>0$. This allows us to define the homogeneous Sobolev space 
\[
\dot{W}^{s,p}(\R^5) = \Big\{f\in\mathcal{S}'(\R^5)\,\, :\,\,|\nabla|^sf\in L^p(\R^5) \Big\},
\]
and the associated norm  
\[
\|f\|_{\dot{W}^{s,p}(\R^5)} = \||\nabla|^sf\|_{L^p(\R^5)},
\]
for $p \geq 1$.  
The inhomogeneous Sobolev space is defined by 
\[
W^{s,p}(\R^5) = \Big\{f\in\mathcal{S}'(\R^5)\,\, :\,\,(1+|\nabla|^2)^{\frac{s}{2}}f\in L^p(\R^5) \Big\},
\]
equipped with the norm 
\[
\|f\|_{W^{s,p}(\R^5)} = \|(1+|\nabla|^2)^{\frac{s}{2}}f\|_{L^p(\R^5)},
\]
where 
\[
\big((1+|\nabla|^2)^{\frac{s}{2}}f\big){}\,\widehat{} = \langle\xi\rangle^s\hat{f}.
\]
Here, $\langle\xi\rangle = (1 + |\xi|^2)^{1/2}$, often called the Japanese bracket. If $p=2$, we denote 
\[
\dot{H}^s(\R^5) = \dot{W}^{s,2}(\R^5) \quad \text{and} \quad H^s(\R^5) = W^{s,2}(\R^5).
\]
For any spacetime slab $I\times\R^5$, we use $L_t^qL_x^r(I\times\R^5)$ to denote the space of functions $u\: :\: I\times\R^5 \rightarrow \C$, whose norm is
$$
\|u\|_{L_I^qL_x^r(I\times\R^5)} = \left(\int_{I}^{}\left(\int_{\R^5}|f(x,t)|^r\,dx\right)^{{q}/{r}} \,dt\right)^{{1}/{q}} < \infty.
$$  
For the next notation, we need the following restriction on the range of pairs $(q, r)$ (as introduced in \cite{G}, also see \cite{HR08}): 
\begin{align}\label{range}
\left(\frac{2}{1-s}\right)^+\leq q\leq \infty,\quad \frac{10}{5-2s}\leq r \leq \left(\frac{10}{3}\right)^-,\,\,\text{if}\,\, N\geq 3.
\end{align}
Here, $n^+$ is a fixed number (slightly) greater than $n$ such that $\frac{1}{n}=\frac{1}{n^+}+\frac{1}{(n^+)'}$. Respectively, $n^-$ is a fixed number (slightly) less than $n$. 
We now introduce the $S(\dot{H}^s)$ notation for $u\: :\: I\times\R^5 \rightarrow \C$: 
\begin{equation*}
\|u\|_{S(\dot{H}^s;I)}=\sup\Big\{ \|u\|_{L_I^q \, L_x^r}: (q,r) ~\mbox{satisfies ~ in ~} \frac{2}{q} + \frac{N}{r} = \frac{5}{2} - s\,\,  ~\mbox{and}~ \eqref{range} \Big\}.
\end{equation*}
Lastly, we write $X \lesssim Y$ (or $Y \gtrsim X$) whenever $X \leq CY$ for some constant $C > 0$. Similarly, we will write $X\sim Y$ if $X\lesssim Y \lesssim X$.

\section{Linearized equation around $Q$}\label{S:lin-oper}
\noindent
We decompose the solution $u(x,t)$ to \eqref{H} around $Q$ via the following decomposition:
\begin{equation}\label{can-decomp}
	u(x,t) = e^{it}(Q(x)+h(x,t)).
\end{equation}
Here, $h=h_1+ih_2$ is a solution of the equation
\begin{equation}\label{lineq}
	h_t + \mathcal{L}h = R(h),\quad \mathcal{L} := \begin{pmatrix}
	0 & -L_-\\
	L_+ & 0
	\end{pmatrix},
\end{equation}
where from the equation \eqref{H} it follows
\begin{equation}\label{L+exp-H}
	L_+h_1 := -\Delta h_1 + h_1 - \big(|x|^{-3} \ast Q^2\big)h_1 - 2 \big(|x|^{-3} \ast (Qh_1)\big)Q,
\end{equation}
\begin{equation}\label{L-exp-H}
L_-h_2 := -\Delta h_2 + h_2 - \big(|x|^{-3} \ast Q^2\big)h_2
\end{equation}
and
\begin{equation}\label{Rexp-H}
	R(h) := i \Big[\big(|x|^{-3} \ast |h|^2\big)Q + 2 \big(|x|^{-3} \ast (Qh_1)\big)h + \big(|x|^{-3} \ast |h|^2\big)h\Big].
\end{equation}

\subsection{Spectral properties of $L_+$ and $L_-$}
We recall the Gagliardo-Nirenberg inequality of convolution type in the setting of \eqref{H}, or more precisely, for the potential term in \eqref{energyH}
\begin{equation}\label{H-GNC}
Z_H(u)=\int_{\R^5}\int_{\R^5}\frac{|u(x)|^2\,|u(y)|^2}{|x-y|^3}\,dx\,dy\leq C_{GN}\,\|\nabla u\|_{L^2(\R^5)}^{3}\,\|u\|_{L^2(\R^5)}
\end{equation}
with the sharp constant $C_{GN}$ (see section 4 in \cite{AKAR1} for details). 

From \cite{AKAR1} (see also \cite{MS13}), we know that $Q$ is the unique minimizer of the Gagliardo-Nirenberg functional given by
\begin{equation}\label{GN-func}
J(u)=\frac{\|\nabla u\|_{L^2(\R^5)}^3\|u\|_{L^2(\R^5)}}{Z_H(u)}
\end{equation}
for $u\in H^1(\R^5)$ with $u\not\equiv 0$ (note that a similar statement holds for other dimensions as well).
Therefore, 
\begin{equation}\label{est4}
	\frac{d^2}{d\epsilon^2}\bigg|_{\epsilon=0}J(Q+\epsilon h)\geq 0,\quad\forall h\in C_0^{\infty}(\R^5),
\end{equation}
where
$$
J(Q+\epsilon h)=\frac{\|\nabla Q + \epsilon\nabla h \|_{L^2(\R^5)}^3\|Q + \epsilon h\|_{L^2(\R^5)}}{Z_H(Q+\epsilon h)}.
$$
Calculating the first (variational) derivative of above expression, we obtain 
\begin{align*}
\frac{J'}{J} =&\,\, \frac{\int\left(Qh_1+\epsilon|h|^2\right)}{\|Q+\epsilon h\|_{L^2(\R^5)}^2} + \frac{3\int\left(\nabla Q\nabla h_1+\epsilon|\nabla h|^2\right)}{\|\nabla Q+\epsilon \nabla h\|_{L^2(\R^5)}^2}\\
&-\frac{\int\left(|x|^{-3}\ast |Q+\epsilon h|^{2}\right)(2Qh_1+2\epsilon|h|^2)+\int\left(|x|^{-3}\ast (2Qh_1+2\epsilon|h|^2)\right)|Q+\epsilon h|^{2}}{Z_H(Q+\epsilon h)}.
\end{align*}
Multiplying the above expression with $J$ and taking the second derivative 
\begin{align*}
\frac{J''}{J} = & \Bigg(\frac{\int\left(Qh_1+\epsilon|h|^2\right)}{\|Q+\epsilon h\|_{L^2(\R^5)}^2} + \frac{3\int\left(\nabla Q\nabla h_1+\epsilon|\nabla h|^2\right)}{\|\nabla Q+\epsilon \nabla h\|_{L^2(\R^5)}^2}\\
&-\frac{\int\left(|x|^{-3}\ast |Q+\epsilon h|^{2}\right)(2Qh_1+2\epsilon|h|^2)+\int\left(|x|^{-3}\ast (2Qh_1+2\epsilon|h|^2)\right)|Q+\epsilon h|^{2}}{Z_H(Q+\epsilon h)}\Bigg)^2\\
&+\frac{\int |h|^2}{\|Q+\epsilon h\|_{L^2(\R^5)}^2}-\frac{2\left(\int Qh_1+\epsilon|h|^2\right)^2}{\|Q+\epsilon h\|_{L^2(\R^5)}^4} + \frac{3\int |\nabla h|^2}{\|\nabla Q+\epsilon\nabla h\|_{L^2(\R^5)}^2}\\
&-\frac{6\left(\int\nabla Q\nabla h_1+\epsilon|\nabla h|^2\right)^2}{\|\nabla Q+\epsilon \nabla h\|_{L^2(\R^5)}^4} - \frac{\int\left(|x|^{-3}\ast 2|h|^2\right)|Q+\epsilon h|^2 + \int\left(|x|^{-3}\ast |Q+\epsilon h|^2\right)2|h|^2}{Z_H(Q+\epsilon h)}\\
&-\frac{2\int\left(|x|^{-3}\ast (2Qh_1+2\epsilon|h|^2)\right)(2Qh_1+2\epsilon|h|^2)}{Z_H(Q+\epsilon h)}\\
&+\frac{\left(\int\left(|x|^{-3}\ast |Q+\epsilon h|^2\right)(2Qh_1+2\epsilon|h|^2)+\int\left(|x|^{-3}\ast (2Qh_1+2\epsilon|h|^2)\right)|Q+\epsilon h|^2\right)^2}{(Z_H(Q+\epsilon h))^2}.
\end{align*} 
Expanding the square, gathering common terms, substituting $\epsilon = 0$ and using the following relations (outcome of Pohozaev identities, see for instance \cite[Section 4]{AKAR2})
\begin{equation}
\label{eq:pohid}
Z_H(Q)=4\|Q\|_{L^2(\R^5)}^2=\frac{4}{3}\|\nabla Q\|_{L^2(\R^5)}^2,
\end{equation}
we get
\begin{align*}
\frac{J''}{J} = & \Bigg(\frac{\int\left(Qh_1\right)}{\|Q\|_{L^2(\R^5)}^2} + \frac{3\int\left(\nabla Q\nabla h_1\right)}{\|\nabla Q\|_{L^2(\R^5)}^2}-\frac{\int\left(|x|^{-3}\ast Q^{2}\right)(2Qh_1)+\int\left(|x|^{-3}\ast (2Qh_1)\right)Q^{2}}{Z_H(Q)}\Bigg)^2\\
&+\frac{\int |h|^2}{\|Q\|_{L^2(\R^5)}^2}-\frac{2\left(\int Qh_1\right)^2}{\|Q\|_{L^2(\R^5)}^4} + \frac{3\int |\nabla h|^2}{\|\nabla Q\|_{L^2(\R^5)}^2} -\frac{6\left(\int\nabla Q\nabla h_1\right)^2}{\|\nabla Q\|_{L^2(\R^5)}^4}\\
& - \frac{\int\left(|x|^{-3}\ast 2|h|^2\right)Q^2 + \int\left(|x|^{-3}\ast Q^2\right)2|h|^2}{Z_H(Q)}-\frac{2\int\left(|x|^{-3}\ast (2Qh_1)\right)(2Qh_1)}{Z_H(Q)}\\
&+\frac{\left(\int\left(|x|^{-3}\ast Q^2\right)(2Qh_1)+\int\left(|x|^{-3}\ast (2Qh_1)\right)Q^2\right)^2}{(Z_H(Q))^2}.
 \end{align*}
Using the ground state equation \eqref{HQ}, integrating by parts and substituting $J=\frac{3\sqrt{3}}{4}\|Q\|_{L^2(\R^5)}^2$, we obtain
\begin{align*}
J''=\frac{3\sqrt{3}}{4}&\left(\mathcal{L}h,h\right) \\
&+\frac{3\sqrt{3}}{4\|Q\|_{L^2(\R^5)}^2}\left(- \left(\int Qh_1\right)^2 + \frac{1}{3} \left(\int\nabla Q\nabla h_1\right)^2 + 2 \left(\int Qh_1\right)\left(\int\nabla Q\nabla h_1\right)\right).
\end{align*}
Now \eqref{est4} is equivalent to 
\begin{equation}\label{L-non-neg}
\big(L_-h_2,h_2\big)\geq 0,
\end{equation}
and
\begin{equation*}
\big(L_+h_1,h_1\big)\geq \frac{1}{\|Q\|_{L^2(\R^5)}^2}\left( \left(\int Qh_1\right)^2 - \frac{1}{3} \left(\int\nabla Q\nabla h_1\right)^2 - 2 \left(\int Qh_1\right)\left(\int\nabla Q\nabla h_1\right)\right).
\end{equation*}
The last estimate can be reduced to
 \begin{equation}\label{L+non-neg}
\big(L_+h_1,h_1\big)\geq 0 \quad\text{if}\quad h_1\perp \Delta Q.
\end{equation} 
Our next goal is to identify the null-space of $L_+$ and $L_-$ (see Proposition \ref{null-sp} below). Before the proof of this Proposition, we obtain several claims and lemmas for the operator $L_+$. We start with the following claim:
\begin{claim}\label{-evalue}
	The operator $L_+$ has exactly one negative eigenvalue.
	\begin{proof}
		First we observe that 
		$$
		\big(Q,L_+Q\big)=-2Z_H(Q)<0,
		$$
		where $Z_H$ is defined in \eqref{energyH}. Recalling the min-max principle (\cite[Theorem XIII.1]{RS4}), we have that the operator $L_+$ has at least one negative eigenvalue. Let $\mu_j$ be the $j$th eigenvalue of $L_+$. Then again using the min-max principle one can deduce that there is exactly one function $\phi_1$ in the orthogonal space (one dimensional subspace of $L^2$) such that 
		$$
		\mu_2(L_+) = \sup_{\phi_1}\quad\inf_{h_1\perp\phi_1;\,\,\,\|h_1\|_{L^2}=1}\big(h_1,L_+h_1\big)\geq 0,
		$$
		since $L_+$ is nonnegative for $\phi_1=\Delta Q$. Thus, $0$ is the second eigenvalue, since we know from a direct computation that $L_+\nabla Q=0$. Therefore, there can be at most one negative eigenvalue. Combining this with the starting observation, we conclude that $L_+$ has exactly one negative eigenvalue.  
	\end{proof}
\end{claim}
We follow the strategy developed by Lenzmann \cite{L09} and claim that 
\begin{claim}\label{Tkernel}
	$\text{ker}\,L_+ = \{0\}$ when $L_+$ is restricted to $L^2_{\text{rad}}(\R^5)$.
	\begin{proof}
		We rewrite the nonlocal term in $L_+$ using the Newton's theorem (see \cite[Theorem 9.7]{LiebL2001}): for any radial function $f(|x|)$, $r=|x|>0$, we have
		\begin{equation*}
		-(|x|^{-1}\ast f)(r)=\int_{0}^{r}K(r,s)\,f(s)\,ds - \int_{\R^5} \frac{f(|x|)}{|x|^3},
		\end{equation*}
		where $K(r,s)$ (in $5$d) is given by
		\begin{equation*}
		K(r,s)=\frac{8\pi^2}{3}\,s\Big(1-\frac{s^3}{r^3}\Big)\geq 0\quad\text{for}\,\,r\geq s.
		\end{equation*}
		Applying the Newton's theorem to $f=v$ for $v\in L^2_{\text{rad}}(\R^5)$, we obtain
		\begin{equation}\label{newL+}
		L_+ v= \mathscr{L}_+v - 2\left(\int_{\R^5}\frac{Q\,v}{|x|^3}\right)Q,
		\end{equation}
		where $\mathscr{L}_+$ is given by
		\begin{equation}\label{newL+exp}
		\mathscr{L}_+v = -\Delta v + v - (|x|^{-3}\ast Q^2)v + W\,v,
		\end{equation}
		with
		\begin{equation}\label{W}
		(W\,v)(r) = 2\left(\int_{0}^{r}K(r,s)\,Q(s)\,v(s)\,ds\right)Q(r).
		\end{equation}
		Before continuing with the proof of Claim \ref{Tkernel}, we establish that solutions $v$ to the linear equation $\mathscr{L}_+v = 0$ are exponentially growing. 
		\begin{lemma}\label{exp.growth}
			Suppose the radial function $v$ solves $\mathscr{L}_+v = 0$ with $v(0)\neq 0$ and $v^{\prime}(0) = 0$. Then the function $v$ has no sign change and grows exponentially as $r \rightarrow \infty$. More precisely, for any	$0 <\delta < 1$, there exist constants $C > 0$ and $R > 0$ such that
			$$
			|v(r)|\geq C e^{+\delta r},
			$$
			for all $r\geq R.$ In particular, we have that $v \notin L^2_{\text{rad}}(\R^5).$
			\begin{proof}
				
				Assuming without loss of generality that $v(0) > Q(0)> 0$,  we write $\mathscr{L}_+v=0$ as 
				\begin{equation}\label{eq1}
				v^{\prime\prime}(r) +\frac{4}{r}v^{\prime}(r) = v(r) - (|x|^{-3}\ast Q^2)v(r)  + (W\,v)(r).
				\end{equation}  
				Next, note that $Q(r)$ satisfies
				\begin{equation}\label{eq2}
				Q^{\prime\prime}(r) +\frac{4}{r}Q^{\prime}(r) = Q(r) - (|x|^{-3}\ast Q^{2})Q(r).
				\end{equation}
				Multiplying \eqref{eq1} with $Q(r)$ and \eqref{eq2} with $v(r)$, subtracting the two resulting equations, we get after multiplying by $r^4$
				\begin{equation}\label{eq3}
				\frac{d}{dr}(r^4(Qv^{\prime}-Q^{\prime}v)) = r^4\,(W\,v)(r)Q(r).
				\end{equation} 
				Integrating \eqref{eq3}, we obtain
				\begin{equation}\label{eq4a}
				r^4(Qv^{\prime}-Q^{\prime}v)(r) = \int_{0}^{r}(W\,v)(s)\,Q(s)\,s^4\,ds.
				\end{equation}
				Recalling that $v(0) > Q(0)$ and by continuity of $v$ (and of $Q$), we have that $v(r) > Q(r)$ for some sufficiently small $r>0$. Suppose now that $v(r)$ intersects $Q(r)$ at $r_1 > 0$ for the first time. Then, since $v'(r)<0$ and $Q(r)>0$ we have that the left hand-side of
			\eqref{eq4a}	at $r_1$ is non-positive due to the monotonicity of both $v(r)$ and $Q(r)$, however, the right hand-side of 
			\eqref{eq4a} is positive, since $v(r) > Q(r)$ for $0 < r < r_1$. This leads to a contradiction, thus, $v(r)$ and $Q(r)$ do not intersect, implying that 
				\begin{equation}\label{eq5}
				v(r) > Q(r)	\quad \text{for all}\,\,\, r \geq 0.
				\end{equation}
We then write \eqref{eq4a} (with the observation that $Q(r)>0$) as
\begin{equation}\label{eq4}
				r^4\left(\frac{v(r)}{Q(r)}\right)^{\prime} = \frac{1}{Q^2(r)}\int_{0}^{r}(W\,v)(s)\,Q(s)\,s^4\,ds.
				\end{equation}
				Inserting \eqref{eq5} together with \eqref{W} into the \eqref{eq4} gives
				\begin{equation}\label{eq6}
				r^4\left(\frac{v(r)}{Q(r)}\right)^{\prime} > \frac{2}{Q^2(r)}\int_{0}^{r}s^4Q^{2}(s)\int_{0}^{s}K(s,\rho)\,Q^{2}(\rho)\,d\rho\,ds.
				\end{equation} 
				We now consider $H=-\Delta + V$, where $V= -(|x|^{-3}\ast Q^{2}).$ Since $HQ = -Q$ with $V\in L^q(\R^5)+L^{\infty}(\R^5)$ for $q>5/2$ (as shown in the proof of Lemma \ref{spectrum}) and $V\rightarrow 0$ as $|x|\rightarrow \infty$, the result  of \cite[Theorem 3.2]{CS81} shows that we have the following bound: for any $\epsilon > 0$, there exists constants $A_\epsilon$, $B_\epsilon > 0$ such that
				\begin{equation}\label{eq7}
				B_\epsilon e^{-(1+\epsilon)r} \leq Q(r) \leq A_\epsilon e^{-(1-\epsilon)r}\quad\text{for all}\,\,\, r\geq 0.
				\end{equation} 
				Substituting the bounds \eqref{eq7} into \eqref{eq6}, we get
				\begin{equation*}
				r^4\left(\frac{v(r)}{Q(r)}\right)^{\prime} > C_\epsilon\, e^{2(1-\epsilon)r}\int_{0}^{r}s^4e^{-2(1+\epsilon)s}(s)\int_{0}^{s}K(s,\rho)\,e^{-2(1+\epsilon)\rho}(\rho)\,d\rho\,ds.
				\end{equation*} 
				Since the integral in the above estimate converges as $r\rightarrow \infty$ to some finite positive value, there exists some $R  > 0$ such that
				\begin{equation*}
				r^4\left(\frac{v(r)}{Q(r)}\right)^{\prime} > C_\epsilon\, e^{2(1-\epsilon)r}\quad\text{for all}\,\,\, r\geq R,
				\end{equation*}
				for some constant $C_\epsilon > 0$. Integrating the above lower bound and using \eqref{eq7}, we obtain
				\begin{equation*}
				v(r) \geq C_\epsilon\,Q(r)\,\int_{R}^{\infty}\frac{ e^{2(1-\epsilon)r}}{r^4}\,dr\geq \frac{C_\epsilon\,Q(r)}{R^{10}}\,\int_{R}^{\infty} e^{2(1-\epsilon)r}\,dr  \geq C\, e^{(1-3\epsilon)r}
				\end{equation*} 
				with some constants $C > 0$ and $R \gg 1$. Taking $0< \epsilon < \frac{1-\delta}{3}$ for any $0 < \delta < 1$ completes the proof of Lemma \ref{exp.growth}.	
			\end{proof}
		\end{lemma}
		Coming back to the proof of Claim \ref{Tkernel}, suppose that there exists $\eta \in L^2_{\text{rad}}(\R^5)$ with $\eta \neq 0$ such that $L_+\eta = 0$. Then, by \eqref{newL+}, $\eta$   solves the inhomogeneous problem
		\begin{equation}\label{eq8}
		\mathscr{L}_+\eta =  2\left(\int_{\R^5}\frac{Q\,\eta}{|x|^3}\right)Q\defeq 2\,\sigma\,Q,\quad \text{where}\,\,\,\sigma = \int_{\R^5}\frac{Q(x)\,\eta(x)}{|x|^3}\,\,\,\text{is a constant}.
		\end{equation}
		Consider $\eta = v + w$, where $v$ is a solution to the homogeneous equation $\mathscr{L}_+v = 0$ and $w$ is a particular solution to \eqref{eq8}. We want to find $w\in L^2_{\text{rad}}(\R^5)$.  Consider a radial function $Q_1 = (5/2)Q +r\partial_{r}Q \in L^2_{\text{rad}}(\R^5)$ for which we know that $L_+Q_1 = -2Q$.
		Using this, we apply \eqref{newL+} to $Q_1$ to obtain 
		\[
		\mathscr{L}_+Q_1 = -2Q + 2 Q\int_{\R^5}\frac{Q\,Q_1}{|x|^3}\defeq 2\,(\tau-1)\,Q,\ \ \text{where}\ 	 \tau = \int_{\R^5}\frac{Q(x)\,Q_1(x)}{|x|^3}\,\,\,\text{is a constant}.\]
		We point out that $\tau \neq 1$ must hold, because otherwise taking $v=Q_1$ and applying Lemma \ref{exp.growth} would yield $Q_1\notin L^2_{\text{rad}}(\R^5)$, which yields a contradiction. Thus,  
		$$
		w = \left(\frac{\int_{\R^5}\frac{Q\,\eta}{|x|^3}}{ \int_{\R^5}\frac{Q\,Q_1}{|x|^3}-1}\right)\,Q_1(r) = \frac{\sigma}{\tau-1}\,Q_1\in L^2_{\text{rad}}(\R^5)
		$$
		is a particular solution to \eqref{eq8}. Now we present two scenarios: 
		
		$\underline{\textit{Scenario 1:}}$ Suppose that $v\equiv 0$. Then, we have $\eta = w$ and $\sigma\neq 0$ (because $\sigma=0$ implies $w=0$, however $\eta\neq 0$ by initial assumption). This contradicts $L_+\eta =0$, since $L_+w = -\frac{2\sigma}{\tau-1}Q\neq 0$. Therefore, $v\neq 0$.
		
		$\underline{\textit{Scenario 2:}}$ Suppose now that $v(0)\neq 0$. Then by Lemma \ref{exp.growth}, we deduce that $v\notin  L^2_{\text{rad}}(\R^5)$, which is again a contradiction, since both $\eta$ and $w$ are in $L^2_{\text{rad}}(\R^5)$. Thus, $v(0) = 0$ holds. Then, $v$ solves $\mathscr{L}_+v = 0$ with $v(0) = 0$ and $v^{\prime}(0)=0$ (by smoothness of $v$). Thus, Volterra integral theory (for example, see Lemmas 2.4-2.6 and Theorem 2.1 in \cite{YRZ1}) guarantees a unique local $C^2$ solution to the initial-value problem $\mathscr{L}_+v = 0$ for a given initial data $v(0)=v_0\in\R$ and $v^{\prime}(0)=0$. Hence, $v(0)=0$ and $v^{\prime} = 0$ implies that $v\equiv 0$, which again yields a contradiction. This completes the proof of Claim \ref{Tkernel}.
	\end{proof}
\end{claim}

We next write the Laplacian in $5$d in spherical coordinates as
$$
\Delta f = \frac{\partial^2f}{\partial r^2} + \frac{4}{r} \frac{\partial f}{\partial r} + \frac{1}{r^2}\Delta_{\mathbb{S}^4}f,
$$
where $\Delta_{\mathbb{S}^4}$ is the Laplace-Beltrami operator on $\mathbb{S}^4$. It is well-known that the (homogeneous) harmonic functions, i.e., solutions of the Laplace equation can be found by separation of variables. Thus, we write $f(r,\theta,\phi)=r^lg(\theta,\phi)$ and obtain that 
$$
\Delta f = r^{l-2}\big(l(l+3)\,g + \Delta_{\mathbb{S}^4}g\big).
$$
Hence, 
$$
\Delta f = 0 \quad \text{if and only if}\quad \Delta_{\mathbb{S}^4} g = -l(l+3)\,g,
$$
i.e., $g$ is an eigenfunction of $\Delta_{\mathbb{S}^4}$ for the eigenvalue $-l(l+3)$. Hence, the Laplacian (in 5d) in spherical coordinates can be expressed as
\begin{equation}\label{sLaplacian}
\Delta_{(l)} = \frac{\partial^2}{\partial r^2} + \frac{4}{r} \frac{\partial }{\partial r} - \frac{l(l+3)}{r^2}.
\end{equation}
Therefore, decomposing any solution $\eta\in L^2(\R^5)$ of $L_+\eta = 0$ via spherical harmonics, we have 
\begin{equation}\label{sdecompose}
\eta = \sum_{l=0}^{\infty}\sum_{k=1}^{C_l}\eta_{l,k}(r)Y_{l,k}(\omega),
\end{equation}
where $C_l=\frac{(2l+3)(l+2)(l+1)}{6}$, $x = r\omega$ with $r=|x|$ and $\omega\in\mathbb{S}^4$. Here, $Y_{l,k}$ denotes the spherical harmonics, which is the eigenfunction corresponding to the eigenvalue $-l(l+3)$ of the spherical Laplacian $\Delta_{\mathbb{S}^4}$.

Now we introduce the Gegenbauer (or ultraspherical) polynomials $P_k^\lambda$ (see \cite[Chapter 4, Section 2]{Stein71}), which are considered to be a natural extensions of Legendre polynomials. We remark that the Newtonian potential in $\R^N$ has the expansion 
\begin{equation}\label{gen-expansion}
	\frac{1}{|x-y|^{b}} = |x|^{-b}\left(1+\frac{|y|^2}{|x|^2}-2\frac{|x|}{|y|}\cos\theta\right)^{-b}= \sum_{l=0}^{\infty}\frac{|x|^l}{|y|^{l+b}}P_l^b(\cos\theta),
\end{equation}
where $0\leq \frac{|y|}{|x|}<1$, $b=N-2>0$ and $P_l^b$ is called the Gegenbauer polynomial of degree $l$ associated with $b$. When $N = 3$, this gives the Legendre polynomial expansion of the gravitational or Coulomb potential, see \cite{L09} for pseudorelativistic Hartree equation. Therefore, to account for the case $\frac{|y|}{|x|}>1$ we use \eqref{gen-expansion} to get    
\begin{equation}\label{expansion}
	\frac{1}{|x-x^\prime|^3} = \sum_{l=0}^{\infty}\frac{r_{min}^l}{r_{max}^{l+3}}\,P_l^{3}(\cos\theta),
\end{equation}
where $r_{min} = \min(|x|,|x^{\prime}|)$ and $r_{max} = \max(|x|,|x^{\prime}|)$. Using the property regarding the spherical harmonics called the addition theorem (see \cite[Theorem 2.26]{Morimoto98}), we have
\begin{equation}\label{multipole}
\frac{1}{|x-x^{\prime}|^3} = \frac{8\,\pi^2}{3}\sum_{l=0}^{\infty}\frac{1}{C_l}\,\frac{r_{min}^l}{r_{max}^{l+3}}\sum_{k=1}^{C_l}Y_{l,k}(\omega)\,Y^*_{l,k}(\omega^{\prime}).
\end{equation}
 Then with $\eta$ given by \eqref{sdecompose}, we have 
\begin{equation}\label{sL+}
L_+\eta = 0\iff L_{+,(l)}\eta_{l,k} = 0,\quad\text{for}\,\,l=0,1,2,\ldots\,\,\text{and}\,\,k=1,\ldots,C_l,
\end{equation}
where 
\begin{equation}\label{sL+exp}
(L_{+,(l)}f)(r) =-\Delta_{(l)}f(r) + f(r) + (V_{(l)}f)(r)
\end{equation}
with $-\Delta_{(l)}$ given by \eqref{sLaplacian} and
\begin{align}\label{Vexp}\notag
&(V_{(l)}f)(r) = -\,V(r)\,f(r) - \frac{16\pi^2}{3\,C_l} (W_{(l)}f)(r)\\
&=-\big(|x|^{-3}\ast Q^{2}\big)(r)\,f(r) - \frac{16\pi^2}{3\,C_l}\left(\int_{0}^{\infty}\,\frac{r_{min}^l}{r_{max}^{l+3}}\,Q(s)\,f(s)\,s^4\,ds\right)Q(r).
\end{align}
Here, $r_{min} = \min(r,s)$ and $r_{max} = \max(r,s)$. Observe that $L_{+,(0)}f=0$ implies $f\equiv 0$ by Lemma \ref{exp.growth}, thus, it is sufficient to consider $l\geq 1$. Now we show that each $L_{+,(l)}$ satisfies Perron-Frobenius property. This result was proved for pseudo-relativistic Hartree equations in \cite[Lemma 7]{L09}. We next prove this for our $5$d gHartree case.  
\begin{lemma}\label{PFproperty}
	For each $l\geq 1$, the operator $L_{+,(l)}$ is essentially self-adjoint on $C^{\infty}_0(\R^+)\subset L^2(\R_+)$ in spherical coordinates	and bounded below. Moreover, each $L_{+,(l)}$ satisfies the Perron-Frobenius property, i.e., if $e_{0,(l)}$ denotes the lowest eigenvalue of $L_{+,(l)}$, then $e_{0,(l)}$ is simple and the corresponding eigenfunction $\phi_{0,(l)}(r)$ is strictly positive.
	\begin{proof}
		Observe that from \eqref{L+non-neg}, we have that 
		$L_{+,(l)}$ is bounded below. We also have from \cite{RS1} (Appendix to X.1, Theorem X.10 and Example 4) that $-\Delta_{(l)}$ (defined in \eqref{sLaplacian}) is essentially self-adjoint on $C_0^{\infty}(\R^+)$ provided $l(l+3)/r^2 \geq 3/4r^2$ (which is true since $l\geq 1$).  
		
		By invoking the Kato-Rellich theorem, we also have that 		$L_+$ (defined in \eqref{L+exp-H}) is essentially self-adjoint on $C_0^\infty(\R^+)$. This implies from the equivalence of $L_+$ and $L_{+,(l)}$ (see \eqref{sL+}) that $L_{+,(l)}$ (defined in \eqref{sL+exp}) is also essentially self-adjoint on $C_0^\infty(\R^+)$. 
		
		To prove the Perron-Frobenius property of $L_{+,(l)}$, we first recall that an operator $T$ is called {\it positivity improving} if $Tf$ is strictly positive whenever $f$ is positive. We also recall (see expression (7-15) in \cite{L09}) that  
		$
		(-\Delta_{(l)} + \mu )^{-1}$ is positivity improving on $L^2(\R^+,r^4dr)$ for all $\mu > 0$.	Next, since $Q$ is positive, we also have that $-V_{(l)}$ (defined in \eqref{Vexp}) is positivity improving on  $L^2(\R^+,r^4dr)$. Therefore, we have
		$$
		\big(L_{+,(l)} + \mu\big)^{-1} = \big(-\Delta_{(l)} + \mu\big)^{-1}\big(1+V_{(l)}\big(-\Delta_{(l)}+\mu\big)^{-1}\big)^{-1}.
		$$ 
		Next, recall the generalization of geometric series given by 
		$$
		\big(\text{Id}-T\big)^{-1} = \sum_{k=0}^{\infty}T^k,
		$$
		where $T$ is a bounded linear operator and $\text{Id}$ is the identity operator. Using this generalization, we write 
		\begin{equation}\label{gen-geometric}
		\big(L_{+,(l)} + \mu\big)^{-1} = \big(-\Delta_{(l)} + \mu\big)^{-1}\sum_{k=0}^{\infty}\big(-V_{(l)}\big(-\Delta_{(l)}+\mu\big)^{-1}\big)^k.
		\end{equation}
		Since $V_{(l)}$ is bounded, we deduce that the convergence of the above series is guaranteed provided $\mu\gg 1$. Using  \eqref{gen-geometric} along with positivity improving property of $\big(-\Delta_{(l)} + \mu\big)^{-1}$ and $-V_{(l)}$, we conclude that  $\big(L_{+,(l)} + \mu\big)^{-1}$ must also be positivity improving.
		
		Now let $e_{0,(l)}$ be the lowest eigenvalue of $L_{+,(l)}$. Observe that $\big(L_{+,(l)} + \mu\big)^{-1}$ is bounded, self-adjoint and positivity improving with $\big(e_{0,(l)}+\mu\big)^{-1}$ being the 
		largest eigenvalue. Hence, by Theorem XIII.43 part (a) in \cite{RS4}, $\big(e_{0,(l)}+\mu\big)^{-1}$ is simple and the corresponding eigenfunction $\phi_{0,(l)}$ is strictly positive. This completes the proof of Lemma \ref{PFproperty}.
	\end{proof}
\end{lemma}   
Finally, we are ready to prove the following  
 \begin{proposition}\label{null-sp}
	\begin{enumerate}
		\item Null-space of $L_+$ is spanned by the vectors $\partial_{x_j}Q$, $1\leq j\leq 5$.  
		\item $L_-$ is a nonnegative operator and the null-space of $L_-$ is spanned by $Q$.
	\end{enumerate}
\begin{proof}
	\begin{enumerate}
		\item Observe that $Q_{x_j} = Q^\prime(r)\frac{x_j}{r}$. A direct computation yields $L_{+,(1)}Q^\prime = 0$. Thus, we deduce that  $Q^{\prime} < 0$ (from the monotonicity of $Q$) is an eigenfunction of the self-adjoint operator $L_{+,(1)}$ that does not change its sign. In fact, Lemma \ref{PFproperty} implies that $Q^\prime = - \phi_{0,(1)}$, where $\phi_{0,(1)}$ is strictly positive with the corresponding eigenvalue $e_{0,(1)} =0$.    
		
		We now consider a fixed $l\geq 2$ and suppose $e_{0,(l)}$ is the lowest eigenvalue with the associated eigenfunction $\phi_{0,(l)}$ for $L_{+,(l)}$. Now we notice (using \eqref{sLaplacian}, \eqref{sL+exp} and \eqref{Vexp}) that
		\begin{align}\notag
		L_{+,(l)}\phi_{0,(l)} = &\,L_{+,(1)}\phi_{0,(l)} + \frac{l(l+3)-4}{r^2}\phi_{0,(l)} \\\notag
		&+ \frac{16\pi^2}{3}\left(\int_{0}^{\infty}\Big(\frac{1}{6}\,\frac{r_{min}}{r_{max}^4}-\frac{1}{C_l}\,\frac{r_{min}^l}{r_{max}^{l+3}}\Big)Q(s)\,\phi_{0,(l)}(s)\,s^4\,ds\right)Q(r)\\\label{final1}
		>&\,\frac{l(l+3)-4}{r^2}\phi_{0,(l)} \\\label{final2}
		&+ \frac{16\pi^2}{3}\left(\int_{0}^{\infty}\frac{1}{6}\,\frac{r_{min}}{r_{max}^4}\left(1-\frac{6}{C_l}\,\Big(\frac{r_{min}}{r_{max}}\Big)^{l-1}\right)Q(s)\,\phi_{0,(l)}(s)\,s^4\,ds\right)Q(r),
		\end{align} 
		where we have used the fact that $L_{+,(1)}\big|_{\Delta Q^\perp}> 0$. Using the strict positivity of $Q$ along with the strict positivity of $\phi_{0,(l)}$ (from Lemma \ref{PFproperty}), we have that $	\eqref{final2} > 0$ for $l\geq 2$ and $\frac{r_{min}}{r_{max}} < 1$. Using the strict positivity of $\phi_{0,(l)}$ along with the fact that $l\geq 2$, we also have that $\eqref{final1} > 0$. Hence, $L_{+,(l)} > 0$ for $l\geq 2$, and thus, $L_{+,(l)}\phi_{0,(l)} = 0$ has no nonzero $L^2$-solution, which completes the proof of Proposition \ref{null-sp} part (1).
		\item We know from \eqref{L-non-neg} that $L_-$ is nonnegative. Since $L_-Q = 0$ and $Q>0$, we have that $Q$ is the ground state and from the Schr\"odinger operator theory (for example, \cite[Chapter 13]{RS4}), we have that it is the unique positive ground state solution, which completes the proof. 
	\end{enumerate}
\end{proof}\end{proposition}

Now we prove the existence of eigenfunctions for the linearized operator $\mathcal{L}$.
\begin{lemma}\label{eigenfunc}
	The operator $\mathcal{L}$ has two eigenfunctions $\mathcal{Y}_+$, $\mathcal{Y}_-\in\mathcal{S}$ such that 
	\begin{equation}\label{ef}
		\mathcal{L}\mathcal{Y}_+ = e_0\mathcal{Y}_+,\quad \mathcal{L}\mathcal{Y}_- = - e_0\mathcal{Y}_-,
	\end{equation}   
	where $0 < e_0 < \infty$ and $\mathcal{Y}_+ = \overline{\mathcal{Y}}_-$.
\end{lemma}
\begin{proof}
	We know from \eqref{L-non-neg} that the operator $L_-$ on $L^2(\R^5)$ with the domain 
	\[
	D(-\Delta) = \{u\in L^2(\R^5),\ \ -\Delta u\in L^2(\R^5)\}
	\] is nonnegative. Using Kato-Rellich theorem, one sees that $L_-$ is also self-adjoint. Therefore, invoking \cite[Theorem 3.35]{Kato76}, we have that it has a unique square root  
	 with domain $D((-\Delta)^{1/2})$.
	 
	 We already know that $L_+$ has exactly one negative eigenvalue (Claim \ref{-evalue}). This would also imply that the self-adjoint operator $(L_-)^{1/2}L_+(L_-)^{1/2}$ on $L^2(\R^5)$ with domain $D(\Delta^2)$ also has exactly one negative eigenvalue, say $-e_0^2$. Define $\mathcal{Y}_1 =(L_-)^{1/2}\chi$, where $\chi$ is the eigenfunction corresponding to the only negative eigenvalue of $L_+$. Next, define $\mathcal{Y}_2 = \frac{1}{e_0}L_+\mathcal{Y}_1$ and let $\mathcal{Y}_{\pm} = \mathcal{Y}_1 \pm i\mathcal{Y}_2$. Then we have that $\mathcal{L}\mathcal{Y}_{\pm} = \pm e_0\mathcal{Y}_{\pm}$, which shows the existence of the real eigenvalues.      
	 
	 We now prove the decay of the eigenfunctions at infinity, i.e., $Y_{\pm} \in \mathcal{S}(\R^5)$. It is sufficient to prove this for $\mathcal{Y}_1 = \Re\mathcal{Y}_+$. Thus, we claim that there exists $\theta >0$ and constant $c(\theta)>0$ such that $\big|\mathcal{Y}_1(x)\big|\leq c(\theta)e^{-\theta|x|}$. We adapt the proof presented in \cite{FJL07} for pseudorelavistic Hartree equation (also see \cite{KLR09} for the $4$d Hartree equation), which is an adaption of the argument of Slaggie \& Wichmann \cite{SW62} for Schr\"odinger operators. The differential equation to $\mathcal{Y}_1$ is given by $-L_-L_+\mathcal{Y}_1 = e_0^2\mathcal{Y}_1$, which can be written as 
	 \begin{equation}\label{Y1eq}
	 	\mathcal{Y}_1 = \Big( H_0 + e_0^2\Big)^{-1}  \Big(2V_1(-\Delta + 1)\mathcal{Y}_1 - V_1^2\mathcal{Y}_1 -2\nabla V_1\nabla\mathcal{Y}_1 - 4 \nabla( V_2(\mathcal{Y}_1))\nabla Q + 3 Q^2 \mathcal{Y}_1\Big),
	 \end{equation}
	 where 
	 \begin{equation}\label{V1V2}
	 	H_0 \defeq (-\Delta + 1)^2,\quad V_1\defeq\frac{1}{|x|^{3}}\ast Q^2\quad\text{and}\quad V_2(\mathcal{Y}_1) \defeq \Big(\frac{1}{|x|^3}\ast (Q\mathcal{Y}_1)\Big).
	 \end{equation}
	 Observe that $H_0 + e_0^2$ on the Fourier side is $(1+|\xi|^2)^2+e_0^2\approx (1+|\xi|^2)^2$. Then from the theory of Bessel potentials (see \cite[Chapter V]{Stein70}, \cite[Appendix C]{FJL07b}), we know that the kernel, $G(x)$,  associated to  $(H_0 + e_0^2)^{-1}$ (in $5$d) satisfies 
	 \begin{equation}\label{Bkernel}
	 	|G(x)|\leq c_1\frac{e^{-\delta |x|}}{|x|},
	 \end{equation}
	 for some $\delta > 0$. We begin with the aim of showing the following bound on $\mathcal{Y}_1$ 
	 \begin{equation}\label{Y1-bound}
	 	\big|\mathcal{Y}_1(x)\big| \leq h_{\theta}(x)M(x) + Ce^{-\epsilon |x|},
	 \end{equation}
	 where $\theta >0 $ and $h_\theta(x)$ to be chosen later and 
	 \begin{equation}\label{M}
	 	M(x)\defeq \sup_{x'\in\R^5}\big|\mathcal{Y}_1(x')\big|e^{-\theta|x-x'|}.
	 \end{equation}
	 Now we need to bound each term in \eqref{Y1eq}, for which we obtain estimates on $V_1$ and $V_2$. Starting with $V_2(\mathcal{Y}_1)$ term, where $V_2$ is defined in \eqref{V1V2}, we have
	 \begin{equation*}
	 	\nabla V_2(\mathcal{Y}_1)(x')\nabla Q(x') = -3\left(\int_{\R^5}\frac{Q(y)\,\mathcal{Y}_1(y)}{|x'-y|^4}\,dy\right)\nabla Q(x').
	 \end{equation*}  
	 Using the identity $\mathcal{Y}_1(y) = \mathcal{Y}_1(y)\, e^{\theta|y-x|}\,e^{-\theta|y-x|}$, triangle inequality $|y-x| \leq |y-x'| + |x'-x|$ on the term $e^{\theta|y-x|}$ and \eqref{M} along with the fact that  the ground state and its derivatives decay exponentially, we obtain
	 \begin{equation}\label{V2est-pre}
	 	\big|\nabla V_2(\mathcal{Y}_1)(x')\nabla Q(x')\big| \leq c_2\,e^{-\epsilon|x'|+\theta|x-x'|}\,M(x)\int_{\R^5}\frac{e^{\theta|x'-y|}\,e^{-\epsilon|y|}}{|x'-y|^4}\,dy.
	 \end{equation}
	 Evaluating the integral, with $0<\theta < \epsilon$ and splitting the integral into two regions: $|x'-y|<|x'|/2$ and $|x'-y|>|x'|/2$, we observe that in the region $|x'-y|>|x'|/2$, the exponential decay inside of the integral in \eqref{V2est-pre} will be dominant and in the region $|x'-y|<|x'|/2$, we have   
	 \begin{equation*}
	 	\int_{|x'-y|<|x'|/2}\frac{e^{\theta|y|}\,e^{-\epsilon|x'-y|}}{|y|^4}\,dy \leq \frac{8\,e^{\epsilon|x'|}}{|x'|^4}\int_{|x'-y|<|x'|/2}e^{(\theta-\epsilon)|y|}dy\leq c_3(\theta)\,\frac{8\,e^{\epsilon|x'|}}{|x'|^4},
	 \end{equation*}
	 where we have used the fact that $-|x'-y| \leq |x'| - |y|$. Combining this with \eqref{V2est-pre} and choosing sufficiently large $|x'|$, we get  
	 \begin{equation}\label{V2est}
	 	\big|\nabla V_2(\mathcal{Y}_1)(x')\nabla Q(x')\big| \leq c_4\,\frac{e^{\theta|x-x'|}}{1+|x'|^4}\,M(x).  
	 \end{equation} 
	 Similarly, the estimate for $V_1$ follows by integrating
	 $$
	 \int_{\R^5}\frac{e^{-\epsilon|y|}}{|x'-y|^3}\,dy
	 $$  
	 in an analogous fashion. Thus, we have
	 \begin{equation}\label{V1est}
	 	\big|V_1(x')\big|\leq \frac{c_5}{1+|x'|^3}.
	 \end{equation} 
	 Finally, using \eqref{V1est}, we have
	 \begin{equation}\label{V1square}
	 	\big|V_1^2(x')\big| \leq \frac{c_6}{1+|x'|^6}\quad\text{and}\quad \big|\nabla V_1(x')\big| \leq \frac{c_7}{1+|x'|^4}.
	 \end{equation} 
	 Thus, all the terms involving $V_1$ and $V_2$ are bounded and decay at least as 
	 \begin{equation}\label{V1V2-decay}
	 	\frac{1}{1+|x'|^3}.
	 \end{equation} 
	 Choosing $h_\theta(x)$ to be
	 \begin{equation}\label{h}
	 	h_\theta(x) \defeq c_8\int_{\R^5}\frac{e^{-(\delta-\theta)|x-x'|}}{|x-x'|}\,\frac{1}{1+|x'|^3}\,dx',
	 \end{equation}
	 where $c_8$ depends on $c_1$, $c_4$, $c_5$, $c_6$ and $c_7$. Inserting \eqref{V1V2-decay} and \eqref{Bkernel} into the integral equation of \eqref{V1V2} and combining it together with the definitions \eqref{h} and \eqref{M} yields \eqref{Y1-bound}. 
	
	Our next task is to show that the function $h_\theta(x)$ is bounded and decays as $|x|\rightarrow +\infty$. It turns out that the decay of $h_\theta$ controls the decay of the eigenfunction. Applying Young's inequality, we have
	\begin{equation}\label{h-bnd}
		\|h_\theta\|_{L^\infty} \leq c_8\|e^{\theta|x|}G(x)\|_{L^1}\|(1+|x|^3)^{-1}\|_{L^\infty}=c_9(\theta)<\infty,
	\end{equation} 
	where $e^{\theta|\cdot|}G(\cdot)\in L^1$ for $\theta < \delta$ follows from \cite[Proposition 2, Page 132]{Stein70})   and $(1+|x|^3)^{-1}\in L^\infty$, since  $(1+|x|^3)^{-1}<1$. For the decay, we split the integral in \eqref{h} into two regions: $|x-x'| < \kappa$ (inner region) and  $|x-x'| > \kappa$ (outer region). We first evaluate the inner region
	$$
	\int_{|x-x'|<\kappa}^{}\frac{e^{-(\delta-\theta)|x-x'|}}{|x-x'|}\,\frac{1}{1+|x'|^3}\,dx' \leq c_8^\prime \int_{|y|<\kappa}^{}\frac{dy}{|y|(1+|x-y|^3)} \leq \int_{|y|<\kappa}^{}\frac{dy}{|y|(1+\big||x|-|y|\big|^3)}.
	$$
	Writing in polar coordinates, we get
	$$
	\int_{|y|<\kappa}^{}\frac{dy}{|y|(1+\big||x|-|y|\big|^3)} = \frac{8\pi}{3}\,\int_{0}^{\kappa}\frac{r^3\,dr}{(1+\big||x|-r\big|^3)}\leq \frac{8\pi}{3}\,\frac{\kappa^4}{(1+\big||x|-\kappa\big|^3)},
	$$ 
	and in the outer region we use the uniform bound of $(1+|x'|^3)^{-1} < 1$, the fact that $|x-x'|^{-1}<\kappa^{-1}$ and triangle inequality $-|x-x'| \leq |x|-|x'|$ to obtain 
	$$
	\int_{|x-x'|>\kappa}^{}\frac{e^{-(\delta-\theta)|x-x'|}}{|x-x'|}\,\frac{1}{1+|x'|^3}\,dx' \leq \frac{e^{-(\delta-\theta)|x|}}{\kappa}\int_{\R^5}e^{-(\delta-\theta)|x'|}\,dx'=\frac{c_8^{\prime\prime}e^{-(\delta-\theta)|x|}}{\kappa}.
	$$
	Taking $\kappa = |x|^{1/4}$ we have that at least $h_\theta \leq c_{10}|x|^{-1/4}$ as $|x|\rightarrow\infty$ and $\theta \leq \delta$ holds. Because $h_\theta$ decays, there exists a radius $R$, for a fixed small $\theta < \delta$ such that for
	$|x| > R$ we have that $h\theta \leq  c_{10}R^{-1/4}$. Thus, in this region \eqref{Y1-bound} can be written as 
	\begin{equation}\label{Y1-bound-ext}
	\big|\mathcal{Y}_1(x)\big| \leq c_{10}M(x)R^{-1/4} + Ce^{-\epsilon |x|}.
	\end{equation}
	We also observe from the definition of $M(x)$ in \eqref{M} that for any $y\in \R^5$, we have
	\begin{align}\label{M2}\notag
		M(x) = \sup_{x'\in\R^5} \big|\mathcal{Y}_1(x')\big|e^{-\theta|x-x'|} &= \sup_{x'\in\R^5} \big|\mathcal{Y}_1(x')\big|\Big(\sup_{y\in\R^5}e^{-\theta|x-y|-\theta|y-x'|}\Big)\\
		 &= \sup_{y\in\R^5} \big|\mathcal{Y}_1(y)\big|e^{-\theta|x-y|},
	\end{align}
	where we have used the following identity
	\begin{equation}\label{id1}
		e^{-\theta|x-x'|} = \sup_{y\in\R^5}e^{-\theta|x-y|-\theta|y-x'|}.
	\end{equation}
	Using \eqref{Y1-bound-ext} along with \eqref{id1}, we get
	\begin{align}\label{ext1}\notag
		\sup_{|x'|>R}\big|\mathcal{Y}_1(x')\big|e^{-\theta|x-x'|}&\leq c_{10}R^{-1/4}\sup_{|x'|>R}M(x')  e^{-\theta|x-x'|} + C\sup_{|x'|>R} e^{-\epsilon |x'|-\theta|x-x'|}\\\notag
		&\leq c_{10}R^{-1/4}\sup_{x'\in\R^5}M(x')  e^{-\theta|x-x'|} + C\sup_{|x'|>R} e^{-\epsilon |x'|-\theta|x-x'|}\\
		&=c_{10}R^{-1/4}M(x)   + C\sup_{|x'|>R} e^{-\epsilon |x'|-\theta|x-x'|}.
	\end{align} 
	In the complement of this region, we have
	\begin{align}\label{int1}\notag
		\sup_{|x'|\leq R}&\big|\mathcal{Y}_1(x')\big|e^{-\theta|x-x'|}\\
		&\leq \sup_{|x'|\leq R}\Big[h_\theta(x')\big(\sup_{|y|\leq R}\big|\mathcal{Y}_1(y)\big|e^{-\theta|x'-y|}+\sup_{|y|> R}\big|\mathcal{Y}_1(y)\big|e^{-\theta|x'-y|}\big) + Ce^{-\epsilon |x'|}\Big]e^{-\theta|x-x'|}.
	\end{align}
	For the interior term we have by the boundedness of $\mathcal{Y}$ and $e^{\theta|y|}$ for $|y|\leq R$ that $|\mathcal{Y}|\leq c'_{11}(R)$ and $e^{\theta|y|} < c''_{11}(R,\theta)$, therefore, $\sup_{|y|\leq R}\big|\mathcal{Y}_1(y)\big|e^{-\theta|x'-y|}\leq c_11(R,\theta)e^{-\theta|x'|}$. Using this with \eqref{ext1} (for exterior term $|y|>R$) and \eqref{h-bnd}, we re-write \eqref{int1} as
	\begin{align}\label{int2}\notag
		\sup_{|x'|\leq R}&\big|\mathcal{Y}_1(x')\big|e^{-\theta|x-x'|}\leq  C\sup_{|x'|\leq R}e^{-\epsilon |x'|-\theta|x-x'|}\\
		&\leq c_9\sup_{|x'|\leq R}\Big[c_{11}(R,\theta)e^{-\theta|x'|} + c_{10}R^{-1/4}M(x') + C\sup_{|y|> R}e^{-\epsilon|y|-\theta|x-y|} \Big]e^{-\theta|x-x'|}.
	\end{align}  
	Adding \eqref{ext1} and \eqref{int2}, we obtain
	\begin{equation}\label{int+ext}
		M(x) \leq c_{12}\big(R^{-1/4}M(x) + e^{-\min(\epsilon,\theta)|x|}\big) + c'_12(R,\theta)e^{-\theta|x|}.
	\end{equation}  
	Choosing $R=R_1$ to be sufficiently large and $\theta >0 $ sufficiently small, we get
	\begin{equation}\label{M-bnd}
		M(x)\leq C(R_1,\theta)e^{-\theta|x|}.
	\end{equation} 
	Invoking \eqref{M-bnd} along with \eqref{h-bnd} into \eqref{Y1-bound} yields the desired result.
\end{proof}

We then have the following characterization of the real spectrum of $\mathcal{L}$.
\begin{lemma}\label{spectrum}
	Let $\sigma(\mathcal{L})$ be the spectrum of the operator $\mathcal{L}$ on $L^2$ with domain $D(\mathcal{L}) = H^2$. Then $\sigma(\mathcal{L}) \cap \R = \{-e_0, 0, e_0\}.$
\end{lemma} 
\begin{proof}
	For $f\in D(-\Delta)$, denote $V_- = \big(|x|^{-3}\ast Q^2\big)$ and write $V_+= V_- + \big(|x|^{-3}\ast Qf\big)Q$. Decompose $|x|^{-3} = g_1(x) + g_2(x)$ with $g_1\in L^p(|x|\leq 1/\epsilon)$ for $p<5/3$ and $g_2\in L^{\infty}(|x|>1/\epsilon)$. We split the potential into two pieces, the part living inside the ball and the part outside. Observe that $\|g_2\|_{L^\infty}\leq \epsilon$ for some arbitrary $\epsilon>0$. Then, $V_+$ can be decomposed as follows 
	\begin{align*}
		V_+ = \underbrace{\big(g_1\ast Q^2\big) + \big(g_1\ast Qf\big)Q}_{V_{+1}} + \underbrace{\big(g_2\ast Q^2\big) + \big(g_2\ast Qf\big)Q}_{V_{+2}},
	\end{align*}
	where $V_{+1}\in L^q(|x|\leq 1/\epsilon)$ for $q\geq 5/2$ as shown below
	\begin{align*}
		\|V_{+1}\|_{L^q}&\leq \|g_1\ast Q^2\|_{L^q} + \|\big(g_1\ast Qf\big)Q\|_{L^q}\\
		&\leq \|g_1\|_{L^{r_1}}\|Q\|^2_{L^\infty} + \|g_1\|_{L^{r_2}}\|Q\|_{L^\infty}\|f\|_{L^2}\|Q\|_{L^\infty}<\infty
	\end{align*} 
	for $f\in D(-\Delta)$. Here, $1<r_1=\frac{q}{q+1}<5/3$ and $10/9<r_2=\frac{2q}{q+2}<5/3$. Next, we show that $V_{+2}\in L^{\infty}(|x|>1/\epsilon)$ with $\|V_{+2}\|_{L^\infty(|x|>1/\epsilon)}<\varepsilon$ as shown below
	\begin{align*}
		\|V_{+2}\|_{L^\infty}&\leq \|g_2\ast Q^2\|_{L^\infty} + \|\big(g_2\ast Qf\big)Q\|_{L^\infty}\\
		&\leq \|g_2\|_{L^\infty}\|Q\|_{L^2}^2 + \|g_2\|_{L^\infty}\|Qf\|_{L^1}\|Q\|_{L^\infty}\quad \text{(Young's inequality)}\\
		&\leq \epsilon\|Q\|_{L^2}\Big(\|Q\|_{L^2} + \|f\|_{L^2}\|Q\|_{L^\infty}\Big)<2\epsilon\|f\|_{L^2},
	\end{align*}
	last inequality follows form the fact that $Q$ decays exponentially. Choosing $\epsilon<\frac{1}{2\|f\|_{L^2}}$ we get that $ \|V_{+2}\|_{L^\infty}<\varepsilon$ for some $0<\varepsilon<1$. Invoking Theorem XIII. 15 (b) from \cite{RS4}, we have that $\sigma_{\text{ess}}(\mathcal{L}) = [1,\infty)$.
\end{proof}

To this end, we prove that the quadratic form $\Phi(h)$ generated by $L_+$ and $L_-$ is positive-definite. We start by introducing the bilinear form 
\begin{equation}
\label{eq:bilin}
B(f,g) \defeq \frac{1}{2}\int_{\R^5}\big(L_+f_1\big)g_1\,dx + \frac{1}{2}\int_{\R^5}\big(L_-f_2\big)g_2\,dx,
\end{equation}
for $f,g\in H^1(\R^5)$, and define the linearized energy
\begin{equation}\label{eq:linenergy}
\begin{aligned}
\Phi(h) \defeq B(h,h) &= \frac{1}{2}\int_{\R^5}\big(L_+h_1\big)h_1\,dx + \frac{1}{2}\int_{\R^5}\big(L_-h_2\big)h_2\,dx\\
&=\frac{1}{2}\int_{\R^5}|\nabla h|^2\,dx + \frac{1}{2}\int_{\R^5}|h|^2\,dx - \frac{1}{2}\int_{\R^5}\big(|x|^{-3}\ast Q^2\big)|h|^2\,dx \\
&\qquad\qquad\qquad\qquad\qquad\qquad\qquad- \int_{\R^5}\big(|x|^{-3}\ast (Qh_1)\big)Qh_1 \,dx.
\end{aligned}
\end{equation}
We then use Proposition \ref{null-sp} to deduce that
\begin{equation}\label{eq:B-ortho}
B(\partial_{x_j}Q,h) =0; \:\:1\leq j\leq 5,\quad \text{and}\quad B(iQ,h) =0.
\end{equation} 
We also mention the antisymmetric property of $\mathcal{L}$ with respect to the bilinear form $B$
\begin{equation}\label{eq:antisym}
B(f,\mathcal{L}g) = -\frac{1}{2}\int_{\R^5}\big(L_+f_1\big)\big(L_-g_2\big)\,dx + \frac{1}{2}\int_{\R^5}\big(L_-f_2\big)\big(L_+g_1\big)\,dx = - B(\mathcal{L}f,g).
\end{equation}
Next from the Lemma \ref{eigenfunc} and \eqref{eq:linenergy}, we get 
\begin{equation}\label{eq:LE-values}
\Phi(\mathcal{Y}_+) = \Phi(\mathcal{Y}_-) = 0.
\end{equation}
Moreover, from \eqref{eq:pohid} 
\begin{equation}\label{eq:LEatQ}
\begin{aligned}
\Phi(Q) &= \frac{1}{2}\int_{\R^5}|\nabla Q|^2\,dx + \frac{1}{2}\int_{\R^5}|Q|^2\,dx - \frac{3}{2}\int_{\R^5}\big(|x|^{-3}\ast Q^2\big)Q^2\,dx \\
&= \int_{\R^5}\big(|x|^{-3}\ast Q^2\big)Q^2 = -4\int_{\R^5}Q^2\,dx < 0.
\end{aligned}
\end{equation}
Now we specify the orthogonality conditions 
\begin{equation}\label{eq:ortho1}
\int_{\R^5}Qh_2\,dx = \int_{\R^5}\big(\partial_{x_j}Q\big)h_1\,dx=0;\:\:\:1\leq j\leq 5,
\end{equation}
\begin{equation}\label{eq:ortho2}
\int_{\R^5}\Delta Qh_1\,dx =0,
\end{equation}
\begin{equation}\label{eq:ortho3}
\int_{\R^5}\mathcal{Y}_1h_2\,dx = \int_{\R^5}\mathcal{Y}_2h_1\,dx =0.
\end{equation}
Then the sign of $\Phi$ under the above set of orthogonality conditions is given by the following Proposition.

\begin{proposition}\label{positivity}
There exists a constant $c > 0$ such that for all radial functions $h\in G_{\perp}\cup G_{\perp}^y$, we have
		\begin{equation}\label{eq:LE-pos}
		\Phi(h)\geq c\,\|h\|_{H^1(\R^5)}^2,
		\end{equation}
		where 
		\[
		G_{\perp}=\{h\in H^1(\R^5)\::\: h \,\text{satisfies}\, \eqref{eq:ortho1}\, \text{and}\, \eqref{eq:ortho2}\},\] 
		and 
		\[G_{\perp}^y=\{h\in H^1(\R^5)\::\:h\, \text{satisfies}\, \eqref{eq:ortho1}\, \text{and} \,\eqref{eq:ortho3}\}.
		\]
\end{proposition}
\begin{proof} The non-negativity of $\Phi$ on $G_{\perp}$ follows from \eqref{L-non-neg} and \eqref{L+non-neg}. Now to establish the coercivity on $G_{\perp}$, we show that there exists a positive constant $c_*$ such that for any $h\in G_{\perp}$
\begin{equation}\label{eq:coerciveG}
\Phi(h)\geq c_*\,\|h\|_{H^1(\R^5)}^2.
\end{equation} 
We start by showing that there exists a positive constant $c_1$ such that for every $h\in G_{\perp}$ (note that the condition $\big(Q,h_2\big)$ is superfluous)
\begin{equation}\label{eq:L+coercive}
\begin{aligned}
\Phi_1(h_1)\defeq\int_{\R^5}\big(L_+h_1\big)h_1\,dx \geq c_1\|h_1\|_{L^2(\R^5)}^2.
\end{aligned}
\end{equation} 
Assume that \eqref{eq:L+coercive} does not hold, then there exists a sequence of real-valued $H^1$-functions $h_n\in G_{\perp}$ such that
\[
\lim\limits_{n\rightarrow +\infty}\Phi_1(h_n)=\lim\limits_{n\rightarrow +\infty}\int_{\R^5}\big(L_+h_n\big)h_n\,dx =0,\:\:\:\text{and}\:\:\:\|h_n\|_{L^2(\R^5)} = 1.
\] 
This yields 
\begin{equation}\label{eq:non0limit}
0\leq \int_{\R^5}|\nabla h_n|^2\,dx = -1 + \int_{\R^5}\big(|x|^{-3}\ast Q^2\big)h_n^2\,dx + \int_{\R^5}\big(|x|^{-3}\ast (Qh_n)\big)Qh_n\,dx. 
\end{equation}
The function $|x|^{-3}$ can be split into two pieces, the part living inside the unit ball and the part outside as $|x|^{-3} = g_1(x) + g_2(x)$ with $g_1\in L^1$ and $g_2\in L^\infty$, where $g_1 = |x|^{-3}$ and $g_2(x) =0$ inside the unit ball, $g_1(x) =0$ and $g_2(x)=|x|^{-3}$ otherwise.  By Young's inequality and the fact that $Q\in\mathcal{S}(\R^5)$, we have
\[
\int_{\R^5}\big(\big(g_1(x)+g_2(x)\big)\ast Q^2\big)h_n^2\,dx \leq \|g_1\|_{L^1}\|Q\|_{L^\infty}^2\|h_n\|_{L^2}^2 + \|g_2\|_{L^\infty}\|Q\|_{L^2}^2\|h_n\|_{L^2}^2\leq C\|h_n\|_{L^2}^2<\infty,
\]
and
\[
\begin{aligned}
\int_{\R^5}\big(|x|^{-3}\ast (Qh_n)\big)Qh_n\,dx &\leq \|g_1\|_{L^1}\|Qh_n\|_{L^2}^2 + \|g_2\|_{L^\infty}\|Qh_n\|_{L^1}^2\\
&\leq \|g_1\|_{L^1}\|Q\|_{L^\infty}^2\|h_n\|_{L^2}^2 + \|g_2\|_{L^\infty}\|Q\|_{L^2}^2\|h_n\|_{L^2}^2\leq C\|h_n\|_{L^2}^2<\infty.
\end{aligned}
\]
Therefore, $h_n$ is bounded in $H^1$ and thus, there exists $h_* \in H^1$ such that a subsequence of $h_n$ converges
weakly in $H^1$ to $h_*$. By the compactness lemma for radial functions, we have $h_n \rightarrow h_*$ in $L^p(\R^5)$ for $2<p<\frac{10}{3}$. We then have that
\[
\int_{\R^5}\big(|x|^{-3}\ast Q^2\big)\big(h_n^2-h_*^2\big)\,dx \lesssim \|Q\|_{L^{\frac{20}{7}}(\R^5)}^2\|h_n+h_*\|_{L^{\frac{20}{7}}(\R^5)}
\|h_n-h_*\|_{L^{\frac{20}{7}}(\R^5)} \longrightarrow 0
\]
and
\[
\begin{aligned}
\int_{\R^5}&\big(|x|^{-3}\ast (Qh_n)\big)Qh_n\,dx - \int_{\R^5}\big(|x|^{-3}\ast (Qh_*)\big)Qh_*\,dx \\
&= \int_{\R^5}\big(|x|^{-3}\ast (Qh_n)\big)Q(h_n-h_*)\,dx +\int_{\R^5}\big(|x|^{-3}\ast (Q(h_n-h_*))\big)Qh_*\,dx \\
&\lesssim \|Q\|^2_{L^{\frac{20}{7}}(\R^5)}\Big(\|h_n\|_{L^{\frac{20}{7}}(\R^5)}+\|h_*\|_{L^{\frac{20}{7}}(\R^5)}\Big)\|h_n-h_*\|_{L^{\frac{20}{7}}(\R^5)} \longrightarrow 0.
\end{aligned}
\]
Thus, by \eqref{eq:non0limit}, we deduce that $h_*\neq 0$. Furthermore, by the lower semi-continuity property of $H^1$ norm together with the compactness and the non-negativity of $\Phi_1$ on $G_\perp$, we have
\[
0\leq \Phi_1(h_*) \leq \liminf_{n\rightarrow +\infty} \Phi_1(h_n) =0. 
\] 
Therefore, 
\begin{equation}\label{eq:Phi1h*}
\int_{\R^5}\big(L_+h_*\big)h_*\,dx=\Phi_1(h_*)=0,
\end{equation}
and $h_*$ is the solution to the minimization problem
\[
0= \frac{\Phi_1(h_*)}{\|h_*\|_{L^2(\R^5)}} = \min_{h\in \Gamma} \frac{\Phi_1(h)}{\|h\|_{L^2(\R^5)}},\ \ \text{where}\ \ \Gamma= \big\{h\in H^1\ : \ h \ \text{satisfies} \ \eqref{eq:ortho1} \ \text{and} \ \eqref{eq:ortho2}\big\}.
\]
Thus, there exists Lagrange multipliers $\lambda_0,\ldots ,\lambda_6$ such that
\begin{equation}\label{eq:LM}
L_+h_* = \lambda_0\Delta Q +\sum_{j=1}^{5}\lambda_j\partial_{x_j}Q +  \lambda_6h_*.
\end{equation}
Taking the inner product of \eqref{eq:LM} with $h_*$ and using the fact that $h_*\neq 0$ satisfies \eqref{eq:ortho1}, \eqref{eq:ortho2} together with $\langle L_+h_*,h_*\rangle_{L^2} = 0$, we have $\lambda_6 = 0$. Now taking the inner product of \eqref{eq:LM} with $\partial_{x_j}Q$ and using the relations $\langle \partial_{x_i}Q,\partial_{x_j}Q\rangle_{L^2} = 0$ for $i\neq j$, $\langle \Delta Q,\partial_{x_j}Q\rangle_{L^2} = 0$ along with the fact that $\partial_{x_j}Q \in \text{ker} L_+$ for $1\leq j \leq 5$, \eqref{eq:LM} reduces to
\[
0= \lambda_j \int_{\R^5} |\partial_{x_j}Q|^2\,dx,
\]  
which yields that $\lambda_1=\lambda_2=\ldots=\lambda_5=0$. Hence, \eqref{eq:LM} becomes 
\begin{equation}\label{eq:LM1}
L_+h_* = \lambda_0\Delta Q. 
\end{equation}
Denote $\Lambda Q = 2Q + x\cdot\nabla Q$. We now use the relation $L_+\big(\Lambda Q\big) = -2Q$ together with $L_+Q=-2\big(|x|^{-3}\ast Q^2\big)Q$ to deduce that
\begin{equation}\label{eq:LM2}
L_+\left(\frac{\lambda_0}{2}\big(Q-\Lambda Q\big)\right) = \lambda_0\Delta Q = L_+h_*.
\end{equation}
Now by \eqref{eq:LM2} and recalling that $\partial_{x_j}Q \in \text{ker} L_+$ for $1\leq j \leq 5$, we infer that there exist $\mu_1,\dots,\mu_5$ such that
\[
h_* = \frac{\lambda_0}{2}\big(Q-\Lambda Q\big) + \sum_{j=1}^{5}\mu_j\partial_{x_j}Q.
\]
Again, taking the inner product of the previous expression with $\partial_{x_j}Q$ and using the fact that $h_*\in G_{\perp}$ along with the identities $\langle\Lambda Q,\partial_{x_j}Q\rangle_{L^2} =0 = \langle Q,\partial_{x_j}Q\rangle_{L^2}$, $\langle \partial_{x_i}Q,\partial_{x_j}Q\rangle_{L^2} = 0$ for $i\neq j$, we deduce that $\mu_j =0$ for all $j$. Hence,
\begin{equation}\label{eq:min-sol}
h_* = \frac{\lambda_0}{2}\big(Q-\Lambda Q\big)
\end{equation} 
Now by using \eqref{eq:LM1} and \eqref{eq:min-sol}, we compute
\[
\Phi_1(h_*) = \int_{\R^5}\big(L_+h_*\big)h_*\,dx = \frac{\lambda_0^2}{2}\int_{\R^5}\Delta Q\big(Q-\lambda Q\big)\,dx = -\frac{\lambda_0^2}{2}\int_{\R^5}|\nabla Q|^2\,dx.
\]
Invoking \eqref{eq:Phi1h*}, we have $\lambda_0 = 0$, which implies that $h_*=0$, a contradiction. This completes the proof of \eqref{eq:L+coercive}. Now by the definition of $\Phi_1$, we have that for $\varepsilon >0$ small enough 
\[
\varepsilon\Phi_1(h_1) \geq \frac{\varepsilon}{2}\int_{\R^5}|\nabla h_1|^2\,dx - \frac{c_1}{2}\int_{\R^5}|h_1|^2\,dx
\]
for all radial functions $h_1\in H^1(\R^5)$, here $c_1$ is the same constant as in \eqref{eq:L+coercive}. Adding this to \eqref{eq:L+coercive}, we get that for some positive constant $c$,
\[
\Phi_1(h_1) \geq c\|h_1\|_{H^1(\R^5)}^2.
\]
Using an analogous argument we can establish that for some positive constant $c_2$ and a real-valued radial function $h_2\in H^1(\R^5)$ such that $\langle Q, h_2\rangle_{L^2}=0$, we have
\[
\Phi_2(h_2)\defeq\int_{\R^5}\big(L_-h_2\big)h_2\,dx \geq c_2\|h_2\|_{H^1(\R^5)}^2.
\]
We now establish the coercivity of $\Phi$ on $G_{\perp}^y$. We start by showing that $\Phi(h)>0$ for all $h\in G_{\perp}^y\setminus \{0\}$. Assume by contradiction that there exists $\tilde{h}\in G_{\perp}^y\setminus \{0\}$ such that $\Phi(\tilde{h})\leq 0$. From \eqref{eq:B-ortho} and the computation
\begin{equation}\label{eq:coer1}
\begin{aligned}
B\big(\mathcal{Y}_+,\tilde{h}\big) &= \frac{1}{2}\int_{\R^5}\big(L_+\mathcal{Y}_1\big)\tilde{h}_1\,dx + \frac{1}{2}\int_{\R^5}\big(L_-\mathcal{Y}_2\big)\tilde{h}_2\,dx\\
&= \frac{e_0}{2}\int_{\R^5}\mathcal{Y}_2\tilde{h}_1\,dx + \frac{e_0}{2}\int_{\R^5}\mathcal{Y}_1\tilde{h}_2\,dx=0,
\end{aligned}
\end{equation}
we have that $\partial_{x_j}Q$ for $j=1,2,\dots,5$, $iQ$, $\mathcal{Y}_+=\mathcal{Y}_1 + i\mathcal{Y}_2$ and $\tilde{h}$ are orthogonal in the bilinear form $B$. Since $\Phi(\partial_{x_j}Q)=\Phi(iQ)=\Phi(\mathcal{Y}_+)=0$, we get $\Phi(h)\leq 0$ for all $h \in \text{span}\{\partial_{x_j}Q,iQ,\mathcal{Y}_+,\tilde{h}\}$, $j=1,2,\dots,5$. We now claim that these vectors in the span are independent. Indeed, assume that for some real numbers $\alpha_j$, $\beta_1$, $\beta_2$, $\beta_3$, we have
\[
\sum_{j=1}^{5}\alpha_j\partial_{x_j}Q + \beta_1 iQ + \beta_2\mathcal{Y}_+ + \beta_3\tilde{h} = 0,
\]
then $\beta_2 B\big(\mathcal{Y}_+,\mathcal{Y}_-\big) = 0$, and since $B\big(\mathcal{Y}_+,\mathcal{Y}_-\big) = -e_0\big(L_-\mathcal{Y}_2,\mathcal{Y}_2\big) \neq 0$, we have $\beta_2 = 0$. Since $\partial_{x_j}Q$, $iQ$ and $\tilde{h}$ are orthogonal in $L^2$, we also obtain that $\alpha_j=\beta_1=\beta_3 =0.$ This finishes the proof of claim. Since we know that $\Phi$ is positive definite on $G_{\perp}$, hence, it cannot be non-positive on $\text{span}\{\partial_{x_j}Q,iQ,\mathcal{Y}_+,\tilde{h}\}$, $j=1,2,\dots,5$, a contradiction. Therefore, $\Phi(h)$ is positive definite on $G_{\perp}^y\setminus \{0\}.$ 

Finally, to prove the coercivity, we assume by contradiction that there exists a sequence $h_n\in G_{\perp}^y$ such that 
\[
\lim\limits_{n\rightarrow +\infty}\Phi(h_n) = 0,\quad\text{with}\ \ \|h_n\|_{L^2(\R^5} =1.
\] 
Now following a similar argument as used before to prove \eqref{eq:Phi1h*}, we deduce that $h_n$ converges weakly to $h_*$ in $H^1\setminus {0}$. In particular, $h_*\neq 0$ and using compactness together with the lower semi-continuity associated with the weak convergence, we have 
\[
\Phi(h_*) \leq \liminf_{n\rightarrow +\infty} \Phi(h_n) =0,
\] 
which contradicts the definite positivity of $\Phi$ on $G_{\perp}^y\setminus \{0\}.$
\end{proof}

\section{Special solutions}\label{S:special-solutions}

The goal of this section is to demonstrate a richer dynamics by constructing special solutions $Q^+$ and $Q^-$ (see Theorem \ref{mainthm1}). These solutions have the same mass and energy as the ground state and are close to the ground state solution in one time direction, however, exhibit a different behavior in other time direction due to a contrasting gradient size in comparison to the ground state. To accomplish this, we begin with the construction of a family of approximate solutions to \eqref{H}. 
\subsection{Approximate solutions}
\begin{proposition}\label{approx-solns}
	Let $A\in\R$. There exists a sequence $(Z_j^A)_{j\geq 1}$ of functions in $\mathcal{S}(\R^5)$ such that $Z_1^A = A\mathcal{Y}_+$, and if $\mathcal{V}_k^A := \sum_{j=1}^{k}e^{-je_0t}Z_j^A$, then as $t\rightarrow +\infty$ we have
	\begin{equation}\label{approx-solns1}
		\epsilon_k := \partial_t\mathcal{V}_k^A + \mathcal{L}\mathcal{V}_k^A - R(\mathcal{V}_k^A) = O(e^{-(k+1)e_0t}) 
	\end{equation}
	in $\mathcal{S}(\R^5)$.
	\begin{proof}
		We construct the functions $Z_1^A,\ldots,Z_k^A$ by induction on $j$. Note that if $\mathcal{V}_1^A = e^{-e_0t}Z_1^A$, then
		$$
		\epsilon_1 = \partial_t\mathcal{V}_1^A + \mathcal{L}\mathcal{V}_1^A - R(\mathcal{V}_1^A) = -R(\mathcal{V}_1^A),
		$$
		since $\mathcal{L}\mathcal{Y}_+ = e_0\mathcal{Y}_+$.
		From the definition \eqref{Rexp-H} and Lemma \ref{eigenfunc} we have that $R(\mathcal{V}_1^A)=R(e^{-e_0t}Z_1^A)=O(e^{-2e_0t})$, since the leading order in $R$ is quadratic, which yields \eqref{approx-solns1} for $k=1$. Assume now that the solutions $Z_1^A,\ldots,Z_k^A$ are known and satisfy \eqref{approx-solns1}. Note that 
		$$
		\partial_t\mathcal{V}_k^A = \sum_{j=1}^{k}(-je_0)e^{-je_0t}Z_j^A.
		$$
		We can write from \eqref{approx-solns1} that 
		\begin{equation}\label{apps1}
			\epsilon_k(x,t) = \sum_{j=1}^{k}e^{-je_0t}\big((-je_0)Z_j^A(x) + \mathcal{L}Z_j^A(x)\big) - R(\mathcal{V}_k^A(x,t)),
		\end{equation}   
		which says that there exist functions $\mathcal{Z}^A_j\in\mathcal{S}(\R^5)$ (recall that $Z_j^A\in\mathcal{S}(\R^5)$) such that \eqref{apps1} can be written as
		\begin{equation}\label{apps2}
			\epsilon_k(x,t) = \sum_{j=1}^{k+1}e^{-je_0t}\mathcal{Z}_j^A(x) - R(\mathcal{V}_k^A(x,t)).
		\end{equation}		 
		By the decay in time in \eqref{approx-solns1} $\mathcal{Z}_j^A=0$ for all $j\leq k$  and observe from \eqref{Rexp-H} that 
		\begin{align*}
			R(\mathcal{V}_{k}^A) &= i  \int_{\R^5}\frac{|\mathcal{V}_{k}^A(x)|^2Q(y)+2Q(x)\mathcal{V}_{k}^A(x)\mathcal{V}_{k}^A(y)+|\mathcal{V}_{k}^A(x)|^2\mathcal{V}_{k}^A(y)}{|x-y|^3}\,dx \\
			&= \int_{\R^5}\frac{\sum_{j=2}^{3k}e^{-je_0t}f_{jk}(x,y)}{|x-y|^3}\,dx,
		\end{align*}
		where $f_{jk}$'s are in $\mathcal{S}(\R^5)$. Thus, by using the argument in Lemma \ref{eigenfunc} (for the decay of potential terms in \eqref{V1V2}), there exits $\mathcal{U}_{k+1}\in\mathcal{S}(\R^5)$ such that as $t\rightarrow +\infty$, \eqref{apps2} becomes		
		\begin{equation}\label{apps3}
			\epsilon_k(x,t) = -R(\mathcal{V}_{k}^A) - e^{-(k+1)e_0t}\mathcal{U}_{k+1}^A + O(e^{-(k+2)e_0t}).
		\end{equation}
		Note that by Lemma \ref{spectrum}, $(k+1)e_0$ is not in the spectrum of $\mathcal{L}$. Define 
		$$
		Z_{k+1} \defeq (\mathcal{L}-(k+1)e_0)^{-1}\mathcal{U}_{k+1}.
		$$ 
		From Lemma \ref{eigenfunc}, we know that $Z_{k+1}\in\mathcal{S}(\R^5)$. Then we have
		\begin{equation*}
			\epsilon_{k+1} = \partial_t\mathcal{V}_{k+1}^A + \mathcal{L}\mathcal{V}_{k+1}^A - R(\mathcal{V}_{k+1}^A)
		\end{equation*} 
		 Using the definition, $\mathcal{V}_{k+1}^A=\mathcal{V}_{k}^A + e^{-(k+1)e_0t}Z_{k+1}^A$, we have
		 \begin{align*}
		 \epsilon_{k+1} &= \epsilon_k -(k+1)e_0e^{-(k+1)e_0t}Z_{k+1}^A + \mathcal{L}\big(e^{-(k+1)e_0t}Z_{k+1}^A\big) +R(\mathcal{V}_{k}^A)- R(\mathcal{V}_{k+1}^A)\\
		 &=\epsilon_k +e^{-(k+1)e_0t}\mathcal{U}_{k+1}^A +R(\mathcal{V}_{k}^A)- R(\mathcal{V}_{k+1}^A).
		 \end{align*} 
	  By \eqref{apps3}, we get
	  $$
	  \epsilon_{k+1} = - R(\mathcal{V}_{k+1}^A) + O(e^{-(k+2)e_0t}),
	  $$  
	  which implies that $R(\mathcal{V}_{k+1}^A) = O(e^{-(k+2)e_0t})$ in $\mathcal{S}(\R^5)$, thus, yielding \eqref{approx-solns1} for $k+1$.
  	\end{proof}
  	
\end{proposition}

\subsection{Special solutions near an approximate solution}
We now make use of a fixed point argument around an approximate solution constructed in the previous section.
For $f\in L_t^8L_x^{\frac{20}{7}}(I\times\R^5)$ such that $|\nabla|^{\frac{1}{2}} f\in L_t^{\frac{8}{3}}L_x^{\frac{20}{7}}(I\times\R^N)$, we define
		\begin{equation}\label{X-norm}
		\|f\|_{X(I)}\defeq \|f\|_{L_I^{8}L_x^{\frac{20}{7}}} + \||\nabla|^{\frac{1}{2}}f\|_{L_I^{\frac{8}{3}}L_x^{\frac{20}{7}}}
		\end{equation} 
		to present the estimates in a compact manner.
\begin{proposition}\label{sp-solns}
	Let $A\in\R$. There exists $k_0 > 0$ such that for any $k \geq k_0$, there exists $C>0$, $t_k  > 0$ and a solution $U^A$ to \eqref{H} such that for $t\geq t_k$
	\begin{equation}\label{sp-solns1}
		\||\nabla|^{\frac{1}{2}}(U^A - e^{it}Q - e^{it}\mathcal{V}_k^A)\|_{L_t^{\frac{8}{3}}L_x^{\frac{20}{7}}([t,+\infty)\times \R^5)} \leq Ce^{-(k+\frac{1}{2})e_0t}. 
	\end{equation}
	Furthermore, $U^A$ is the unique solution of \eqref{H} satisfying \eqref{sp-solns1} for large $t$. Finally, $U^A$ is independent of $k$ and satisfies 
	\begin{equation}\label{sp-solns2}
		\|U^A - e^{it}Q - Ae^{(i-e_0)t}\mathcal{Y}_+\|_{H^1(\R^5)} \leq Ce^{-2e_0t}, \quad\text{for large}\,\,\,t.
	\end{equation}
	\begin{proof}
		We first observe that the equation \eqref{lineq} can be written as
		\begin{align}\label{eqS}
			ih_t + \Delta h - h &= -S(h)\\\label{Sexp}
			& \defeq -\big[(|x|^{-3}\ast Q^2)h +2(|x|^{-3}\ast(Qh_1))(Q+h) +(|x|^{-3}\ast|h|^2)(Q+h)\big].
		\end{align}
		Define $h=e^{-it}U^A - Q - \mathcal{V}_k^A$ such that $U^A$ is a solution to \eqref{H} if and only if $h$ satisfies
		\begin{align}\label{h-eq}
			ih_t + \Delta h -h = -V(h)- R(h+\mathcal{V}_k^A) + R(\mathcal{V}_k^A) - \epsilon_k,
		\end{align} 
		where 
		\begin{equation}\label{eq:Vexp}
		V(h)\defeq (|x|^{-3}\ast Q^2)h +2(|x|^{-3}\ast(Qh_1))Q,
		\end{equation}
		$R$ is defined in \eqref{Rexp-H} and by Proposition \ref{approx-solns}, as $t\rightarrow+\infty$, we have $\epsilon_k=O(e^{-(k+1)e_0t})\in \mathcal{S}(\R^5)$.  Now showing the existence of $U^A$ (the solution to \eqref{H}) satisfying \eqref{sp-solns1} is equivalent to the fixed-point problem of the integral equation 
		\begin{align*}
		h(t) &= \mathcal{M}(h)(t)\\
		&\defeq -i \int_{t}^{+\infty}e^{i(t-s)(\Delta-1)}\Big[ V(h(s)) + R(h(s)+\mathcal{V}_k^A(s)) - R(\mathcal{V}_k^A(s)) + \epsilon_k(s)\Big]ds
		\end{align*}
		associated to \eqref{h-eq} with 
		$$
		\| h\|_{X([t,+\infty))} \leq e^{-(k+\frac{1}{2})e_0t}
		$$ 
		for $k\geq 1$ and $t\geq t_k$ to be chosen later. We show that $\mathcal{M}$ is a contraction on $B$ defined by
		$$
		B = B(k,t_k) \defeq \big\{h\in E,\,\,\,\,\|h\|_{E} \leq 1\big\},
		$$
		where
		\begin{align*}
			E &= E(k,t_k)\\
			& \defeq \Big\{h \in L_t^{8}L_x^{\frac{20}{7}}([t,+\infty)\times \R^5),\,\,\,|\nabla|^{\frac{1}{2}} h \in L_t^{\frac{8}{3}}L_x^{\frac{20}{7}}([t,+\infty)\times \R^5),\\
			&\qquad\qquad\qquad\qquad\qquad\qquad\qquad \|h\|_{E} = \sup_{t\geq t_k}e^{(k+\frac{1}{2})e_0t}\| h\|_{X([t,+\infty))}  < \infty\Big\}.
		\end{align*}
		Using Sobolev embedding and Strichartz estimate, we have for $h\in E$ and $k$, $t_k$ large such that $t\geq t_k$
		\begin{align}\label{V-est1}
			\|\mathcal{M}(h)\|_{X([t,+\infty))} \leq &\,\,c \,\||\nabla|^{\frac{1}{2}} V(h)\|_{L_t^{\frac{8}{5}}L_x^{\frac{20}{13}}([t,+\infty)\times \R^5)} \\\label{R-est1}
			&+ c\,\||\nabla|^{\frac{1}{2}}(R(h+\mathcal{V}_k^A) - R(\mathcal{V}_k^A))\|_{L_t^{\frac{8}{5}}L_x^{\frac{20}{13}}([t,+\infty)\times \R^5)} \\\label{eps-est}
			&+c \,\||\nabla|^{\frac{1}{2}}\epsilon_k\|_{L_t^{\frac{8}{5}}L_x^{\frac{20}{13}}([t,+\infty)\times \R^5)}.
		\end{align}
		To proceed further, we need the following lemma to estimate the terms on the right-hand side of the above inequality.   
	
				\begin{lemma}\label{R-V-est}
					Let $I$ be a finite time interval and consider $h$, $z$, $w\in L_t^8L_x^{\frac{20}{7}}(I\times\R^5)$ such that $|\nabla|^{\frac{1}{2}} h$, $|\nabla|^{\frac{1}{2}} z$, $|\nabla|^{\frac{1}{2}} w\in L_t^{\frac{8}{3}}L_x^{\frac{20}{7}}(I\times\R^N)$. Then
					\begin{equation}\label{V-est}
					\||\nabla|^{\frac{1}{2}}V(h)\|_{L_I^{\frac{8}{5}}L_x^{\frac{20}{13}}} \lesssim |I|^{\frac{1}{2}}\|h\|_{X(I)},
					\end{equation}  
					and
					\begin{align}\label{R-est}\notag
					\||\nabla|^{\frac{1}{2}}(R(z) &\,\,-\,\, R(w))\|_{L_I^{\frac{8}{5}}L_x^{\frac{20}{13}}}\\
					&\lesssim \Big(|I|^{\frac{1}{8}}\Big(\|z\|_{X(I)}+\|w\|_{X(I)}\Big) +\|z\|^2_{X(I)}+\|w\|^2_{X(I)}\Big)\|z-w\|_{X(I)}
					\end{align}
				\end{lemma}
			\begin{proof}
				Using the H\"older's and Hardy-Littlewood-Sobolev inequalities, we get 
				\begin{align*}
		\||\nabla|^{\frac{1}{2}} V(h)\|_{L_I^{\frac{8}{5}}L_x^{\frac{20}{13}}} \lesssim &\,\||\nabla|^{\frac{1}{2}}(|x|^{-3}\ast Q^2)\|_{L_I^{2}L_x^{\frac{10}{3}}}\|h\|_{L_I^{8}L_x^{\frac{20}{7}}}
				+ \,\,||x|^{-3}\ast Q^2\|_{L_I^{4}L_x^{\frac{10}{3}}}\||\nabla|^{\frac{1}{2}}h\|_{L_I^{\frac{8}{3}}L_x^{\frac{20}{7}}}\\
				+\||\nabla|^{\frac{1}{2}}(&|x|^{-3}\ast (Qh_1))\|_{L_I^{2}L_x^{\frac{10}{3}}}\|Q\|_{L_I^{8}L_x^{\frac{20}{7}}}
				+ \||x|^{-3}\ast (Qh_1)\|_{L_I^{4}L_x^{\frac{10}{3}}}\||\nabla|^{\frac{1}{2}}Q\|_{L_I^{\frac{8}{3}}L_x^{\frac{20}{7}}}\\
				\leq&\,\,  \|Q\|_{L_I^{8}L_x^{\frac{20}{7}}}\||\nabla|^{\frac{1}{2}}Q\|_{L_I^{\frac{8}{3}}L_x^{\frac{20}{7}}}\|h\|_{L_I^{8}L_x^{\frac{20}{7}}}
				+ \|Q\|^2_{L_I^{8}L_x^{\frac{20}{7}}} \||\nabla|^{\frac{1}{2}}h\|_{L_I^{\frac{8}{3}}L_x^{\frac{20}{7}}}\\
				&+ \Big(\|Q\|_{L_I^{8}L_x^{\frac{20}{7}}}\||\nabla|^{\frac{1}{2}}h\|_{L_I^{\frac{8}{3}}L_x^{\frac{20}{7}}}+\||\nabla|^{\frac{1}{2}}Q\|_{L_I^{\frac{8}{3}}L_x^{\frac{20}{7}}}\|h\|_{L_I^{8}L_x^{\frac{20}{7}}}\Big)\|Q\|_{L_I^{8}L_x^{\frac{20}{7}}}\\
				&+ \|Q\|_{L_I^{8}L_x^{\frac{20}{7}}} \|h\|_{L_I^{8}L_x^{\frac{20}{7}}} \||\nabla|^{\frac{1}{2}}Q\|_{L_I^{\frac{8}{3}}L_x^{\frac{20}{7}}}.
				\end{align*}  
				Since $Q$ and its derivatives are in $\mathcal{S}(\R^5)$ and $I$ is finite, we therefore, obtain \eqref{V-est}. Now we estimate the difference \eqref{R-est1}  
							
				\begin{align*}
				\||\nabla|^{\frac{1}{2}}(R(z) &- R(w))\|_{L_I^{\frac{8}{5}}L_x^{\frac{20}{13}}} \\
				\lesssim &\,\,\||\nabla|^{\frac{1}{2}}\big((|x|^{-3}\ast(|z|^2-|w|^2))Q\big)\|_{L_I^{\frac{8}{5}}L_x^{\frac{20}{13}}} + \||\nabla|^{\frac{1}{2}}\big((|x|^{-3}\ast(Q(z-w))z\big)\|_{L_I^{\frac{8}{5}}L_x^{\frac{20}{13}}}\\
				&+ \||\nabla|^{\frac{1}{2}}\big((|x|^{-3}\ast(Q(\bar{z}-\bar{w}))z\big)\|_{L_I^{\frac{8}{5}}L_x^{\frac{20}{13}}} + \||\nabla|^{\frac{1}{2}}\big((|x|^{-3}\ast(Qw))(z-w)\big)\|_{L_I^{\frac{8}{5}}L_x^{\frac{20}{13}}}\\
				&+ \||\nabla|^{\frac{1}{2}}\big((|x|^{-3}\ast(Q\bar{w}))(z-w)\big)\|_{L_I^{\frac{8}{5}}L_x^{\frac{20}{13}}} + \||\nabla|^{\frac{1}{2}}\big((|x|^{-3}\ast(|z|^2-|w|^2))z\big)\|_{L_I^{\frac{8}{5}}L_x^{\frac{20}{13}}} \\
				&+ \||\nabla|^{\frac{1}{2}}\big((|x|^{-3}\ast|w|^2)(z-w)\big)\|_{L_I^{\frac{8}{5}}L_x^{\frac{20}{13}}}\\
				\equiv&\, (I) + (II) + (III) + (IV) + (V) + (VI) + (VII).
				\end{align*}
				We again use H\"older's and Hardy-Littlewood-Sobolev inequalities to estimate (I) and we get,
				\begin{align*}
				(I)
				\lesssim &\,\,\Big(\|z\|_{L_I^{8}L_x^{\frac{20}{7}}} + \|w\|_{L_I^{8}L_x^{\frac{20}{7}}}\Big)\||\nabla|^{\frac{1}{2}}(z-w)\|_{L_I^{\frac{8}{3}}L_x^{\frac{20}{7}}}\|Q\|_{L_I^{8}L_x^{\frac{20}{7}}}\\
				&+\Big(\|z\|_{L_I^{8}L_x^{\frac{20}{7}}} + \|w\|_{L_I^{8}L_x^{\frac{20}{7}}}\Big)\|z-w\|_{L_I^{8}L_x^{\frac{20}{7}}}\||\nabla|^{\frac{1}{2}}Q\|_{L_I^{\frac{8}{3}}L_x^{\frac{20}{7}}}.
				\end{align*}
				Similarly, the estimate for $(II)$ is given by 
				\begin{align*}
				(II) \lesssim &\,\, \|Q\|_{L_I^{8}L_x^{\frac{20}{7}}}\||\nabla|^{\frac{1}{2}}(z-w)\|_{L_I^{\frac{8}{3}}L_x^{\frac{20}{7}}}\|z\|_{L_I^{8}L_x^{\frac{20}{7}}}
				+\||\nabla|^{\frac{1}{2}}Q\|_{L_I^{\frac{8}{3}}L_x^{\frac{20}{7}}}\|z-w\|_{L_I^{8}L_x^{\frac{20}{7}}}\|z\|_{L_I^{8}L_x^{\frac{20}{7}}}\\
				&+\|Q\|_{L_I^{8}L_x^{\frac{20}{7}}}\|z-w\|_{L_I^{8}L_x^{\frac{20}{7}}}  \||\nabla|^{\frac{1}{2}}z\|_{L_I^{\frac{8}{3}}L_x^{\frac{20}{7}}}.
				\end{align*}
				The estimate for $(III)$ is analogous to $(II)$. To estimate $(IV)$, we write
				\begin{align*}
				(IV)\lesssim &\,\,\|Q\|_{L_I^{8}L_x^{\frac{20}{7}}}\||\nabla|^{\frac{1}{2}}w\|_{L_I^{\frac{8}{3}}L_x^{\frac{20}{7}}}\|z-w\|_{L_I^{8}L_x^{\frac{20}{7}}}
				+ \||\nabla|^{\frac{1}{2}}Q\|_{L_I^{\frac{8}{3}}L_x^{\frac{20}{7}}}\|w\|_{L_I^{8}L_x^{\frac{20}{7}}}\|z-w\|_{L_I^{8}L_x^{\frac{20}{7}}}\\
				&+ \|Q\|_{L_I^{8}L_x^{\frac{20}{7}}} \|w\|_{L_I^{8}L_x^{\frac{20}{7}}} \||\nabla|^{\frac{1}{2}}(z-w)\|_{L_I^{\frac{8}{3}}L_x^{\frac{20}{7}}}.
				\end{align*}
				The estimate for $(V)$ is similar to $(IV)$. To estimate $(VI)$, we write
				\begin{align*}
				(VI) \lesssim &\,\,\Big(\|z\|_{L_I^{8}L_x^{\frac{20}{7}}} + \|w\|_{L_I^{8}L_x^{\frac{20}{7}}}\Big)\|z-w\|_{L_I^{8}L_x^{\frac{20}{7}}}\||\nabla|^{\frac{1}{2}}z\|_{L_I^{\frac{8}{3}}L_x^{\frac{20}{7}}}\\
				 	&+ \Big(\|z\|_{L_I^{8}L_x^{\frac{20}{7}}} + \|w\|_{L_I^{8}L_x^{\frac{20}{7}}}\Big)\||\nabla|^{\frac{1}{2}}(z-w)\|_{L_I^{\frac{8}{3}}L_x^{\frac{20}{7}}}\|z\|_{L_I^{8}L_x^{\frac{20}{7}}}. 
				\end{align*}
				Finally, the estimate of $(VII)$ is given by
				\begin{align*}
				(VII)\lesssim \|w\|_{L_I^{8}L_x^{\frac{20}{7}}}\||\nabla|^{\frac{1}{2}}w\|_{L_I^{\frac{8}{3}}L_x^{\frac{20}{7}}}\|z-w\|_{L_I^{8}L_x^{\frac{20}{7}}}
				+ \|w\|^2_{L_I^{8}L_x^{\frac{20}{7}}} \||\nabla|^{\frac{1}{2}}(z-w)\|_{L_I^{\frac{8}{3}}L_x^{\frac{20}{7}}}.
				\end{align*}
Putting all the estimates together along with the properties of $Q$ and the definition of $X(I)$ norm, we obtain \eqref{R-est}. 
			\end{proof}
			To finish the argument, let $0<\tau_0<1$ and use Lemma \ref{R-V-est}, in particular, invoke \eqref{V-est} first together with the definitions of $B=B(k,t_k)$ and $E=E(k,t_k)$ for $t\geq t_k$ to get
			\[
			\||\nabla|^{\frac{1}{2}}V(h)\|_{L_t^{\frac{8}{5}}L_x^{\frac{20}{13}}([t,t+\tau_0]\times\R^5)}\lesssim \tau_0^{\frac{1}{2}}\|h\|_{X([t,t+\tau_0])}\lesssim \tau_0^{\frac{1}{2}}e^{-(k+\frac{1}{2})e_0t}\|h\|_{E} \lesssim \tau_0^{\frac{1}{2}}e^{-(k+\frac{1}{2})e_0t},
			\]
			Summing up this for all the time intervals $[t + j\tau_0, t + (j+1)\tau_0]$ for $j=0,1,2,3,\ldots$, we obtain a geometric series, which yields the estimate for \eqref{V-est1}
			\begin{align}\label{V-est-final}\notag
			\eqref{V-est1} = ||\nabla|^{\frac{1}{2}}V(h)\|_{L_t^{\frac{8}{5}}L_x^{\frac{20}{13}}([t,+\infty)\times\R^5)}&\leq C\tau_0^{\frac{1}{2}}e^{-(k+\frac{1}{2})e_0t}\sum_{j=0}^{\infty} e^{-(k+\frac{1}{2})e_0j\tau_0}\\
			&= \frac{C\tau_0^{\frac{1}{2}}e^{-(k+\frac{1}{2})e_0t}}{1-e^{-(k+\frac{1}{2})e_0\tau_0}}.
			\end{align}
						We now shift our attention to the difference term and employ a similar reasoning via Lemma \ref{R-V-est}, in particular, we use \eqref{R-est}  to write
			\begin{align*}
			\||\nabla|^{\frac{1}{2}}&(R(h+\mathcal{V}_k^A) - R(\mathcal{V}_k^A))\|_{L_t^{\frac{8}{5}}L_x^{\frac{20}{13}}([t,t+\tau_0]\times\R^N)} \\
			&\leq C_1\Big(\|h\|_{X([t,t+\tau_0])}+\|\mathcal{V}_k^A\|_{X([t,t+\tau_0])}+\|h\|^2_{X([t,t+\tau_0])}+\|\mathcal{V}_k^A\|^2_{X([t,t+1])}\Big)\|h\|_{X([t,t+\tau_0])}\\
			&\leq C_2(k)e^{-(k+\frac{3}{2})e_0t},
			\end{align*}
			 where in the last step we have used the explicit form of $\mathcal{V}_k^A$ (see Proposition \ref{approx-solns}) and the fact that $h\in E$. Again, summing up this over all time intervals $[ t + j\tau_0, t+(j+1)\tau_0]$ for $j=0,1,2,3,\ldots$, we obtain the control on \eqref{R-est1} for $t\geq t_k$
			 \begin{align}\label{R-est-final}\notag
			 \eqref{R-est1} = \||\nabla|^{\frac{1}{2}}(R(h+\mathcal{V}_k^A) - R(\mathcal{V}_k^A))\|_{L_t^{\frac{8}{5}}L_x^{\frac{20}{13}}([t,+\infty)\times\R^5)}&\leq C_2(k)e^{-(k+\frac{3}{2})e_0t}\sum_{j=0}^{\infty}e^{-(k+\frac{3}{2})e_0j\tau_0}\\
			 &=\frac{C_ke^{-(k+\frac{3}{2})e_0t}}{1-e^{-(k+\frac{3}{2})e_0\tau_0}}.
			 \end{align}
			 Finally, since $\epsilon_k=O(e^{-(k+1)e_0t})$ for $t\geq t_k$ large enough, we have  
			 \[
			\eqref{eps-est} = \||\nabla|^{\frac{1}{2}}\epsilon_k\|_{L_t^{\frac{8}{5}}L_x^{\frac{20}{13}}([t,+\infty)\times \R^5)} \leq C_3(k) e^{-(k+1)e_0t}.
			 \]
			 Putting the estimates together, we obtain for $t\geq t_k$ with $t_k$ large enough that
			\[
			\eqref{V-est1}+\eqref{R-est1}+\eqref{eps-est} \leq \left(\frac{C\tau_0^{\frac{1}{2}}}{1-e^{-(k+\frac{1}{2})e_0\tau_0}}+\frac{C_2(k)e^{-e_0t_k}}{1-e^{-(k+\frac{3}{2})e_0\tau_0}} + C_3(k)e^{-\frac{e_0t_k}{2}}\right)e^{-(k+\frac{1}{2})e_0t_k}. 
			\]
Choosing $\tau_0$, $k_0$ and a large enough time $t_k$ such that $\tau_0^{1/2} = \frac{1}{8C}$, $e^{-(k_0+\frac{1}{2})e_0\tau_0}\leq \frac{1}{2}$ and $\tfrac{C_2(k_0)e^{-e_0t_k}}{1-e^{-(k_0+\frac{3}{2})e_0\tau_0}}+C_3(k_0)e^{-\tfrac{e_0t_k}{2}}\leq \frac{1}{4}$, we finally get 
			\[
			\|\mathcal{M}(h)\|_{X([t,+\infty))} \lesssim \frac{1}{2},
			\]
for all $k\geq k_0$ and $t\geq t_k$. Therefore, $\mathcal{M}(h)\in E$ and $\mathcal{M}$ maps $B=B(k,t_k)$ into itself. We are left to prove the contraction property of $\mathcal{M}$.  Consider $h$, $g\in B$ and by the definition of $\mathcal{M}$, we have
\begin{align*}
\|\mathcal{M}(h)-\mathcal{M}(g)\|_{E} = \|\mathcal{M}(h)-\mathcal{M}(g)\|_{X([t,+\infty))} \leq &\,\,c\||\nabla|^{\frac{1}{2}}\big(V(h)-V(g)\big)\|_{L_t^{\frac{8}{5}}L_x^{\frac{20}{13}}([t,+\infty)\times\R^5)} \\
+c\||\nabla|^{\frac{1}{2}}\big(R(h&+\mathcal{V}_k^A) - R(g+\mathcal{V}_k^A)\big)\|_{L_t^{\frac{8}{5}}L_x^{\frac{20}{13}}([t,+\infty)\times\R^5)}.
\end{align*}
Using a similar reasoning as above (i.e., we inoke the definition of $V(h)$, \eqref{V-est}, \eqref{V-est-final}, \eqref{R-est} with $z=h+\mathcal{V}_k^A$ and $w=g+\mathcal{V}_k^A$ and \eqref{R-est-final} along with the definitions of $B=B(k,t_k)$ and $E=E(k,t_k)$), we obtain for $t\geq t_k$ with $t_k$ large enough
\[
\|\mathcal{M}(h)-\mathcal{M}(g)\|_{E} \leq\left(\frac{C\tau_0^{\frac{1}{2}}}{1-e^{-(k+\frac{1}{2})e_0\tau_0}}+\frac{C_2(k)e^{-e_0t_k}}{1-e^{-(k+\frac{3}{2})e_0\tau_0}} \right)e^{-(k+\frac{1}{2})e_0t_k} \|h-g\|_{E}.
\]
Again, choosing $\tau_0$, $k_0$ and a large enough time $t_k$ such that $\tau_0^{1/2} = \frac{1}{8C}$, $e^{-(k_0+\frac{1}{2})e_0\tau_0}\leq \frac{1}{2}$ and $\tfrac{C_2(k_0)e^{-e_0t_k}}{1-e^{-(k_0+\frac{3}{2})e_0}}\leq \frac{1}{4}$, shows that $\mathcal{M}$ is a contraction on $B$. Thus, for every $k\geq k_0$, there exists a unique solution $U^A$ to \eqref{H} satisfying \eqref{sp-solns1}. Observe that the uniqueness holds in the class of solutions to \eqref{H} satisfying \eqref{sp-solns1} for larger times, since the fixed point argument still remains valid for $t\geq t_{k_1}\geq t_k$. Therefore, uniqueness of solutions of \eqref{H} (and uniqueness of fixed point) shows that $U^A$ does not depend on $k$. 

Finally, to show \eqref{sp-solns2}, by Strichartz estimate, definition of $\mathcal{M}$ and Lemma \ref{R-V-est}, we have that for a large $k>0$, $t\geq t_k$ large enough and $h\in B$
\begin{align*}
\|h\|_{H^1}&\leq c \,\|\nabla \big(V(h) +(R(h+\mathcal{V}_k^A) - R(\mathcal{V}_k^A))+\epsilon_k\big)\|_{L_t^{\frac{8}{5}}L_x^{\frac{20}{13}}([t,+\infty)\times \R^5)}\leq ce^{-(k+\frac{1}{2})e_0t}.
\end{align*}
This together with the definition that $\mathcal{V}_k^A=Ae^{-e_0t}\mathcal{Y}_++\mathcal{O}\big(e^{-2e_0t}\big)$ from Proposition \ref{approx-solns} inserted into $h=e^{it}U^A-Q-\mathcal{V}_k^A$ yields \eqref{sp-solns2}.	
	\end{proof}
	
\end{proposition}

\section{Modulation}\label{sec:modulat}
In this section, we study the proximity of the solutions of \eqref{H} to the ground state solution. We first state a lemma based on the variational characterization of $Q$ (see \cite[Section 4]{AKAR2}) and the concentration-compactness principle (see \cite[Lemma III.1]{Lions84I}) showing that for any $u\in H^1(\R^5)$ there exists $x_0 \in \R^5$ and $\gamma_0\in\R$ such that $e^{i\gamma_0}u(\cdot + x_0) \rightarrow Q$ strongly in $H^1(\R^5)$, which can be considered for any minimizer). To be more precise, for $u\in H^1(\R^5)$, we define
\begin{equation}\label{alpha}
\alpha(u) = \left|\int|\nabla Q|^2 - \int|\nabla u|^2\right|
\end{equation}
and assume that 
\begin{equation}\label{ME-cond}
	M[u] = M[Q],\qquad E[u] = E[Q].
\end{equation} 
We then have
\begin{lemma}\label{var-char}
Let $u\in H^1(\R^5)$ satisfy \eqref{ME-cond} and $\delta(\alpha(u))$ is small enough. Then there exists parameters $x_0\in\R^5$ and $\gamma_0\in\R$ such that 
\[
\|Q-e^{i\gamma_0}u(\cdot-x_0)\|_{H^1(\R^5)} < \delta(\alpha(u)),
\]
where $\delta(\alpha(u))\rightarrow 0$ as $\alpha(u)\rightarrow 0$.
\end{lemma}
We now choose the modulation parameters $\gamma$ and $X$ through certain orthogonality conditions and use the Implicit Function Theorem to obtain the following orthogonal decomposition of the solution.
\begin{lemma}\label{modulation}
	There exist $\alpha_0 > 0$ and a positive function $\delta(\alpha)$ defined for $0<\alpha\leq \alpha_0$ such that the following is true. For all $u \in H^1(\R^5)$ satisfying \eqref{ME-cond} and $\alpha(u) < \alpha_0$, there exist $(X, \gamma)$ such that $v=e^{-i\gamma}u(x+X)$ satisfies
	\begin{equation}\label{orthocond}
	\Im\int v\,Q =0,\quad \Re\int v\,\partial_{x_k}Q =0 
	\end{equation}
	and
	$$
	\|v-Q\|_{H^1(\R^5)} \leq \delta(\alpha),
	$$
	where $	\delta(\alpha)\rightarrow 0$ as $\alpha\rightarrow 0$. The parameters $(X, \gamma)$ are unique and the mapping $u \mapsto (X, \gamma)$ is $C^1$.
\end{lemma}
\begin{proof}
We define the following functionals on $H^1 \times \R^5 \times \R$ as $(u,X,\gamma) \mapsto \rho^k_{u,X,\gamma}\in\R^6$, $1\leq k\leq 6$:
\[
\rho^k_{u,X,\gamma} \defeq \Re\int e^{-i\gamma}u(x+X)\,\partial_{x_k}Q\,dx,\,\,\,1\leq k\leq 5,
\]
and
\[
\rho^6_{u,X,\gamma} \defeq\Im\int e^{-i\gamma}u(x+X) \,Q\,dx.
\]
We now set up the Jacobian matrix. Computing $\frac{\partial}{\partial X_j}\big(\rho^k_{u,X,\gamma}\big)$ for $1\leq j\leq 5$ at $u=Q$, $X=0$ and $\gamma=0$, we have for $j=k$ that
\[
\frac{\partial\rho^k_{u,X,\gamma}}{\partial X_j}\bigg\vert_{u=Q, \,X=0,\, \gamma=0} = -\|\partial_{x_j}Q\|_{L^2(\R^5)}, 
\]
and if $j\neq k$, we have
\[
\frac{\partial\rho^k_{u,X,\gamma}}{\partial X_j}\bigg\vert_{u=Q, \,X=0,\, \gamma=0} = -\Re\int \,\partial_{x_j}Q\,\partial_{x_k}Q\,dx = 0, 
\]
since $Q$ is radial. Next, we compute $\frac{\partial}{\partial \gamma}\big(\rho^k_{u,X,\gamma}\big)$,
\[
\frac{\partial\rho^k_{u,X,\gamma}}{\partial \gamma} = \Re\int \,-ie^{-i\gamma}u(x+X)\,\partial_{x_k}Q\,dx,
\] 
and thus,
\[
\frac{\partial\rho^k_{u,X,\gamma}}{\partial \gamma}\bigg\vert_{u=Q, \,X=0,\, \gamma=0} = \Im \int \,Q\,\partial_{x_k}Q\,dx =0.
\]
We now compute the derivatives for the functional $\rho^6_{u,X,\gamma}$. We have that
\[
\frac{\partial\rho^6_{u,X,\gamma}}{\partial X_j}\bigg\vert_{u=Q, \,X=0,\, \gamma=0} =\Im\int Q \,\partial_{x_j}Q\,dx = 0,
\]
and
\[
\frac{\partial\rho^6_{u,X,\gamma}}{\partial \gamma}\bigg\vert_{u=Q, \,X=0,\, \gamma=0} =-\int Q^2\,dx = -\|Q\|_{L^2(\R^3)}.
\]
Note that at $u=Q, X=0, \gamma=0$, we have
 $$
 \big(\rho^1_{Q,0,0},\,\rho^2_{Q,0,0},\,\rho^3_{Q,0,0},\,\rho^4_{Q,0,0},\,\rho^5_{Q,0,0},\,\rho^{6}_{Q,0,0}\big)=(0,0,0,0,0,0)
 $$ 
	and the Jacobian matrix is nonzero, thus, we can apply the Implicit Function theorem to obtain the existence of $\delta_0>0$, $\mu_0>0$, a neighborhood $V$ of $(Q,0,0)\in\R^{6}$ and a unique $C^1$ map $(X,\gamma)\,:\,\{u\in H^1\,:\,\|u-Q\|_{H^1}<\delta_0\}\rightarrow V$ such that the orthogonality conditions \eqref{orthocond} are satisfied and $|X| + |\gamma| \leq \mu_0$. 
\end{proof}

Let $U_{\alpha_0} = \{t\in I\,:\,\alpha(u(t)) < \alpha_0\}$, where $I$ is the maximal time interval of existence of $u$. Observe that since one can always assume a smaller $\delta_1$ by the continuity of $u(t)$ with respect to $t$ on a closed time interval before the blow-up time such that $\alpha_0<\delta_1$ for all time $u(t)$ is defined, and therefore, by Lemma \ref{modulation}, there exists functions $X(t)\in\R^5$ and a modified parameter $\theta(t)=\gamma(t)-t\in\R$ (since we want to show that $u$ is close to $e^{it}Q$, up to a constant modulation parameter) for $t\in U_{\alpha_0}$ with the following decomposition
\begin{equation}\label{decomp}
	e^{-i\theta(t)-it}u(x+X(t),t) = (1+\beta(t))Q(x) + h(x,t),
\end{equation}
where 
\begin{equation*}
	\beta(t) = \frac{1}{\|\nabla Q\|_{L^2(\R^5)}^2}\,\Re\left(e^{-i\theta(t)-it}\int\nabla u(x+X(t),t)\cdot\nabla Q(x)\,dx\right) - 1.
\end{equation*}
Here, $\beta(t)$ is chosen such that the orthogonality condition 
\begin{equation}\label{orthocond1}
\Re\int h(x,t)\,\Delta Q(x)\,dx = 0,
\end{equation}
appearing in \eqref{L+non-neg} and in Proposition \ref{positivity}, is satisfied. Moreover, by Lemma \ref{modulation}, $h$ also satisfies the orthogonality conditions \eqref{orthocond}. 

The next two lemmas related to the parameters follows the proof  \cite[Lemma 4.2, 4.3 and 4.4]{DR10} verbatim. We therefore omit the details here. 
\begin{lemma}\label{mod-para}
	Let $u(t)$ be a solution to \eqref{H} satisfying \eqref{ME-cond}.
	Then, taking a smaller $\alpha_0$, if necessary, the following estimate hold for $t\in U_{\alpha_0}$:
$$
|\beta(t)|\approx \left|\int Qh_1(t)\right| \approx  \|h(t)\|_{H^1}\approx \alpha(u(t)).
$$
Furthermore, we have
$$
|\beta^\prime(t)| + |X^\prime(t)| + |\gamma^\prime(t)| \lesssim \alpha(u(t)). 
$$
\end{lemma}

\begin{lemma}\label{mod-exp}
	Let $u$ be a solution to \eqref{H} satisfying $M[u] = M[Q]$ and $E[u] = E[Q]$. Assume that $u$ is defined on $[0,+\infty)$ and that there exists constants $c$, $C > 0$ such that
	$$
	\int_{t}^{+\infty}\alpha(u(s))\,ds\leq Ce^{-ct}.
	$$
	Then there exists $\gamma_0\in\R$ and $x_0\in\R^5$ such that 
	$$
	\|e^{i\gamma_0}u(\cdot+x_0)-e^{it}Q\|_{H^1}\leq Ce^{-ct}.
	$$
\end{lemma}

\section{Covergence to $Q$ in the case $\|u_0\|_{L^2(\R^5)}\|\nabla u_0\|_{L^2(\R^5)} < \| Q\|_{L^2(\R^5)}\|\nabla Q\|_{L^2(\R^5)}$}\label{S:less}

In this section, we prove that the solutions on the same mass and energy level as $Q$ that do not blow-up in finite positive time and that do not scatter forward in time must converge exponentially to $Q$ as $t\rightarrow +\infty$. 
\begin{proposition}\label{blow-up}
	Let $u$ be a solution to \eqref{H} defined on $[0,+\infty)$ (i.e., globally defined for positive times) satisfying
	\begin{equation}\label{above}
		M[u]=M[Q],\quad E[u]=E[Q]\quad\text{and}\quad\|\nabla u_0\|_{L^2(\R^5)} > \|\nabla Q\|_{L^2(\R^5)}.
	\end{equation}
	Assume furthermore that either $|x|u_0\in L^2(\R^5)$ or $u_0$ is radial. 
	Then there exist $x_0\in\R^5$, $\gamma_0\in\R$ and constants $c$, $C > 0$ such that
		\begin{equation}\label{exp-conv2Q}
	\|u-e^{it+i\gamma_0}Q(\cdot+x_0)\|_{H^1(\R^5)}\leq Ce^{-ct}.
	\end{equation}
	Moreover, $u$ blows up for finite negative time of existence.
\end{proposition}

\subsection{Finite variance solutions}\label{subs:finvar}
We start by proving the Proposition \ref{blow-up} for the solutions with the assumption of finite variance, i.e., $|x|u_0\in L^2(\R^5)$.
Define the variance, $ V(t) \defeq \int|x|^2|u(x,t)|^2\,dx$, then 
$$
V^\prime(t) = 4\Im\int (x\cdot\nabla u(x,t))\,\bar{u}(x,t)\,dx.
$$  
Recalling $Z_H$ defined in \eqref{energyH} together with $Z_{H}(Q) = \frac{4}{3}\|\nabla Q\|_{L^2(\R^5)}^2$ (as a consequence of Pohozhaev identities, see \cite[Section 4]{AKAR2}) and $E[u]=E[Q]$, we have
$$
V^{\prime\prime}(t) = 8\,\|\nabla u\|_{L^2(\R^5)}^2 - 6\,Z_{H}(u) = 8\, \big(\|\nabla Q\|_{L^2(\R^5)}^2-\|\nabla u\|_{L^2(\R^5)}^2\big) 
$$
Therefore,
$$
V^{\prime\prime}(t) = -8\,\alpha (u(t)) < 0.
$$
The Proposition \ref{blow-up} is now proved via the following lemma. 
\begin{lemma}\label{4blow-up}
	Let $u(t)$ be a solution to \eqref{H} defined on $[0,+\infty)$ (i.e., globally defined for positive times) satisfying \eqref{above} and $|x|u_0\in L^2(\R^5)$. Then, for all $t$ in the interval of existence of $u$,
	\begin{equation}\label{pos-Vder}
		\Im\int (x\cdot\nabla u(x,t))\,\bar{u}(x,t)\,dx > 0, 
	\end{equation}
	and there exists constants $c$, $C > 0$ such that for all $t\geq 0$
	\begin{equation}\label{exp-control}
		\int_{t}^{+\infty}\alpha(u(s))\,ds\leq Ce^{-ct}.
	\end{equation}
\end{lemma}
\begin{proof}
If \eqref{pos-Vder} does not hold, then there exists $t_1$ such that $V^{\prime}(t_1)\leq 0$. Since $V^{\prime\prime}(t) < 0$, for any $t_0>t_1$, $V^{\prime}(t) <V^{\prime}(t_0) < 0$, for all $t\geq t_0$. This implies that $V(t) < 0$ for large time $t$, yielding a contradiction to the definition of $V(t)$. This proves \eqref{pos-Vder}.

To prove \eqref{exp-control}, we claim that 
\begin{equation}\label{uncertain-pri}
\big(V^{\prime}(t)\big)^2 \leq C\,V(t)\big(V^{\prime\prime}(t)\big)^2,
\end{equation}
which follows from the following lemma that is proved using the argument presented in \cite[Appendix B]{DR10}.
\begin{lemma}
Let $\varphi\in C^1(\R^5)$ and $f\in H^1(\R^5)$. Assume that $V(f) < +\infty$, and $M[f]=M[Q]$, $E[f]=E[Q]$. Then
\begin{equation*}
\left|\int \big(\nabla\varphi \cdot \nabla f\big) \bar{f}\,dx\right|^2 \leq C\,\alpha^2(f)\int |\nabla\varphi|^2|f|^2.
\end{equation*}
\end{lemma}
\begin{proof}
We recall the Gagliardo-Nirenberg inequality \eqref{H-GNC} and the fact that $Q$ is the unique minimizer of Gagliardo-Nirenberg functional \eqref{GN-func}, which for $u = e^{i\lambda \varphi} f$
yields 
\[
Z_H(f)\|\nabla Q\|_{L^2(\R^5)}^3 \leq Z_H(Q)\|\nabla (e^{i\lambda \varphi} f)\|_{L^2(\R^5)}^3,
\]
since $\|e^{i\lambda \varphi} f\|_{L^2(\R^5)} = \|Q\|_{L^2(\R^5)}$. We now rewrite the above inequality (by expanding $\|\nabla (e^{i\lambda \varphi} f)\|_{L^2(\R^5)}^2$ after taking the power of $\frac{2}{3}$) as
\[
\lambda^2\int |\nabla\varphi|^2| f |^2\,dx + 2\lambda\Im\int (\nabla\varphi\cdot\nabla f)\bar{f}\,dx + \|\nabla f\|_{L^2(\R^5)}^2 - \frac{\|\nabla Q\|_{L^2(\R^5)}^2}{(Z_H(Q))^{2/3}} (Z_H(f))^{2/3} \geq 0
\]
for all $\lambda \in \R$. The discriminant of the above quadratic equation in $\lambda$ needs to be negative, which gives
\[
\left|\Im\int (\nabla\varphi\cdot\nabla f)\bar{f}\,dx\right|^2 \leq \left(\int |\nabla\varphi|^2| f |^2\,dx\right)\left(\|\nabla f\|_{L^2(\R^5)}^2 - \frac{\|\nabla Q\|_{L^2(\R^5)}^2}{(Z_H(Q))^{2/3}} (Z_H(f))^{2/3}\right).
\] 
Let $\tilde{\alpha}(f) = \| |\nabla Q\|_{L^2(\R^5)}^2  - \| |\nabla f\|_{L^2(\R^5)}^2$ such that $\alpha(f) = |\tilde{\alpha}(f)|$. Thus, since 
\[
 \| \nabla f\|_{L^2(\R^5)}^2 = \| \nabla Q\|_{L^2(\R^5)}^2 -  \tilde{\alpha}(f),
 \]
 combining this with $E[f] = E[Q]$, we have
 \[
 Z_H(f) = Z_H(Q) - 2\tilde{\alpha}(f).
 \]
 We then write 
 \begin{align*}
 \|\nabla f\|_{L^2(\R^5)}^2 &- \frac{\|\nabla Q\|_{L^2(\R^5)}^2}{(Z_H(Q))^{2/3}} (Z_H(f))^{2/3}\\
 & =  \| \nabla Q\|_{L^2(\R^5)}^2 -  \tilde{\alpha}(f) - \frac{\|\nabla Q\|_{L^2(\R^5)}^2}{(Z_H(Q))^{2/3}} \big(Z_H(Q) - 2\tilde{\alpha}(f)= \big)^{2/3}\\
 & =  \| \nabla Q\|_{L^2(\R^5)}^2 -  \tilde{\alpha}(f) - \|\nabla Q\|_{L^2(\R^5)}^2\left(1 - \frac{4}{3}\frac{\tilde{\alpha}(f)}{Z_H(Q)} + O\big(\alpha^2(f)\big)\right).
 \end{align*}
 Using the relation $Z_H(Q) = \frac{4}{3} \|\nabla Q\|_{L^2(\R^5)}^2$ from \eqref{eq:pohid}, we get
 \[
 \|\nabla f\|_{L^2(\R^5)}^2 - \frac{\|\nabla Q\|_{L^2(\R^5)}^2}{(Z_H(Q))^{2/3}} (Z_H(f))^{2/3} = O\big(\alpha^2(f)\big),
 \]
 which concludes the proof.
\end{proof}
Now taking $\varphi = |x|^2$, we obtain \eqref{uncertain-pri}. Since $V^{\prime}(t)> 0$ and $V^{\prime\prime}(t) < 0$ for all $t\geq 0$, we rewrite \eqref{uncertain-pri} as
\[
\frac{V^{\prime}(t)}{\sqrt{V(t)}} \leq -C\,V^{\prime\prime}(t).
\] 
Integrating from $0$ to $t\geq 0$, we get
\[
\sqrt{V(t)} - \sqrt{V(0)} \leq -C \big(V^{\prime}(t)-V^{\prime}(0)\big)\leq CV^{\prime}(0),
\]
which shows that $V(t)$ is bounded for $t\geq 0$. Using this fact in \eqref{uncertain-pri} implies that
\[
V^{\prime}(t) \leq -CV^{\prime\prime}(t) \implies V^{\prime}(t) \leq Ce^{-ct}.
\]
This together with the observation that
\[
V^{\prime}(t) = -\int_{t}^{+\infty}V^{\prime\prime}(s)\,ds = 8\int_{t}^{+\infty}\alpha(u(s)\,ds,
\]
implies \eqref{exp-control}, which completes the proof of Lemma \ref{4blow-up}.
\end{proof}
Now we are equipped to finish the proof of Proposition \ref{blow-up} in the finite variance case. Assume that $u(x,t)$ is globally defined for negative times, and consider $v(x,t) = u(x,-t)$. Then $v(x,t)$ is a solution of \eqref{H} satisfying Lemma \ref{4blow-up}. Therefore, by \eqref{pos-Vder}, we have
\[
0 < \Im\int (x\cdot\nabla v(x,-t))\,\bar{v}(x,-t)\,dx = - \Im\int (x\cdot\nabla u(x,t))\,\bar{u}(x,t)\,dx, 
\]
for all times $t$ in the domain of existence of $u(x,t)$. This contradicts \eqref{pos-Vder}, which then implies that the negative time of existence of $u(x,t)$ is finite. And, finally \eqref{exp-conv2Q} follows from combining \eqref{exp-control} together with Lemma \ref{mod-exp}. 

\label{subs:radialsoln}\subsection{Radial solutions}
We now prove Proposition \ref{blow-up} for radial solutions, i.e., assume that $u$ is radial satisfying \eqref{above}. The key is to prove that the radial solutions have finite variance, which will then conclude Proposition \ref{blow-up} via the argument presented in the case of finite variance solutions in Section \ref{subs:finvar}. 

We will work with a truncated variance, for which we consider a radial function $\phi$ such that
$$
\phi(r) \geq 0\quad\text{and}\quad\phi^{\prime\prime}(r)\leq 2\quad\text{for}\,\,r\geq 0,
$$
furthermore, 
\[
\phi(r)=\begin{cases}
r^2 &\text{if}\,\,\,0\leq r\leq 1,\\
0 &\text{if}\,\,\,r\geq 2.
\end{cases}
\]
Define $\phi_R=R^2\phi\left(\tfrac{|x|}{R}\right)$ and let 
\begin{equation}\label{eq:locvar}
V_R(t) = \int\phi_R(x)|u(x,t)|^2\,dx
\end{equation}
be the localized variance. Then
\begin{equation}\label{VR-1stDer}
V_R^\prime(t) = 2\,\Im\int u(t)\,\nabla u(t) \cdot \nabla\phi_R\,dx
\end{equation}
and, recalling $Z_H$ from \eqref{H-GNC},
\begin{equation}\label{VR-2ndDer}
	V_{R}^{\prime\prime} = 8\,\|\nabla u\|_{L^2}^2 - 6\,Z_{H}(u) + A_{R}(u(t)) = -8\,\alpha(u(t)) + A_{R}(u(t)),
\end{equation}
where
\begin{align}\label{eq:locvarAR}
	A_{R}(u(t)) = \int_{\R^5}&\,4\Big(\phi_R^{\prime\prime}-2\Big)|\nabla u|^2\,dx - \int_{\R^5}\Delta^2\phi_R|u|^2\,dx\\\notag 
	&+6\int_{\R^5}\int_{\R^5}\left(1-\frac{1}{2}\,\frac{R}{|x|}\phi^\prime\left(\frac{|x|}{R}\right)\right)\frac{x(x-y)|u(x)|^2\,|u(y)|^2}{|x-y|^{5}}\,dx\,dy\\\notag
	&-6\int_{\R^5}\int_{\R^5}\left(1-\frac{1}{2}\,\frac{R}{|y|}\phi^\prime\left(\frac{|y|}{R}\right)\right)\frac{y(x-y)|u(x)|^2\,|u(y)|^2}{|x-y|^{5}}\,dx\,dy.
\end{align}
To establish Proposition \ref{blow-up}, we will also need the following lemma.
\begin{lemma}\label{4blow-up-rad}
Let $u(t)$ be a radial solution to \eqref{H} defined on $[0,+\infty)$ (i.e., globally defined for positive times) satisfying \eqref{above}. Then, there exists $R_1 > 0$ such that for all $R\geq R_1$ and for all $t$ in the interval of existence of $u$,
	\begin{equation}\label{concavity}
	V_R^{\prime\prime}(t) \leq -4\alpha(u(t)),
	\end{equation}
	\begin{equation}\label{pos-Vder-rad}
		V_R^{\prime} (t)> 0. 
	\end{equation}
\end{lemma}
\begin{proof} We start by showing that 
\begin{equation}\label{bound-AR}
A_R(u(t)) \leq  4\alpha(u(t)),
\end{equation}
which will imply \eqref{concavity} via \eqref{VR-2ndDer}.
Using $\phi_R^{\prime\prime} \leq 2$ (by definition), we have that the first term is negative, and again by definition, we have $|-\Delta^2\phi_R|\lesssim 1/R^2$, thus  
$$
- \int\Delta^2\phi_R|u|^2\,dx\leq \frac{c_1}{R^2}\|u\|_{L^2(\R^5)}^2.
$$
Therefore,
\begin{align*}
A_{R}&(u(t)) \leq\,  \frac{c_1}{R^2}\|u\|_{L^2(\R^5)}^2 \\
&+6\int_{\R^5}\int_{\R^5}\left[\left(1-\frac{1}{2}\,\frac{R}{|x|}\phi^\prime\left(\frac{|x|}{R}\right)\right)x-\left(1-\frac{1}{2}\,\frac{R}{|y|}\phi^\prime\left(\frac{|y|}{R}\right)\right)y\right]\frac{(x-y)|u(x)|^2\,|u(y)|^2}{|x-y|^{5}}\,dx\,dy.
\end{align*}
The last term in the above inequality can be estimated by breaking down into the following regions (observe that the integral vanishes in the region $|x|,|y|\leq R$);
\begin{itemize}
	\item Region I: $|x|\approx |y|.$ In this region we have
	$$
	|x|>R,\,\,\,|y|> R.
	$$
	Observe that
	$$
	\left|\left(1-\frac{R}{2|x|}\phi'\left(\frac{|x|}{R}\right)\right)x-\left(1-\frac{R}{2|y|}\phi'\left(\frac{|y|}{R}\right)\right)y\right|\lesssim |x-y|.
	$$
	Thus, using H\"older's, Hardy-Littlewood-Sobolev and radial Sobolev inequalities, we get
	$$
	\int\int\frac{\chi_{\{|y|>R\}}|u(y)|^2}{|x-y|^3}\chi_{\{|x|>R\}}|u(x)|^2\,dx\,dy\lesssim\frac{1}{R^{\frac{12}{5}}}\|\nabla u\|^{\frac{3}{5}}_{L^2(\R^5)}\|u\|^{\frac{17}{5}}_{L^2(\R^5)}.
	$$ 	
	\item Region II: $\max\{|x|,|y|\}\gg\min\{|x|,|y|\}$ and $\max\{|x|,|y|\}>R.$ We consider two cases:
	\begin{itemize}
		\item Case (a):	$|x|\ll|y|\approx |x-y|,\quad |y|>R$ and $|x|<R$. In this case the term becomes
		$$
		\int\int\frac{1}{|x-y|^3}\,\chi_{\{|y|>R\}}|u(y)|^2\,|u(x)|^2\,dx\,dy.
		$$
		Using the triangle inequality and  the definition of $\phi$, we have
		\begin{align*}
		\Big|&\left(1-\frac{R}{2|x|}\phi'\left(\frac{|x|}{R}\right)\right)x-\left(1-\frac{R}{2|y|}\phi'\left(\frac{|y|}{R}\right)\right)y\Big|\\
		&\leq |x|\left(1-\frac{R}{2|x|}\phi'\left(\frac{|x|}{R}\right)\right)+|y|\left(1-\frac{R}{2|y|}\phi'\left(\frac{|y|}{R}\right)\right)\\
		&\lesssim |y|\approx|x-y|,
		\end{align*}
		since $1-\frac{R}{2|x|}\phi'\left(\frac{|x|}{R}\right)<1$ and $1-\frac{R}{2|y|}\phi'\left(\frac{|y|}{R}\right)>\frac{1}{2}$. Again using H\"older's, Hardy-Littlewood-Sobolev and  radial Sobolev inequalities, we bound the above integral by
		$$
		\frac{1}{R^{\frac{12}{5}}}\,\|\nabla u\|^{\frac{3}{5}}_{L^2(\R^5)}\|u\|^{\frac{17}{5}}_{L^2(\R^5)}.
		$$
		\item Case (b): $|y|\ll|x|\approx |x-y|,\quad |x|>R$ and $|y|<R$. This case is symmetric and treated with a similar argument as in Case (a).
	\end{itemize}
\end{itemize}
This gives that 
\begin{equation}\label{pre-bound}
A_{R}(u(t)) \leq \frac{c_2}{R^2} + \frac{c_3}{R^{12/5}}\|\nabla u\|_{L^2(\R^5)}^{3/5}
\end{equation}
by conservation of mass.

We now obtain the estimate on $A_R$ in terms of $\alpha(u(t)$, when $\alpha(u(t))$ is small. First, observe that $\tfrac{d}{dt}\big(\int\,\phi_R(x)|Q(x)|^2\,dx\big) = 0$, and thus, $A_R(Q) = 0$, for all $R>0$. We now invoke Lemma \ref{mod-para} and denote $g=\alpha Q + h$, we then have
\[
u(t) = e^{it}\big(Q +  g(t)\big),\quad  \text{with}\,\,\,\| g(t) \|_{H^1} \lesssim \alpha(u(t)).
\]
Here, we assume that $t\in U_{\alpha_1} = \{t\in I\,\,:\,\,\,\alpha(u(t))<\alpha_1\},$ where $I$ is the maximal time interval of existence of $u(t)$ and $\alpha_1\in (0,\alpha_0)$. We write using the H\"older's, Hardy-Littlewood-Sobolev and Sobolev inequalities,
\begin{align*}
|A_R(u(t))| =&\,\, |A_R(Q+g)| = |A_R(Q+g) - A_R(Q)|\\
\leq&\,\, c_4\int_{|x|\geq R} |\nabla g|^2 + |\nabla Q\cdot \nabla g| + \frac{1}{R^2}\big(|g|^2 + Q|g|\big)\\
&+c_5\iint\limits_{|x|\geq R \,\,\cup\,\, |y|\geq R}\frac{|Q(x)|^2|g(y)|^2 + |g(x)|^2|Q(y)|^2 + |g(x)|^2|g(y)|^2}{|x-y|^3}\\
&+c_5\iint\limits_{|x|\geq R \,\,\cup\,\, |y|\geq R}\frac{|Q(x)|^2Q(y)|g(y)| + |g(x)|^2Q(y)|g(y)| + Q(x)|g(x)||g(y)|^2}{|x-y|^3}\\
\leq&\,\, c_6\Big(\|\nabla g\|^2_{L^2(\R^5)} + e^{-cR}\Big(\|\nabla g\|_{L^2(\R^5)} + \|\nabla g\|^2_{L^2(\R^5)} + \|\nabla g\|^3_{L^2(\R^5)}\Big) + \|\nabla g\|^4_{L^2(\R^5)}\Big).
\end{align*}
Now we invoke Lemma \ref{mod-para} such that for $t\in U_{\alpha_0}$ and $\alpha(u(t))$ small enough, we have
\begin{align*}
|A_R(u(t))|&\leq c_6\Big( (\alpha(u(t)))^2+ e^{-cR}\Big(\alpha(u(t)) + (\alpha(u(t)))^2 + (\alpha(u(t)))^3\Big) + (\alpha(u(t)))^4\Big)\\
&= c_6\Big( \alpha(u(t))+ e^{-cR}\Big(1 + \alpha(u(t)) + (\alpha(u(t)))^2\Big) + (\alpha(u(t)))^3\Big)\alpha(u(t)),
\end{align*}
 which shows that there exist $R_1>0$ such that for a large enough $R\geq R_1$, $e^{-cR_1}<1$ and $\alpha_1>0$ such that for $t\in U_{\alpha_1}$, we obtain the bound \eqref{bound-AR}, as claimed.

On the other hand, if $\alpha(u(t))>\alpha_1$, we consider the function 
\[
f_R(\alpha) \defeq \frac{c_2}{R^2} + \frac{c_3}{R^{12/5}}\Big(\|\nabla Q\|^2_{L^2(\R^5)}+\alpha(u(t))\Big)^{3/10} - 4 \alpha(u(t)).
\]
Observe that $f_R^{\prime\prime}(\alpha) < 0$ for any $\alpha(u(t)) > 0$. Choose $R_2\geq R_1$ such that $f_R^{\prime}(\alpha_1) < 0$ and $f_R(\alpha_1) \leq 0$. This implies that $f_R(\alpha) $ is non-positive for all $\alpha(u(t))>\alpha_1$. Therefore, we obtain \eqref{bound-AR} from \eqref{pre-bound} for $R\geq \max\{R_1,R_2\}$ when $\alpha(u(t))>\alpha_1$. Combining \eqref{bound-AR} with \eqref{VR-2ndDer} yields \eqref{concavity}. Lastly, since $V_R^{\prime\prime}(t)<0$ and $V_R(t)$ is positive, we must have $V_R^{\prime}(t) > 0$, as desired.
\end{proof}

Next we establish the following exponential bound.
\begin{lemma}\label{lem:4bup-exp}
Let $u$ satisfies the assumptions of Proposition \ref{blow-up}, then there exists constants $c,C>0$ such that for $R\geq \max\{R_1,R_2\}$, we have
\begin{equation}\label{eq:exp-bnd}
\int_{t}^{+\infty}\alpha(u(s))\,ds\leq Ce^{-ct},
\end{equation}
for all $t\geq 0$.
\end{lemma}
\begin{proof}
We start by showing that under the hypothesis of Proposition \ref{blow-up}
\begin{equation}\label{eq:expc1}
\big|V_R^{\prime}(t)\big| \leq CR^2\alpha(u(t))
\end{equation}
for any $R>0$ and all $t\geq 0$. 

By the definition of $\alpha(u(t))$ from Section \ref{sec:modulat}, it will be sufficient to prove \eqref{eq:expc1} for $\alpha(u(t)) < \alpha_0$ (see Lemma \ref{mod-para}). By Lemma \ref{mod-para} and denoting $g=\alpha Q + h$, we then have
\[
u(t) = e^{it}\big(Q +  g(t)\big),\quad  \text{with}\,\,\,\| g(t) \|_{H^1} \lesssim \alpha(u(t)).
\]
Therefore, using \eqref{VR-1stDer} and that $|\nabla\phi_R|\leq CR$, we have
\begin{equation*}
\begin{aligned}
\big|\partial_tV_R(t)\big| &= 2\left|\Im \int_{\R^5} (Q+\overline{g(t)})(\nabla Q+\nabla g(t))\cdot\nabla \phi_Rdx\right|\\
&\leq CR^2 \int_{\R^5} \frac{1}{|x|}\big|(Q\nabla g(t)+\overline{g(t)}\nabla Q+\overline{g(t)}\nabla g(t))\big|dx\\
&\leq CR^2 \left(\|\nabla g(t)\|_{L^2(\R^5)}\left( \int_{\R^5} \frac{1}{|x|^2}|Q(x)|^2dx\right)^{1/2} + \|\nabla Q\|_{L^2(\R^5)}\left( \int_{\R^5} \frac{1}{|x|^2}|g(t)|^2dx\right)^{1/2}\right)\\
&\qquad\qquad\qquad\qquad\qquad\qquad+ CR^2\|\nabla g(t)\|_{L^2(\R^5)}\left( \int_{\R^5} \frac{1}{|x|^2}|g(t)|^2dx\right)^{1/2}.
\end{aligned}
\end{equation*}
By the Hardy's inequality and the fact that $Q$ and its derivatives are in $\mathcal{S}(\R^5)$, we obtain
\[
\big|\partial_tV_R(t)\big| \leq C R^2\left(\|\nabla g(t)\|_{L^2(\R^5)} +\|\nabla g(t)\|_{L^2(\R^5)}^2\right),
\]
which yields \eqref{eq:expc1} from Lemma \ref{mod-para} for $t\in U_{\alpha_0}$ and $\alpha(u(t))$ small enough.

Let us fix $R\geq \max\{R_1,R_2\}$. Then from \eqref{concavity} and \eqref{eq:expc1}, we obtain
\[
4\int_t^T \alpha(u(s))\,ds \leq -\int_t^T V^{\prime\prime}_R(t)\,dt = V^{\prime}_R(t) - V^{\prime}_R(T) \leq CR^2\alpha(u(t)).
\]  
Letting $T\rightarrow +\infty$, we have
\[
\int_t^{+\infty} \alpha(u(s))\,ds \leq C\alpha(u(t)),
\] 
for some $C>0$, and therefore, the Gronwall's inequality gives \eqref{eq:exp-bnd}. 
\end{proof}

Now invoking Lemma \ref{mod-exp}, there exists a sequence $t_n\rightarrow+\infty$ such that $\alpha(u(t_n)) \rightarrow 0$. Furthermore, Lemma \ref{var-char} allows us to extract a subsequence (if necessary) such that $u(t_n)\rightarrow e^{i\theta_1}Q$ in $H^1$ for some $\theta_1\in\R$. Since $V_R(t)$ is increasing (see \eqref{pos-Vder-rad}), we get
\[
V_R(0) = \int R^2\phi\left(\frac{|x|}{R}\right)|u_0|^2\,dx \leq \int R^2\phi\left(\frac{|x|}{R}\right)|Q|^2\,dx.
\]
Thus, letting $R\rightarrow +\infty$ yields the finite variance of radial solution. To obtain the blow-up for negative times, we assume that $u(x,t)$ is globally defined for negative times, and consider $v(x,t) = u(x,-t)$. Then $v(x,t)$ is a solution of \eqref{H} satisfying Lemma \ref{4blow-up-rad}. Therefore, by \eqref{pos-Vder-rad}, we have
\[
0 < \Im\int (x\cdot\nabla v(x,-t))\,\bar{v}(x,-t)\,dx = - \Im\int (x\cdot\nabla u(x,t))\,\bar{u}(x,t)\,dx, 
\]
for all times $t$ in the domain of existence of $u(x,t)$. This contradicts \eqref{pos-Vder-rad}, which then implies that the negative time of existence of $u(x,t)$ is finite. And, finally \eqref{exp-conv2Q} follows from combining Lemma \ref{lem:4bup-exp} together with Lemma \ref{mod-exp},  thereby concluding the proof of Proposition \ref{blow-up}. 


\section{Covergence to $Q$ in the case $\|u_0\|_{L^2(\R^5)}\|\nabla u_0\|_{L^2(\R^5)} < \| Q\|_{L^2(\R^5)}\|\nabla Q\|_{L^2(\R^5)}$}\label{S:more}

The main purpose of this section is to establish that if $u(t)$ is a threshold solution with $\|\nabla u_0\|_{L^2(\R^5)} < \|\nabla Q\|_{L^2(\R^5)}$, which does not scatter as $t\rightarrow\infty$, then $u(t)$ converges to the ground state solution $Q$ exponentially as $t\rightarrow\infty$.
\begin{proposition}\label{main-sub}
	Let $u$ be a solution to \eqref{H} satisfying
	\begin{equation}\label{below}
		M[u]=M[Q],\quad E[u]=E[Q],\quad\|\nabla u_0\|_{L^2(\R^5)} < \|\nabla Q\|_{L^2(\R^5)}
	\end{equation}
	and $\|u\|_{S(\dot{H}^{1/2};[0,+\infty))} = +\infty$, i.e., $u$ does not scatter for positive times. Then there exist $x_0\in\R^5$, $\gamma_0\in\R$ and constants $c$, $C > 0$ such that
	$$
	\|u-e^{it+i\gamma_0}Q(\cdot+x_0)\|_{H^1(\R^5)}\leq Ce^{-ct}.
	$$	
\end{proposition}

\subsection{Compactness properties}\label{subs:compact}
The first step is to establish that the solution $u(t)$ to the \eqref{H} satisfying \eqref{below} is compact in $H^1(\R^5)$.
\begin{lemma}\label{compact}
	Let $u$ be a solution of \eqref{H} satisfying the assumptions of Proposition \ref{main-sub}.Then there exists a continuous function $x(t)$ such that
	\begin{equation}
		K \defeq \{u(x+x(t),t);\,\,\,t\in[0,+\infty)\}
	\end{equation}
	has a compact closure in $H^1(\R^5)$.
\end{lemma}
\begin{proof}
	Take a sequence $0\leq t_n\rightarrow +\infty$; we want to show that for every sequence $t_n$ there exists a subsequence $x_n$ such that $u(x+x_n,t_n)$ has a converging subsequence in $H^1(\R^5)$. Applying the linear profile decomposition (\cite[Theorem 6.1]{AKAR1}) to $ u(x,t_n)$ (which is uniformly bounded in $H^1$ and non-scattering), we have
	\begin{equation}\label{crit.eledecomp}
	u_n\defeq u(x,t_n) = \sum_{j=1}^{M}e^{-it_n^j\Delta}\psi^j(x-x^j_n)+W^M_n(x),
	\end{equation}
	 where profiles $\psi^j\in H^1$, $W_n^M$ satisfies 
	 \begin{equation}\label{smallrem}
		\lim\limits_{M\rightarrow\infty}\left(\lim\limits_{n\rightarrow\infty}\|e^{it\Delta}W_n^M\|_{L_t^8L_x^{\frac{20}{7}}}\right)=0,
		\end{equation}
		and sequences $x_n^j$ and $t_n^j$ satisfy
		\begin{equation}\label{pairdivg}
		\lim\limits_{n\rightarrow\infty} |t_n^j-t_n^k|+|x_n^j-x_n^k|=+\infty,
		\end{equation}
		for $1\leq k\neq j\leq M$.
		 The key is to show that there is exactly one nonzero profile, i.e., $M=1$ with $W_n^1\rightarrow 0$ in $H^1$. If more than one $\psi^j \neq 0$, then by the energy Pythagorean expansion (\cite[Proposition 6.2]{AKAR1}) and for $s=0$ in \cite[(6.4)]{AKAR1} there exists an $\varepsilon>0$ such that for all $j$, we have
	\begin{align}\label{eq:ME}
	M[e^{it_n^j\Delta}\psi^j]E[e^{-it_n^j\Delta}\psi^j]&\leq M[Q]E[Q]-\varepsilon,\\\label{eq:MG}
	\|\psi^j\|_{L^2(\R^5)}\|\nabla\psi^j\|_{L^2(\R^5)}&\leq \|Q\|_{L^2(\R^5)}\|\nabla Q\|_{L^2(\R^5)} - \varepsilon.
	\end{align}
	Now recall that by  \cite[Theorem 1.1]{AKAR1}, if a solution $v(t)$ of \eqref{H} with initial data $v_0\in H^1$ satisfies 
\begin{align}\label{eq:MEv}
M[v_0] E[v_0] &< M[Q]E[Q],\\\label{MGv}
\left\| v_0 \right\|_{L^2(\R^5)} \left\| \nabla v_0 \right\|_{L^2(\R^5) }& <  \left\| Q \right\|_{L^2(\R^5)} \left\| \nabla Q \right\|_{L^2(\R^5)}, 
\end{align}
then $v(t)$ scatters in both time directions. We now invoke the existence of wave operators \cite[Lemma 5.4]{AKAR1} for the equation \eqref{H}, which says that there exist a function $v^j_0\in H^1$ such that $v^j(t)$ solving \eqref{H} satisfies
\[
 \|v^j(t_n^j) - e^{-it_n^j\Delta}\psi^j\|_{H^1(\R^5)} \rightarrow 0\quad\text{as}\quad n\rightarrow\infty.
\]
   One can think of $v^j(t)$ as the nonlinear evolution of each separate profile (``bump") and consider a linear sum of nonlinear evolutions of those ``bumps":
	$$
	\widetilde{u}_n(t,x)=\sum_{j=1}^{M}v^j(t-t_n^j,x-x^j_n).
	$$
	Now following similar arguments as in \cite[Theorem 6.3]{AKAR1}, we compare $u_n(t)$ (nonlinear evolution of the entire sum of the bumps) with $\widetilde{u}_n(t)$ (a sum of nonlinear evolutions of each bump) and one can show that only for large $n$, $u_n$ can be approximated by $\widetilde{u}_n$ (for large $M$ and positive times). Therefore, $u_n$ must also scatter (by \cite[Theorem 3.3]{AKAR1}) for positive time, which yields a contradiction, thereby, showing that there is exactly one nonzero profile, i.e., $\psi^1 \neq 0$ and $\psi^j =0$ for all $2\leq j\leq M$, such that
	\begin{align}\label{crit.decomp}
	u_n=e^{-it_n^1\Delta}\psi^1(x-x^1_n)+W^1_n
	\end{align}
	with
	\begin{align}\label{limrem}
	\lim\limits_{n\rightarrow +\infty}\|\widetilde{W}^1_n\|_{H^1}=0.
	\end{align}
	If not, then we are back to the situation, where there exists $\varepsilon>0$ such that for all $j$, we have
	\[
	M[e^{it_n^1\Delta}\psi^1]E[e^{-it_n^1\Delta}\psi^1]\leq M[Q]E[Q]-\varepsilon.	\]
	Therefore, we can conclude by the previous arguments that $u_n$ has finite scattering size.
	Now, we show that $t_n^1$ converges to some finite $t^*$. Since then $e^{-it_n^1\Delta}\psi^1\rightarrow e^{-it^*\Delta}\psi^1$, combining this with \eqref{limrem} and \eqref{crit.decomp} implies that $u_n$ converges in $H^1(\R^5)$. Thus, for each $n$, select $x(t_n^1)=x_n^1\in\R^5$ such that $u_n=u(\cdot-x_n^1,t_n^1)\in K$. Define $x(t)$ to be the continuous function that connects	$x(t_n^1)$ to $x(t_{n+1}^1)$ by a straight line in $\R^N$. Suppose the opposite happens, i.e., $|t_n^1|\rightarrow +\infty$, then we have two cases. First, consider $t_n^1\rightarrow -\infty$. Then applying nonlinear flow to \eqref{crit.decomp}, we obtain
	\begin{align*}
	\|e^{-it\Delta}u_n\|_{S(\dot{H}^{1/2};[0,+\infty))} \leq \|&e^{-i(t-t_n^1)\Delta}\psi^1(x-x^1_n)\|_{S(\dot{H}^{1/2};[0,+\infty))}+\|e^{-it\Delta}W^1_n\|_{S(\dot{H}^{1/2};[0,+\infty))}.
	\end{align*}
	Since $t_n^1 \rightarrow -\infty$, we get
	\begin{align*}
	 \|e^{-i(t-t_n^1)\Delta}\psi^1(x-x^1_n)\|_{S(\dot{H}^{1/2};[0,+\infty))}=\|e^{-it\Delta}\psi^1(x-x^1_n)\|_{S(\dot{H}^{1/2};[-t_n^1,+\infty))} \leq \frac{\delta}{2},
	\end{align*}
	and by Strichartz estimate along with \eqref{limrem}, we get
	$$
	\|e^{-it\Delta}W^1_n\|_{S(\dot{H}^{1/2})} \leq\frac{\delta}{2}
	$$
	for $\delta>0$ (given) for sufficiently large $n$, $M$. Thus, 
	$$
	\|e^{-it\Delta}u_n\|_{S(\dot{H}^{1/2};[0,+\infty))}\leq \delta.
	$$
	Choosing $\delta>0$ sufficiently small, we obtain $\|u_n\|_{S(\dot{H}^{1/2})}\leq 2\delta$ (by the small data theory \cite[Proposition 3.1]{AKAR1}), which is in contradiction with the fact that $u_n$ is non-scattering. With the similar argument, for $n$ large, assuming $t_n^1\rightarrow +\infty$, we obtain
	$$
	\|e^{-it\Delta}u_n\|_{S(\dot{H}^{1/2};(-\infty,0])}\leq \delta,
	$$
	which again yields a contradiction as above. Hence, $t_n^1$ must converge to some finite $t^*$. This finishes the proof of Lemma \ref{compact}.
\end{proof}

\begin{remark}\label{modify-x(t)} 
Let $u$ be a solution of \eqref{H} satisfying \eqref{below}. One may choose $x(t)$ to be continuous on $\R$ and can modify $x(t)$ so that $K$ remains pre-compact in $H^1$ and for all $t\in U_{\alpha_0}$, $x(t) = X(t)$ (see \cite[Corollary 6.3]{DR10}).
\end{remark}
Proof of the following lemma, which establishes the control over parameter $x(t)$, follows verbatim as in the nonlinear Schr\"odinger equation, see \cite[Lemma 6.4]{DR10} (or refer \cite{Thesis} for details). 

\begin{lemma}\label{x(t)-control}
	Let $u$ be as in Proposition \ref{main-sub}, $x(t)$ be continuous as in Remark \ref{modify-x(t)} and set $K$ is pre-compact in $H^1$. Then
	\begin{equation}\label{0-moment}
		P[u] = \Im\int \bar{u}\nabla u\,dx=0.
	\end{equation}
	Moreover,
	\begin{equation}\label{control}
		\frac{x(t)}{t}\rightarrow 0 \,\,\,\text{as}\,\,\, t\rightarrow +\infty.
	\end{equation}
\end{lemma}

\subsection{Proof of Proposition \ref{main-sub}}\label{subs:Prop6.1}
To prove Proposition \ref{main-sub}, we now require the following key pieces. In Section \ref{s-sec:expconv}, we use the localized virial argument that allows us to control the parameter $\alpha(u(t))$ defined in \eqref{alpha}. Then, in Section \ref{s-sec:meanconv}, we first prove that the parameter $\alpha(u(t))$ converges to $0$ in mean, followed by establishing the control on the variations of the parameter $x(t)$. 
\subsubsection{Exponential convergence}\label{s-sec:expconv}
We start with the first key piece and prove that $\alpha(u(t))$ is bounded using the local virial argument together with the modulation estimates from Section \ref{sec:modulat}.
\begin{lemma}\label{control1}
Let $u$ be a solution of \eqref{H} satisfying the assumptions of Proposition \ref{main-sub} and $x(t)$ be as in Remark \ref{modify-x(t)}. Then there exists a constant $c_{10}>0$ such that  if $0 \leq \sigma \leq \tau$ 
\begin{equation}\label{alpha-control}
\int_{\sigma}^{\tau} \alpha(u(t))\,dt \leq c_{10}  \Big[ 1+ \sup_{ t \in [\sigma, \tau]} |x(t)| \Big] \big[ \alpha(u(\sigma))+ \alpha(u(\tau)) \big]   .
\end{equation}
\end{lemma}
\begin{proof}
We again consider the localized variance defined by \eqref{eq:locvar} in section \ref{subs:radialsoln}. Then
\begin{equation}\label{eq:locvar-1Der}
V_R^{\prime}(t) =  2R\,\Im\int \bar{u}\,\nabla u \cdot \nabla\phi\left(\frac{x}{R}\right)\,dx
\end{equation}
and
\begin{equation}\label{eq:locvar-2Der}
	V_{R}^{\prime\prime}(t) =  8\,\alpha(u(t)) + A_{R}(u(t)),
\end{equation}
where $A_R$ is defined by 
\begin{align}\label{eq:locvar-AR}\notag
A_R(u(t)) \defeq  &4 \sum_{j \neq k }\int \frac{\partial^2\phi}{\partial x_j \partial x_k} \left(\frac{x}{R} \right) \frac{ \partial u }{\partial x_j} \frac{\partial \bar{u} }{\partial x_k} + 4 \sum_{j} \int \left( \frac{\partial^2 \phi}{\partial x_j^2 } \left(\frac{x}{R} \right) -2\right) \left| \partial_{x_j} u \right|^2\\
&- \frac{1}{R^2} \int |u|^2 \Delta^2\phi \left(\frac{x}{R} \right)+6\int\int\left(1-\frac{1}{2}\,\frac{R}{|x|}\phi^\prime\left(\frac{|x|}{R}\right)\right)\frac{x(x-y)|u(x)|^2\,|u(y)|^2}{|x-y|^{5}}\\\notag
	&-6\int\int\left(1-\frac{1}{2}\,\frac{R}{|y|}\phi^\prime\left(\frac{|y|}{R}\right)\right)\frac{y(x-y)|u(x)|^2\,|u(y)|^2}{|x-y|^{5}}.
\end{align}
We start with proving that for $\varepsilon >0,$ there exists a constant $R_\varepsilon >0$ such that
\begin{equation}\label{bound-AR}
    \forall t \geq 0 , \quad R \geq R_{\varepsilon} ( 1+ |x(t)|) \Longrightarrow \left| A_R(u(t))\right| \leq \varepsilon \,\alpha(u(t)). 
\end{equation}
First, consider the case when $\alpha(u(t))$ is small, i.e., $\alpha(u(t))<\alpha_0$, where $\alpha_0>0$ as in section \ref{sec:modulat}. Let $g=\alpha Q + h$ as in the proof of Lemma \eqref{4blow-up-rad}, then invoking Lemma \ref{mod-para} and decomposition \eqref{decomp}, we have
\begin{equation}\label{eq:Qg-decomp}
u(x,t) = e^{i(t+\theta(t))}\big(Q (x-X(t))+  g(x-X(t),t)\big)\,\,\, \text{with}\,\,\,\| g(t) \|_{H^1(\R^5)} \lesssim \alpha(u(t)).
\end{equation}
Here, we assume that $t\in D_{\alpha_1}$ with $\alpha_1\in (0,\alpha_0)$. Applying the change of variables $\tilde{x}=x-X(t)$ in \eqref{eq:locvar-AR} and using the fact that $A_R\big(e^{i{\theta_0}+it}Q(\cdot+x_0)\big) =0$ for all $R$ and $t$ with $\theta_0$ and $x_0$ are fixed, we get
\begin{align*}
|A_R(u(t))| &= \big|A_R(u(t)) - A_R(e^{i(t+\theta(t))}Q (x-X(t)))\big|\\
\leq c_7&\,\int_{|\tilde{x}+X(t)|\geq R} |\nabla g(\tilde{x})|^2 + |\nabla Q(\tilde{x})\cdot \nabla g(\tilde{x})| + \frac{1}{R^2}\big(|g(\tilde{x})|^2 + Q(\tilde{x})|g(\tilde{x})|\big)\\
+\,c_7&\iint\limits_{|\tilde{x}+X(t)|\geq R \,\,\cup\,\, |\tilde{y}+Y(t)|\geq R}\frac{|Q(\tilde{x})|^2|g(\tilde{y})|^2 + |g(\tilde{x})|^2|Q(\tilde{y})|^2 + |g(\tilde{x})|^2|g(\tilde{y})|^2}{|\tilde{x}-\tilde{y}|^3}\\
+\,c_7&\iint\limits_{|\tilde{x}+X(t)|\geq R \,\,\cup\,\, |\tilde{y}+Y(t)|\geq R}\frac{|Q(\tilde{x})|^2Q(\tilde{y})|g(\tilde{y})| + |g(\tilde{x})|^2Q(\tilde{y})|g(\tilde{y})| + Q(\tilde{x})|g(\tilde{x})||g(\tilde{y})|^2}{|\tilde{x}-\tilde{y}|^3}\\
\leq c_8&\,\,\Big(\|\nabla g\|^2_{L^2(\R^5)} + e^{-c|\tilde{x}|}\Big(\|\nabla g\|_{L^2(\R^5)} + \|\nabla g\|^2_{L^2(\R^5)} + \|\nabla g\|^3_{L^2(\R^5)}\Big) + \|\nabla g\|^4_{L^2(\R^5)}\Big).
\end{align*}
Now using $\| g \|_{H^1} \lesssim \alpha(u(t)) $ yields that for any $R \geq R_3+ |X(t)|$, we have
\begin{align*}
      \left| A_R(u(t)) \right|  &\leq c_9 \Big( (\alpha(u(t)))^2+ e^{-R_3}\Big(\alpha(u(t)) + (\alpha(u(t)))^2 + (\alpha(u(t)))^3\Big) + (\alpha(u(t)))^4\Big)\\ 
    &=c_9 \Big( \alpha(u(t))+ e^{-R_3}\Big(1 + \alpha(u(t)) + (\alpha(u(t)))^2\Big) + (\alpha(u(t)))^3\Big)\alpha(u(t)) \\ 
    &\leq \varepsilon \alpha(u(t)),
\end{align*}
provided $R_3>0$ is such that $c_9e^{-R_3}\leq \frac{\varepsilon}{2} $ and $\alpha_1$ is such that $c_9\big(\alpha_1 + e^{-R_3}\big(\alpha_1 + \alpha_1^2\big) +\alpha_1^3\big) \leq \frac \varepsilon 2$. Since $\alpha_1\in (0,\alpha_0)$ and $x(t)=X(t)$ for all $t \in U_{\alpha_0}$ (by Remark \ref{modify-x(t)}), we obtain \eqref{bound-AR} for $\alpha(u(t)) < \alpha_1$.
 
 For the case $\alpha(u(t))\geq \alpha_1$, by \eqref{eq:locvar-AR}, we have 
\begin{align*}
    \left| A_R(u(t))\right| \lesssim &\int_{|x-x(t)|\geq R-|x(t)|} | \nabla u (t)|^2 + |u(t)|^2 dx\\
     +6&\iint\limits_{\Omega}\left[\left(1-\frac{1}{2}\,\frac{R}{|x|}\phi^\prime\left(\frac{|x|}{R}\right)\right)x-\left(1-\frac{1}{2}\,\frac{R}{|y|}\phi^\prime\left(\frac{|y|}{R}\right)\right)y\right]\frac{(x-y)|u(x)|^2\,|u(y)|^2}{|x-y|^{5}},
\end{align*}
where $\Omega = \{|x-x(t)|\geq R-x(t) \,\,\cup\,\, |y-y(t)|\geq R-y(t)\}.$
Note that in the region $\Omega$, we have
\begin{equation}\label{omReg1}
	\left|\left(1-\frac{R}{2|x|}\phi'\left(\frac{|x|}{R}\right)\right)x-\left(1-\frac{R}{2|y|}\phi'\left(\frac{|y|}{R}\right)\right)y\right|\lesssim |x-y|
	\end{equation}
	Thus, by H\"older's, Hardy-Littlewood-Sobolev and Sobolev inequalities, we get
	\begin{equation}\label{omReg1est}
	\iint\limits_{}\frac{\chi_{|y|>R}|u(y)|^2}{|x-y|^3}\chi_{|x|>R}|u(x)|^2\,dx\,dy\lesssim\|u\|^{4}_{L^{\frac{20}{7}}(\R^5)} \lesssim \|\nabla u\|^{4}_{L^2(\R^5)}.
	\end{equation}
	 Now by the compactness of $K$ (see Remark \ref{modify-x(t)}), there exists $R_4>0$ such that for $R \geq |x(t)|+R_4$ and $\alpha(u(t))\geq \alpha_1$ 
\begin{equation}\label{alpha-big}
     \left| A_R(u(t)) \right| \leq  \varepsilon \alpha_1 \leq \varepsilon \alpha(u(t)).
\end{equation}
This completes the proof of \eqref{bound-AR}. 

To conclude the proof of Lemma \ref{control1}, we now invoke \eqref{bound-AR} along with the \eqref{eq:locvar-2Der} and get that there exists $R_5>0$ such that for $R \geq R_5(1+|x(t)|)$, we have
\begin{align*}
     \left| V_R^{\prime\prime}(t) \right| \geq 4 \alpha(u(t)).
\end{align*}
Let $R=R_5(1+\sup_{\sigma \leq t \leq \tau } |x(t)|).$ Then
\begin{align}
\label{alpha-bdd}
    2 \int_{\sigma}^{\tau} \alpha(u(t)) \,dt \leq \int_{\sigma}^{\tau} V^{''}_R(t) \,  dt \leq V^{'}_R(\tau)-V^{'}_R(\sigma). 
\end{align}
Again, considering two cases: for $\alpha(u(t))<\alpha_0,$ we again use the decomposition \eqref{eq:Qg-decomp} along with change of the variable $y=x-X(t)$ and write
\begin{align*}
    V^{\prime}_R(t)=&\,\,2 R \Im \left(\int \bar{g}(y)\; \nabla \phi\left(\frac{y+X(t)}{R}\right) \cdot \nabla Q(y)dy +  \int Q(y)  \nabla\phi \left( \frac{y+X(t)}{R} \right) \cdot \nabla g(y) \, dy\right) \\ 
   & + 2R \Im \int \bar{g}(y) \nabla\phi \left( \frac{y+X(t)}{R} \right) \cdot  \nabla g(y) dy ,
\end{align*}
which yields 
\[ 
 \big| V^{'}_R(t) \big| \lesssim R\big(e^{-|y|}\alpha(u(t))+\big(\alpha(u(t))\big)^2\big) \leq c_{10}R \,\alpha(u(t)). 
 \]
This inequality is also valid for $\delta(t) \geq \delta_0,$ because we can always choose a bigger $\alpha(u(t))$ to satisfy the estimate (see \eqref{alpha-big}). Going back to \eqref{alpha-bdd}, we have 
\begin{align*}
\int_{\sigma}^{\tau} \alpha(u(t))\,dt &\leq c_{10}\, R \big(\alpha(u(\sigma)) + \alpha(u(\tau))\big) \\ &\leq c_{10} \, R_5\left(1+\sup_{\sigma \leq t \leq \tau } |x(t)| \right) \big[\alpha(u(\sigma)) + \alpha(u(\tau))\big].
\end{align*}
This concludes the proof of Lemma \ref{control1}.
\end{proof}
\subsubsection{Convergence in mean}\label{s-sec:meanconv}
The following convergence result is the key to proving that the parameter $x(t)$ is bounded, which is the essential ingredient in proving Proposition \ref{main-sub}.
\begin{lemma}\label{control2}
Suppose $u(t)$ is a solution of \eqref{H} satisfying assumptions of  Proposition \ref{main-sub}. Then 
\begin{equation}
    \label{alpha-conv}
        \lim_{T \rightarrow + \infty} \frac 1T \int_{0}^T \alpha(u(t))\, dt=0.
    \end{equation}
\end{lemma}
\begin{proof}
We start with recalling the localized variance and its second derivative from \eqref{eq:locvar-2Der}
\[
V_{R}^{\prime\prime}(t) =  8\,\alpha(u(t)) + A_{R}(u(t)),
\]
where $A_R$ is defined in \eqref{eq:locvar-AR}. Observe that by the definition of $\phi_R$, we have that $A_R(u(t)) = 0$ for $|x|\leq R$. Therefore,
\begin{align}\label{AR-bound4x>R}\notag
  \big| A_R(u(t)) \big| &\lesssim \int_{|x|\geq R} \left| \nabla u \right|^2+ \frac{1}{R^2} |u|^2\\
  +6&\iint\limits_{\Omega}\left[\left(1-\frac{1}{2}\,\frac{R}{|x|}\phi^\prime\left(\frac{|x|}{R}\right)\right)x-\left(1-\frac{1}{2}\,\frac{R}{|y|}\phi^\prime\left(\frac{|y|}{R}\right)\right)y\right]\frac{(x-y)|u(x)|^2\,|u(y)|^2}{|x-y|^{5}},
\end{align}
where $\Omega = \{|x|\geq R \,\,\cup\,\, |y|\geq R\}.$
We now appeal to \eqref{control} and choose $T_0(\varepsilon)>0$ with $\varepsilon > 0$ fixed such that for all $t \geq T_0(\varepsilon)$
\[ 
|x(t)| \leq \varepsilon t .
\]
For $T\geq T_0(\varepsilon),$ select $R=\varepsilon T+ R_6(\varepsilon)+1.$
 This selection of $R$ implies $R_6(\varepsilon)+\varepsilon T \leq R,$ which then gives $R\leq |x| = |x-x(t)|+|x(t)| \leq |x-x(t)|+\varepsilon T \Longrightarrow |x-x(t)| \geq R-\varepsilon T \geq R_6(\varepsilon)$ (analogously, $|y-y(t)|\geq R_7(\varepsilon)$), thus 
\begin{align}\label{AR-bound4x>R-fin}\notag
    \big| A_R(u(t)) \big| &\lesssim  \int_{|x-x(t)|\geq R_8(\varepsilon)} \left| \nabla u \right|^2+ \frac{1}{R^2} |u|^2\\\notag
  +6&\iint\limits_{\Omega(t)}\left[\left(1-\frac{1}{2}\,\frac{R}{|x|}\phi^\prime\left(\frac{|x|}{R}\right)\right)x-\left(1-\frac{1}{2}\,\frac{R}{|y|}\phi^\prime\left(\frac{|y|}{R}\right)\right)y\right]\frac{(x-y)|u(x)|^2\,|u(y)|^2}{|x-y|^{5}}\\
 & \lesssim \varepsilon,
\end{align}
where $\Omega(t) = \{|x-x(t)|\geq R_8(\varepsilon) \,\,\cup\,\, |y-y(t)|\geq R_8(\varepsilon)\}$ and $R_8=\max\{R_6,R_7\}$. The last inequality follows from the compactness of $K$, Lemma \ref{compact} and Remark \ref{modify-x(t)}.

Now recalling the first derivative of localized variance from \eqref{eq:locvar-1Der} and the fact that $\big|V^{'}_R(t)\big|\lesssim R$, we have
\[  \int_{T_0(\varepsilon)}^T V^{''}_{R}(t)\, dt \leq \big| V^{'}_{R}(T) \big|+\big| V^{'}_{R}(T_0(\varepsilon)) \big| \lesssim R .
\]
On the other hand, using the expression of second derivative of localized variance \eqref{eq:locvar-2Der} along with \eqref{AR-bound4x>R} and \eqref{AR-bound4x>R-fin}, we have
\begin{equation*}
    \int_{T_0(\varepsilon)}^T \alpha(u(t))\, dt \lesssim R+\varepsilon(T-T_0(\varepsilon)) \lesssim  \varepsilon T+R_8(\varepsilon)+1,
\end{equation*}
where the (hidden) constant is independent of $T$ and $\varepsilon.$ Finally, we compute
\begin{align*}
    \frac 1T \int_{0}^T \alpha(u(t)) \,dt &= \frac 1T \int_{0}^{T_0(\varepsilon)} \alpha(u(t)) \, dt + \int_{T_0(\varepsilon)}^T \alpha(u(t))\, dt\\
    &\lesssim \frac 1T \int_{0}^{T_0(\varepsilon)} \alpha(u(t)) \, dt + \frac 1T (R_8(\varepsilon)+1)+ \varepsilon.
\end{align*}
Letting $T \rightarrow + \infty,$ we obtain 
\[
\limsup_{T \rightarrow + \infty} \frac 1T \int_{0}^T \alpha(u(t))\, dt \lesssim \varepsilon.
\]
Now taking $\varepsilon \rightarrow 0$ (since $\varepsilon$ is arbitrary), we obtain \eqref{alpha-conv}, as desired.
\end{proof}

The following lemma provides the control on the variations of translation parameter. We remark that the proof is similar to the one in \cite[Lemma 6.8]{DR10}.
\begin{lemma}\label{control3}
There exists a constant $c_{11}>0$ such that 
\begin{equation}\label{x-var-control}
\forall \sigma,\tau>0\quad \text{with} \quad \sigma+1 \leq \tau, \quad |x(\tau)-x(\sigma)| \leq c_{11} \int_{\sigma}^{\tau} \alpha(u(t))\, dt  .
\end{equation}
\end{lemma}
\begin{proof}
Let $\alpha_0>0$ be as in Lemma \ref{modulation}. We start with proving that there exists $\alpha_2>0$ such that for all $\tau\geq 0$,
\begin{equation}\label{alpha-inf-sup}
     \inf_{t\in[\tau,\tau+2]} \alpha(u(t)) \geq \alpha_2 \qquad \text{ or } \qquad \sup_{t\in[\tau,\tau+2]} \alpha(u(t)) < \alpha_0.
\end{equation}
If not, there exist $t_n,t_n'\geq 0$ such that 
\begin{equation}\label{contrad-inf-sup}
    \alpha(u(t_n)) \xrightarrow[n \rightarrow +\infty]{} 0 ,\qquad \alpha(u(t_n'))\geq \alpha_0, \qquad |t_n-t_n'|\leq 2.
\end{equation}
If necessary, we may assume (extracting a subsequence)
\begin{equation}
\label{def-tau}
    \lim_{n \rightarrow +\infty} t_n-t_n'=\tau \in [-2,2].
\end{equation}
By the compactness of $K,$ we have  
\[
u(t_n', \cdot + x(t_n'))  \xrightarrow[n \longrightarrow + \infty ]{} v_0  \in H^1(\R^5).
\]
Note that if $t'_n$ goes to $+ \infty,$ then $\left| x(t_n')\right| $ converges (after extraction) to a limit $x_0\in \R^5$ by Lemma \ref{x(t)-control}. 
Introducing $w_0(x)=v_0(x+x_0),$ then we have 
\[
u(t_n', \cdot+x(t_n^{'})) \xrightarrow[n \longrightarrow + \infty ]{}    w_0(\cdot-x_0)     \in H^1(\R^5).
\]
In other words, we obtain
\begin{equation}\label{compact-conv}
u(t_n') \xrightarrow[n \longrightarrow + \infty ]{}  w_0 \in  H^1(\R^5).
\end{equation}
We now appeal to the second hypothesis in \eqref{contrad-inf-sup}$,\alpha(u(t_n'))= \| \nabla Q \|_{L^2(\R^5)}^2 - \|\nabla u(t_n',\cdot+x(t_n'))\|_{L^2(\R^5)}^2 \geq \alpha_0 >0,$ which implies that
\[
    \|  \nabla w_0 \|^2_{L^2(\R^5)}  < \| \nabla Q \|_{L^2(\R^5)} .
\]
Let $w(t)$ be a solution of \eqref{H} with initial data $w_0$, then by continuity of the flow of the equation \eqref{H}, we have for all $t \in I$ (here, $I$ is the maximal time of existence)
\begin{equation}\label{contradiction}
    \| \nabla  w(t) \|^2_{L^2(\R^5)}  < \| \nabla Q \|^2_{L^2(\R^5)} .
\end{equation}
It follows that $I=\R$ and again utilizing the continuity of the flow along with \eqref{def-tau} and \eqref{compact-conv}, we have 
    \[
    u(t_n) \xrightarrow[n \longrightarrow + \infty]{} w(\tau) \in H^1(\R^5).
    \] 
Finally, invoking the first hypothesis in \eqref{contrad-inf-sup}, $\delta(t_n) \rightarrow 0$, i.e.,  $\| \nabla w(\tau) \|^2_{L^2(\R^5)} = \| \nabla Q  \|_{L^2(\R^5)}, $ which contradicts \eqref{contradiction}, completing the proof of \eqref{alpha-inf-sup}. 

The rest of the argument follows \cite[Lemma 6.8]{DR10} verbatim. 
\end{proof}
\subsubsection{Closing the curtain on Proposition \ref{main-sub}}
We now have all the ingredients to complete the proof of Proposition \ref{main-sub}. We begin with showing that $x(t)$ is bounded. By invoking Lemma \ref{control2}, there exists a sequence $t_n$ such that $t_{n+1} \geq t_n + 1$ for all $n$ and $\alpha(u(t_n)) \rightarrow 0$. Appealing to Lemmas \ref{control1} and \ref{control3}, there exists a constant $C>0$ such that if $n>N$ and $t\in [t_N +1, t_n]$, then
\[
\big|x(t)-x(t_N)\big| \leq C \Big[ 1+ \sup_{ s \in [t_N, t_n]} |x(s)| \Big] \big[ \alpha(u(t_N))+ \alpha(u(t_n)) \big].
\]
Fix $N$ large enough such that $\alpha(u(t_N))+ \alpha(u(t_n))\leq \frac{1}{2C}$ and choose $t\in [t_N+1,t_n]$ such that $\big|x(t)\big| = \sup\limits_{t_N+1\leq s\leq t_n}\big|x(s)\big|$, then     
\[
\sup\limits_{t_N+1\leq s\leq t_n}\big|x(s)\big|  \leq C(N) + \frac{1}{2} \sup\limits_{t_N+1\leq s\leq t_n}\big|x(s)\big|,
\]
where
\[
C(N) = \big|x(t_N)\big| + \frac{1}{2} \Big(1+ \sup\limits_{t_N\leq s\leq t_N+1}\big|x(s)\big|\Big).
\]
Taking $n\rightarrow \infty$, we see that $|x(t)|$ is bounded on $[t_N+1,+\infty)$, and hence by continuity, on $[0,+\infty)$. 

Now using the boundedness of $x(t)$ along with Lemma \ref{control1}, we have 
\[
\int_{\sigma}^{\tau} \alpha(u(t))\,dt  \leq  C \big[ \alpha(u(\sigma))+ \alpha(u(\tau)) \big].
\]  
We now fix $\sigma \geq 0$ and choose $\tau=t_n$ (note that $t_n$ is given by Lemma \ref{control2} such that $\alpha(u(t_n))\rightarrow 0$), letting $n\rightarrow \infty$, we get  
\[
\int_{\sigma}^{\infty} \alpha(u(t))\,dt  \leq C \alpha(u(\sigma)).
\] 
Applying Gronwall's inequality and then invoking Lemma \ref{mod-exp}, we conclude the proof of Proposition \ref{main-sub}.
\subsection{Scattering for the negative times}
Assume that $u(t)$ does not scatter for negative time. Then by applying the analogous arguments from Sections \ref{subs:compact} and \ref{subs:Prop6.1}, we have
\begin{enumerate}
\item there exists a continuous function $x(t)$ such that 
		$
		K \defeq \{u(\cdot+x(t),t);\,\,\,t\in\R\}
		$
	has a compact closure in $H^1(\R^5)$,
	\item there exists a constant $C_1>0$ such that  if $0 \leq \sigma \leq \tau$ 
\[
\int_{\sigma}^{\tau} \alpha(u(t))\,dt \leq C_1  \Big[ 1+ \sup_{ t \in [\sigma, \tau]} |x(t)| \Big] \big[ \alpha(u(\sigma))+ \alpha(u(\tau)) \big]   ,
\]
\item there exists a decreasing sequence $\tilde{t}_n\rightarrow -\infty$ such that  $\alpha(u(\tilde{t}_n)) \rightarrow 0$ as $n\rightarrow\infty$,
\item lastly, there exists a constant $C_2>0$ such that 
\[
\forall \sigma,\tau>0\quad \text{with} \quad \sigma+1 \leq \tau, \quad |x(\tau)-x(\sigma)| \leq C_2 \int_{\sigma}^{\tau} \alpha(u(t))\, dt  .
\]
\end{enumerate} 
We know that (2)-(4) implies that $ \alpha(u(t)) \rightarrow 0$ as $t\rightarrow \pm\infty$ and $x(t)$ is bounded for $t\in\R$. From (2), we then obtain 
\[
\int_{\sigma}^{\tau} \alpha(u(t))\,dt  \leq  C \big[ \alpha(u(\sigma))+ \alpha(u(\tau)) \big],
\]  
where $-\infty <\sigma \leq \tau<\infty$. Allowing $\sigma$ to approach $-\infty$ and $\tau$ to $\infty$, we get 
\[
\int_{\R}^{} \alpha(u(t))\,dt =0.
\]
Therefore, $\alpha(u(t)) = 0$ for all $t\in\R$, which contradicts the gradient assumption in \eqref{below}, i.e., $\|\nabla u_0\|_{L^2(\R^5)} < \|\nabla Q\|_{L^2(\R^5)}$.
\section{Uniqueness and classification result}\label{S:Uniq}
This section is devoted to establishing a uniqueness result for threshold solutions which will conclude the proof of Theorem \ref{mainthm2}. 
\subsection{Estimates on exponentially decaying solutions of the linearized equation} We start by studying the solutions of the (approximate) linearized equation
\begin{equation}\label{eq:expo1}
h_t + \mathcal{L}h = \varepsilon,\quad x\in\R^5,\ t\in (t_0,+\infty),
\end{equation}
such that the solutions decay exponentially as $t\rightarrow +\infty$
\begin{equation}\label{eq:expo2}
\|h\|_{H^1(\R^5)}\leq Ce^{-\gamma_1t},\quad \|\nabla\varepsilon\|_{L_{[t,+\infty)}^{\frac{8}{5}}L_x^{\frac{20}{13}}} + \|\varepsilon\|_{L_{[t,+\infty)}^{\frac{8}{3}}L_x^{\frac{20}{7}}}\leq Ce^{-\gamma_2t},
\end{equation}
where $0<\gamma_1 < \gamma_2$. For the next couple of lemmas, denote by $\gamma^{-}$ a positive number arbitrary close to $\gamma$ such that $0 < \gamma^{-} < \gamma$. As a consequence of Strichartz estimates together with Lemma \ref{R-V-est}, we have the following lemma.
\begin{lemma}\label{lem:expo0}
Let $h$ be a solution to \eqref{eq:expo1}. If for some positive constants $c$ and $C$
\begin{equation}\label{eq:expo0}
\|h\|_{H^1(\R^5)}\leq Ce^{-ct},\quad \forall t>0,
\end{equation}
then 
\begin{equation}\label{eq:expo2.5} 
\| h\|_{X([t,+\infty))} \leq Ce^{-ct}.
\end{equation}
\end{lemma}
\begin{proof}
We write \eqref{eqS} as $ih_t+\Delta h -h  + V(h) +R(h)=0,$
where $V(h)$ is defined in \eqref{eq:Vexp} and $R(h)$ is given by \eqref{Rexp-H}. By Duhamel's formula, Strichartz estimates and Lemma \ref{R-V-est}, we get
\[
\|h\|_{X([t,t+\tau))}\leq c\left( \||\nabla|^{\frac{1}{2}}h\|_{L^2_x} + \tau^{\frac{1}{2}}\|h\|_{X([t,t+\tau])} + \tau^{\frac{1}{8}}\|h\|^2_{X([t,t+\tau])} + \|h\|^3_{X([t,t+\tau])}\right).
\]
By \eqref{eq:expo0}, there exits a $\tau=\tau_0=\frac{1}{4C^2}$ small enough such that
\[
\|h\|_{X([t,t+\tau_0))}\leq C\left(e^{-ct} + \tau_0^{\frac{1}{2}}\|h\|_{X([t,t+\tau_0])} +  \|h\|^3_{X([t,t+\tau_0])}\right) < 2Ce^{-ct}.
\]
If not, then there exists $\tau \in (0, \tau_0]$ such that $\|h\|_{X([t,t+\tau])}=2Ce^{-ct}$ for a fixed $t>0$. Then
\[
2Ce^{-ct} \leq Ce^{-ct} + \tau^{\frac{1}{2}}2C^2e^{-ct} + 8C^4e^{-3ct}\leq 2Ce^{-ct}+ 8C^4e^{-3ct},
\]
which is a contradiction for large $t$. Therefore, for large $t$, we have
\[
\|h\|_{X([t,t+\tau_0))}\leq 2Ce^{-ct},\quad \tau_0=\frac{1}{4C^2}.
\]
Invoking \cite[Claim 5.8]{DM09}, yields \eqref{eq:expo2.5}.
\end{proof}

This brings us to the main ingredient for proving the uniqueness result. The following estimate is proved using the coercivity of the linearized energy, see \eqref{eq:linenergy} and Proposition \ref{positivity}.
\begin{lemma}\label{lem:expo1}
Under the assumptions \eqref{eq:expo2}, the following estimates hold.
\begin{enumerate}
\item[(i)] If $\gamma_2\leq e_0$ or $e_0 < \gamma_1$, then
\begin{equation}\label{eq:expo2a}
\|h\|_{H^1(\R^5)}\leq  Ce^{-\gamma_2^{-}t}.
\end{equation}
\item[(ii)] If $ \gamma_1 <e_0< \gamma_2$, then there exists $A\in\R$ such that
\begin{equation}\label{eq:expo2b}
 \|h - Ae^{-e_0t}\mathcal{Y}_+\|_{H^1(\R^5)}\leq Ce^{-\gamma_2^{-}t}.
\end{equation}
\end{enumerate}
\end{lemma}
\begin{proof}
We write the normalized eigenfunctions of $\mathcal{L}$ as
\[
Q_0 = \frac{iQ}{\|Q\|_{L^2}},\quad Q_j = \frac{\partial_{x_j}Q}{\|\partial_{x_j}Q\|_{L^2}},
\] 
and observe that by \eqref{eq:B-ortho}, we have
\[
B\big(Q_j,h\big) = 0, \quad \|Q_j\|_{L^2} = 1,\forall j = 0,1,\dots, 5,\ \forall h\in H^1. 
\]
We also recall that $B\big(\mathcal{Y}_+,\mathcal{Y}_-\big) \neq 0$. Normalize the eigenfunctions $\mathcal{Y}_+$, $\mathcal{Y}_-$ such that $B\big(\mathcal{Y}_+,\mathcal{Y}_-\big) =1$. We now decompose $h(t)$ as 
\begin{equation}\label{eq:expo3}
h(t) = \alpha_+(t)\mathcal{Y}_+ + \alpha_-(t)\mathcal{Y}_- + \sum_{j=0}^{5}\beta_j(t)Q_j + h_\perp, 
\end{equation} 
where $h_\perp \in G_{\perp}^y$ and from \eqref{eq:LE-values} together with the fact that $Q_j\in \text{ker}\mathcal{L}$, $j=0,1,\ldots,5$, we have
\begin{equation}\label{eq:expo4}
\alpha_+(t) = B\big(h(t),\mathcal{Y}_-\big),\quad \alpha_-(t) = B\big(h(t),\mathcal{Y}_+\big)
\end{equation}
\begin{equation}\label{eq:expo5}
\beta_j(t) = \big(h(t),Q_j\big) -\alpha_+(t) \big(\mathcal{Y}_+,Q_j\big) -\alpha_-(t) \big(\mathcal{Y}_-,Q_j\big),\ \ j=0,\ldots,5.
\end{equation}
We now proceed in several steps.\\

{\bf Step 1:} {\it Differential equations on the coefficients.} We show that
\begin{equation}\label{eq:expo6}
\frac{d}{dt}\left(e^{e_0t}\alpha_+(t)\right) = e^{e_0t}B\left(\varepsilon,\mathcal{Y}_-\right),\quad \frac{d}{dt}\left(e^{-e_0t}\alpha_-(t)\right) = e^{-e_0t}B\left(\varepsilon,\mathcal{Y}_+\right),
\end{equation}  
\begin{equation}\label{eq:expo7}
\frac{d}{dt}\left(\beta_j(t)\right) = \left(\varepsilon,Q_j\right) - \left(\mathcal{Y}_+,Q_j\right)B\left(\varepsilon,\mathcal{Y}_-\right) - \left(\mathcal{Y}_-,Q_j\right)B\left(\varepsilon,\mathcal{Y}_+\right) - \left(\mathcal{L}h_{\perp}, Q_j\right),
\end{equation}  
\begin{equation}\label{eq:expo8}
\frac{d}{dt}\left(\Phi(h(t))\right) = 2B\left(\varepsilon,h\right).
\end{equation}
By \eqref{eq:expo1} and \eqref{eq:expo4}, we have
\[
\begin{aligned}
\alpha^{\prime}_+(t) = B\left(h_t,\mathcal{Y}_-\right) &= B\left(\varepsilon-\mathcal{L}h,\mathcal{Y}_-\right) \\
&= B\left(\varepsilon,\mathcal{Y}_-\right) + B\left(h,\mathcal{L}\mathcal{Y}_-\right)=B\left(\varepsilon,\mathcal{Y}_-\right) - e_0\alpha_+(t),
\end{aligned}
\]
which yields the first differential equation for $\alpha_+(t)$ in \eqref{eq:expo6}. The second equation on $\alpha_-(t)$ in \eqref{eq:expo6} follows analogously. 

Differentiating \eqref{eq:expo5} using \eqref{eq:expo1}, we get
\[
\beta_j^{\prime}(t) = \left(\varepsilon - \mathcal{L}h -\alpha_+^{\prime}\mathcal{Y}_+ - \alpha_-^{\prime}\mathcal{Y}_-, Q_j\right).
\]
Using $\mathcal{L}h=e_0\alpha_+\mathcal{Y}_+ - e_0\alpha_-\mathcal{Y}_- + \mathcal{L}h_\perp$ yields
\begin{equation}\label{eq:expo7a}
\beta_j^{\prime}(t) = \left(\varepsilon,Q_j\right) - \left(e_0\alpha_++\alpha_+^{\prime}\right)\left(\mathcal{Y}_+,Q_j\right) - \left(e_0\alpha_-+\alpha_-^{\prime}\right)\left(\mathcal{Y}_-, Q_j\right) - \left(\mathcal{L}h_{\perp},Q_j\right),
\end{equation}
which together with \eqref{eq:expo6} gives \eqref{eq:expo7}. 

Lastly, we differentiate $\Phi(h(t))$ to obtain
\[
\frac{d}{dt}\left(\Phi(h(t))\right) = 2B\left(h_t,h\right) = -2B\left(\mathcal{L}h,h\right) + 2B\left(\varepsilon,h\right),
\] 
using \eqref{eq:antisym} we infer that $B\left(\mathcal{L}h,h\right)=0$, which in turn gives \eqref{eq:expo8}. 

{\bf Step 2:} {\it Estimates on $\alpha_{\pm}(t)$.} We now claim that 
\begin{equation}\label{eq:expo9}
\left|\alpha_-(t)\right| \leq Ce^{-\gamma_2t},
\end{equation}
\begin{equation}\label{eq:expo10a}
\left|\alpha_+(t)\right| \leq Ce^{-\gamma_2^{-}t},\ \ \text{if}\ \gamma_2< e_0\ \text{or}\ e_0 < \gamma_1,
\end{equation}
and there exist $A\in \R$ such that
\begin{equation}\label{eq:expo10b}
\left|\alpha_+(t) - Ae^{-e_0t}\right| \leq Ce^{-\gamma_2t}, \ \ \text{if}\  \gamma_1 < e_0 < \gamma_2.
\end{equation}
We use \eqref{eq:bilin} to write 
\[\begin{aligned}
2B\left(f,g\right) = &\int_{\R^5}\nabla f_1\nabla g_1\,dx +\int_{\R^5}f_1g_1\,dx +\int_{\R^5}\nabla f_2\nabla g_2\,dx +\int_{\R^5}f_2g_2\,dx\\
&-\int_{\R^5}\left(|x|^{-3}\ast Q^2\right)\left(f_1g_1+f_2g_2\right)\,dx-2\int_{\R^5}\left(|x|^{-3}\ast \left(Qf_1\right)\right)Qg_1\,dx 
\end{aligned}\]
and for any finite time interval $I$, we estimate
\begin{equation}\label{eq:expo11}
\int_I\left|B\left(f,g\right)\right|dt\lesssim \|\nabla f\|_{L_I^{\frac{8}{5}}L_x^{\frac{20}{13}}}\|\nabla g\|_{L_I^{\frac{8}{3}}L_x^{\frac{20}{7}}} +3|I|^{\frac{1}{4}}\|Q\|^2_{L_I^{\infty}L_x^{\frac{20}{7}}}\|f\|_{L_I^{\frac{8}{3}}L_x^{\frac{20}{7}}}\|g\|_{L_I^{\frac{8}{3}}L_x^{\frac{20}{7}}}.
\end{equation}
Therefore, using the assumption on $\varepsilon(t)$ in \eqref{eq:expo2} together with the \cite[Claim 5.8]{DM09}, we get
\[
\int_{t}^{+\infty}\left|e^{-e_0s}B\left(\varepsilon(s), \mathcal{Y}_-\right)\right|ds \leq Ce^{-(e_0+\gamma_2)t}.
\]
Integrating the equation on $\alpha_-$ in \eqref{eq:expo6} from $t$ to $+\infty$ and using the above estimate, we obtain \eqref{eq:expo9}.

To prove \eqref{eq:expo10a}, we first consider the case $e_0<\gamma_1$. Integrating the equation on $\alpha_+$ in \eqref{eq:expo6} from $t$ to $+\infty$ and using the assumption on $\varepsilon(t)$ in \eqref{eq:expo2} together with estimate \eqref{eq:expo11}, we obtain
\[
\left|e^{e_0t}\alpha_+(t)\right|\leq \int_{t}^{+\infty}\left|e^{e_0s}B\left(\varepsilon(s), \mathcal{Y}_+\right)\right|ds \leq Ce^{(e_0-\gamma_2)t},
\]
which implies \eqref{eq:expo10a} for $e_0 <\gamma_1$. Now if $e_0> \gamma_2$, integrating the $\alpha_+$ equation in \eqref{eq:expo6} between $0$ and $t$, we get
\[
e^{e_0t}\alpha_+(t) - \alpha_+(0) = \int_0^t e^{e_0s}B\left(\varepsilon(s), \mathcal{Y}_+\right)ds.
\]
By the assumption on $\varepsilon$ in \eqref{eq:expo2} together with \eqref{eq:expo11}, we have
\[
\left|e^{e_0t}\alpha_+(t)\right| \leq \alpha_+(0) +
Ce^{(e_0-\gamma_2)t},
\]
which yields \eqref{eq:expo10a} for $e_0>\gamma_2$. Lastly, we consider the case $\gamma_1<e_0<\gamma_2$. Again, using the assumption on $\varepsilon(t)$ in \eqref{eq:expo2} together with the \cite[Claim 5.8]{DM09}, we get
\[
\int_{t}^{+\infty}\left|e^{e_0s}B\left(\varepsilon(s), \mathcal{Y}_+\right)\right|ds \leq Ce^{(e_0-\gamma_2)t}<\infty.
\]
Combining this together with the $\alpha_+$ equation in \eqref{eq:expo6}, we get that 
\[
\lim\limits_{t\rightarrow\infty}e^{e_0t}\alpha_+(t) = A\quad\text{and}\quad \left|e^{e_0t}\alpha_+(t) - A\right|\leq Ce^{(e_0-\gamma_2)t},
\]
which implies \eqref{eq:expo10b}. 

{\bf Step 3:} {\it Conclusion when $e_0>\gamma_2$ or when $e_0<\gamma_2$ and $A=0$.} By Steps 1 and 2, we have
\begin{equation}\label{eq:expo12a}
\left|\alpha_+(t)\right| + \left|\alpha_+^{\prime}(t)\right| \leq Ce^{-\gamma_2^{-}t},
\end{equation}
and
\begin{equation}\label{eq:expo12b}
\left|\alpha_-(t)\right| + \left|\alpha_-^{\prime}(t)\right| \leq Ce^{-\gamma_2t}.
\end{equation}
Invoking the decomposition \eqref{eq:expo3} together with \eqref{eq:LE-values} and the fact that $h_\perp\in G_{\perp}^y$, we get
\[
\Phi(h(t)) = B(h(t),h(t)) = 2\alpha_+(t)\alpha_-(t) + B\left(h_\perp(t),h_\perp(t)\right).  
\]
Again by integrating \eqref{eq:expo8} and using the assumption \eqref{eq:expo2} together with \cite[Claim 5.8]{DM09} and estimate \eqref{eq:expo11}, we obtain
\[
\left|\Phi(h(t))\right| \leq Ce^{-(\gamma_1+\gamma_2)t}.
\]
This together with \eqref{eq:expo12a} and \eqref{eq:expo12b} gives
\[
\left|B\left(h_\perp,h_\perp\right)\right| \leq C\left(e^{-\gamma_2^{-}t} + e^{-(\gamma_1+\gamma_2)t}\right) \leq Ce^{-(\gamma_1+\gamma_2)t}.
\]
By Proposition \ref{positivity}, we obtain
\begin{equation}\label{eq:expo12c}
\|h_\perp\|^2_{H^1(\R^5)} \lesssim \left|B\left(h_\perp,h_\perp\right)\right| \leq Ce^{-(\gamma_1+\gamma_2)t}.
\end{equation}
Now we deduce the bound on $\beta_j$. Consider \eqref{eq:expo7a} and observe that the assumption on $\varepsilon$ in \eqref{eq:expo2} gives that $\big|(\varepsilon, Q_j)\big| \leq Ce^{-\gamma_2t}$. Furthermore, by \eqref{eq:expo12a} and \eqref{eq:expo12b}, we have 
\[
 \big|- \left(e_0\alpha_++\alpha_+^{\prime}\right)\left(\mathcal{Y}_+,Q_j\right) - \left(e_0\alpha_-+\alpha_-^{\prime}\right)\left(\mathcal{Y}_-, Q_j\right)\big| \leq Ce^{-\gamma_2t}.
 \]
We are left to deal with the last term in \eqref{eq:expo7a}, i.e., $(\mathcal{L}h_\perp,Q_j)$. We observe that $(\mathcal{L}h_\perp,Q_j) = (h_\perp,\mathcal{L}^*Q_j)$, where
$\displaystyle \mathcal{L}^* =  \begin{pmatrix}
	0 & L_+\\
	-L_+ & 0
	\end{pmatrix}
$ is the $L^2$-adjoint of $\mathcal{L}$. By the decay properties of $Q$ and its derivatives, we have that $\mathcal{L}^*Q_j\in L^2(\R^5)$ and therefore, 
\[
\big|(\mathcal{L}h_\perp,Q_j)\big|\leq C\|h_\perp\|_{L^2(\R^5)}\leq C\|h_\perp\|_{H^1(\R^5)} \leq Ce^{-\frac{\gamma_1+\gamma_2}{2}t}.
\]
Putting this together and integrating \eqref{eq:expo7a} (or \eqref{eq:expo7}) between $t$ and $+\infty$, we get
\begin{equation}\label{eq:expo12d}
\big|\beta_j(t)\big| \leq Ce^{-\frac{\gamma_1+\gamma_2}{2}t}.
\end{equation}
Hence, collecting \eqref{eq:expo12a}, \eqref{eq:expo12b}, \eqref{eq:expo12d} and \eqref{eq:expo12d} with the decomposition \eqref{eq:expo3}, we see that $h$ and $\varepsilon$ satisfy the assumptions \eqref{eq:expo2}, but with $\frac{\gamma_1+\gamma_2}{2}$, a different number than $\gamma_1$. By iterating the argument, we obtain 
\[
\|h(t)\|_{H^1(\R^5)}\leq Ce^{-\gamma_2^{-}t},
\]
concluding the proof when $e_0>\gamma_2$ or when $e_0<\gamma_2$ and $A=0$.

{\bf Step 4:} {\it Conclusion when $\gamma_1<e_0 < \gamma_2$.} Observe that since $\mathcal{L}\mathcal{Y}_+ = e_0\mathcal{Y}_+$, we have that $\tilde{h}(t) = h(t)-Ae^{-e_0t}\mathcal{Y}_+$ satisfies the equation \eqref{eq:expo1} with the same $\varepsilon$ and the assumptions \eqref{eq:expo2}. From \eqref{eq:expo10b}, we have $\lim\limits_{t\rightarrow\infty} e^{e_0t}\tilde{\alpha}_+(t)$, where $\tilde{\alpha}_+$ is the coefficient of $\mathcal{Y}_+$ in the decomposition of $\tilde{h}(t)$. Therefore, by Steps 1 and 2, $\tilde{h}$ and $\varepsilon$ satisfy the assumptions of Step 3, i.e., $\tilde{h}$ and $\varepsilon$ satisfy the assumptions \eqref{eq:expo2}, but with $\gamma_1$ replaced with $\frac{\gamma_1+\gamma_2}{2}$. By iterating the argument, we get \eqref{eq:expo2a} for $\tilde{h}$, which in turn implies the condition \eqref{eq:expo2b} for the original $h$, thereby completing the proof.
\end{proof}


 \subsection{Uniqueness}
 We are now able to show the following uniqueness result.
\begin{proposition}\label{prop:unique}
Let $u$ be a solution of \eqref{H}, defined on $[t_0,+\infty)$ such that $E[u]=E[Q]$, $M[u]=M[Q]$ and there exist constants $c,C>0$ with
\begin{equation}\label{eq:uni1}
\|u-e^{it}Q\|_{H^1(\R^5)} \leq Ce^{-ct},
\end{equation}
for all $t\geq t_0$. Then there exists $A\in\R$ such that $u=U^A$, where $U^A$ is the solution of \eqref{H} defined in Proposition \ref{sp-solns}.
\end{proposition}
\begin{proof} Let $u=e^{it}(Q+h)$ be a (radial) solution to \eqref{H} satisfying the above hypothesis.
{\bf Step 1:} We start by showing that for all $t\geq t_0$,
\begin{equation}\label{eq:uni2}
\|h(t)\|_{H^1(\R^5)} \leq Ce^{-e_0^-t}. 
\end{equation}
 Observe that $h$ satisfies the linearized equation \eqref{lineq}, i.e., $h_t + \mathcal{L}h = R(h)$, where $R(h)$ is given by \eqref{Rexp-H}. By the estimate \eqref{R-est} in Lemma \ref{R-V-est} and Lemma \ref{lem:expo0} together with \eqref{eq:uni1} and \cite[Claim 5.8]{DM09}, we get 
\[
\|\nabla R(h)\|_{L_{[t,+\infty)}^{\frac{8}{5}}L_x^{\frac{20}{13}}} + \|R(h)\|_{L_{[t,+\infty)}^{\frac{8}{3}}L_x^{\frac{20}{7}}}\leq Ce^{-2ct}.
\]
We are therefore now in the setting of Lemma \ref{lem:expo1} with $g=R(h)$, $\gamma_1 = c$ and $\gamma_2 = 2c$. If $c < e_0 < 2c$, \eqref{eq:expo2b} gives
\[
\|h(t)\|_{H^1(\R^5)} \leq C\big(e^{-e_0t} + e^{-2c^-t}\big).
\] 
Now if $e_0 \leq 2c^-$, the above estimate directly implies \eqref{eq:uni2}. On the other hand, if $2c^-<e_0$, we get $\|h(t)\|_{H^1(\R^5)} \leq Ce^{-2c^-t}$, and an iteration argument again yields \eqref{eq:uni2}. And, in the case $e_0\notin (\gamma_1,\gamma_2]$, we again get  $\|h(t)\|_{H^1(\R^5)} \leq Ce^{-2c^-t}$, which as before yields \eqref{eq:uni2}.

{\bf Step 2:} We now consider the special solutions $U^A = e^{it}(Q+\mathcal{V}_k^A+h)$ constructed in Proposition \ref{sp-solns}. We will show that there exists $A\in \R$ such that for all $\gamma>0$, there exists a $C>0$ and for all $t\geq t_0$, we have 
\begin{equation}\label{eq:uni3}	
 \|\mathcal{V}_k^A\|_{H^1} + \|\mathcal{V}_k^A\|_{X([t,+\infty))} \leq Ce^{-\gamma t}.
 \end{equation}
 By Step 1, $h(t)$ satisfies the assumptions of Lemma \ref{lem:expo1} with $\gamma_1 = e_0^-$, $\gamma_2 = 2e_0^-$. Thus, there exists $A \in \R$ such that
 \begin{equation}\label{eq:uni4}
 \|h - Ae^{-e_0t}\mathcal{Y}_+\|_{H^1(\R^5)} \leq Ce^{-2e_0^-t}.
 \end{equation}
 On the other hand, by \eqref{sp-solns2} in the Proposition \ref{sp-solns},  
 \[
 \|h + \mathcal{V}_k^A- Ae^{-e_0t}\mathcal{Y}_+\|_{H^1(\R^5)} \leq Ce^{-2e_0t}.
 \] 
 This together with \eqref{eq:uni4} and Lemma \ref{lem:expo0} gives \eqref{eq:uni3} for $\gamma<2e_0^-$. Now we assume that \eqref{eq:uni3} holds for some $\gamma =\gamma_3 > e_0$, and show that it then also holds for $\gamma_3 + \frac{e_0}{2}$. Observe that $\mathcal{V}_k^A$ satisfies the equation
 \[
 \partial_t\mathcal{V}_k^A + \mathcal{L}\mathcal{V}_k^A = R(h)-R(h+\mathcal{V}_k^A).
 \]
 By Lemma \ref{R-V-est}, \eqref{R-est} to be precise, together with \cite[Claim 5.8]{DM09}, we get
 \[
\| R(h)-R(h+\mathcal{V}_k^A)\|_{L_{[t,+\infty)}^{\frac{8}{5}}L_x^{\frac{20}{13}}} + \| R(h)-R(h+\mathcal{V}_k^A)\|_{L_{[t,+\infty)}^{\frac{8}{3}}L_x^{\frac{20}{7}}} \leq Ce^{-(e_0+\gamma_3)t}.
 \] 
We therefore infer that $\mathcal{V}_k^A$ satisfies the assumptions of Lemma \ref{lem:expo1} with $\gamma_1 = \gamma_3$ and $\gamma_2 = \gamma_3+e_0$ with $\gamma_2^- = \gamma_3 + \frac{e_0}{2}$. Invoking Lemma \ref{lem:expo0} yields \eqref{eq:uni3}. This completes Step 2.

{\bf Step 3:} We now conclude the proof by using \eqref{eq:uni3}with $\gamma = (k_0 + 1)e_0$, where $k_0$ is given by
Proposition \ref{sp-solns}, we get that for large $t > 0$,
\[
\|u-U^A\|_{X([t,+\infty))} = \|\mathcal{V}_{k_0}^A\|_{X([t,+\infty))} \leq Ce^{-(k_0+\frac{1}{2})e_0t}.
\]
Hence, by uniqueness in Proposition \ref{sp-solns}, we get as desired that $u = U^A$.
 \end{proof}
 

 \subsection{Proof of Theorem \ref{mainthm2}} We first show that if $A\neq 0$, $U^A$ is either equal to $Q^+$ (if $A>0$) or $Q^-$ (if $A<0$) up to a translation in time and a multiplication by a complex number of modulus $1$. Fix $A>0$ and choose $T_A$ such that $|A|e^{-e_0T_A} = 1$. By \eqref{sp-solns2} in Proposition \ref{sp-solns}, we have 
\begin{equation}\label{eq:last1}
\|e^{-iT_A}U^A(t+T_A) - e^{it}Q - e^{(i-e_0)t}\mathcal{Y}_+\|_{H^1(\R^5)}\leq Ce^{-2e_0(t+T_A)} \leq Ce^{-2e_0t}.
\end{equation}
We observe that $U^A(t+T_A)$ satisfies the hypotheses of Proposition \ref{prop:unique}, and thus, there exists $\tilde{A}$ such that $e^{-iT_A}U^A(t+T_A) = U^{\tilde{A}}$. In addition, by \eqref{eq:last1}, we know that $\tilde{A}=1$. Now letting $Q^+(x,t) = e^{-it_0}U^{+1}(x,t+t_0)$, we obtain $U^A=e^{iT_A}Q^+(t-T_A)$. A similar argument for $A<0$ gives $U^A= e^{iT_A}Q^-(t-T_A)$.

We are now equipped to present the proof of Theorem \ref{mainthm2}. Let $u$ be a (radial) solution of \eqref{H} satisfying $M[u]E[u]=M[Q]E[Q],$ rescaling $u$ we may assume $M[u]=M[Q]$ and $E[u]=E[Q]$.

The case (a), i.e., if $\|\nabla u_0\|_{L^2} > \|\nabla Q\|_{L^2}$. By assumption, $u$ does not blow-up for both finite positive and negative times. Replacing $u(x,t)$ by $u(x,-t)$ if necessary, we may assume that $u$ does not blow-up for positive times. Then by Proposition \ref{blow-up}, there exist $x_0\in\R^5$, $\gamma_0\in\R$, and $c,C>0$ such that  
\[
\|u-e^{it+i\gamma_0}Q(\cdot+x_0)\|_{H^1(\R^5)}\leq Ce^{-ct}.
\]
Hence, $e^{-i\gamma_0}u(x+x_0,t)$ satisfies the assumptions of Proposition \ref{prop:unique}, which shows that  $e^{-i\gamma_0}u(x+x_0,t)= U^A$ for some $A$. Since $\|\nabla u_0\|_{L^2} > \|\nabla Q\|_{L^2}$, the parameter $A$ must be positive, and hence $u=Q^+$ up to the symmetries of the equation.

The case (b), i.e., $\|\nabla u_0\|_{L^2} = \|\nabla Q\|_{L^2}$, follows from the variational characterization of $Q$.

The case (c), i.e., $\|\nabla u_0\|_{L^2} < \|\nabla Q\|_{L^2}$ is similar to the case (a). We know (by assumption) $u$ does not scatter in both time directions. Replacing $u(x,t)$ by $u(x,-t)$ if necessary, we may assume that $u$ does not scatter for positive times. By Proposition \ref{main-sub}, there exist $x_0\in\R^5$, $\gamma_0\in\R$, and $c,C>0$ such that  
\[
\|u-e^{it+i\gamma_0}Q(\cdot+x_0)\|_{H^1(\R^5)}\leq Ce^{-ct}.
\]
Invoking a similar argument as in the case (a), we obtain that for some parameter $A<0$ 
\[
e^{-i\gamma_0}u(x+x_0,t)= U^A = e^{-tT_A}Q^-(t+T_A).
\]
Hence, $u=Q^-$ up to the symmetries of the equation. This finishes the proof of Theorem \ref{mainthm2}.


\bibliography{Andy-references}
\bibliographystyle{abbrv}
\end{document}